\documentclass[11pt, reqno]{amsart}
\usepackage{amssymb}
\usepackage{amsmath}
\usepackage{amscd}
\usepackage{mathtools} % extends amsmath
\usepackage{amsfonts}
\usepackage{enumitem}
\usepackage{mathrsfs}
\usepackage{stmaryrd}
\usepackage{tikz}
\usepackage[
  colorlinks=true,   % false = boxed links; true = colored links
  linkcolor=blue,    % color of internal links
  urlcolor=blue,     % color of URLs
  citecolor=blue     % color of citations
]{hyperref}
\usepackage{mathtools}
\usepackage{cancel}
\usepackage[normalem]{ulem}
\usepackage{bbm}
\usepackage{scalerel,stackengine}
\usepackage{xparse}
\usepackage[letterpaper, margin=3cm]{geometry}

\theoremstyle{definition}
\newtheorem{thm}{Theorem}[section]
\newtheorem{defn}[thm]{Definition}
\newtheorem{prop}[thm]{Proposition}
\newtheorem{cor}[thm]{Corollary}
\newtheorem{lemma}[thm]{Lemma}
\newtheorem{rmk}[thm]{Remark}

\newtheorem{nota}[thm]{Notation}

\newcommand{\p}{\partial}
\newcommand{\wt}{\mbox{\rm wt}\,}

\newcommand{\lbar}{\bigg\vert}

\newcommand{\C}{\mathbb{C}}
\newcommand{\Z}{\mathbb{Z}}

\newcommand{\N}{\mathbb{N}}

\newcommand{\F}{\mathcal{F}}

\newcommand{\one}{\mathbf{1}}

\newcommand{\B}{\mathcal{B}}
\newcommand{\h}{\mathfrak{h}}

\newcommand{\perm}{\operatorname{perm} }
\newcommand{\mat}{\operatorname{Mat}}
\newcommand{\Log}{\text{Log }}

\newcommand{\nn}{\nonumber \\}
\newcommand{\nord}{\typecolon}

\renewcommand{\hom}{\mbox{\rm Hom}}
\renewcommand{\d}{\mathbf{d}}
\renewcommand{\k}{\mathbf{k}}

\stackMath
\allowdisplaybreaks

\newcommand\reallywidehat[1]{%
\savestack{\tmpbox}{\stretchto{%
  \scaleto{%
    \scalerel*[\widthof{\ensuremath{#1}}]{\kern-.6pt\bigwedge\kern-.6pt}%
    {\rule[-\textheight/2]{1ex}{\textheight}}%WIDTH-LIMITED BIG WEDGE
  }{\textheight}% 
}{0.5ex}}%
\stackon[1pt]{#1}{\tmpbox}%
}
\newcommand{\hcancel}[2][red]{%
  \setbox0=\hbox{$\displaystyle #2$}%
  \rlap{\raisebox{0.6ex}{\textcolor{#1}{\rule{\wd0}{0.6pt}}}}#2%
}

\DeclareFontEncoding{LS1}{}{}
\DeclareFontSubstitution{LS1}{stix}{m}{n}
\DeclareSymbolFont{symbols2}{LS1}{stixfrak}{m}{n}
\DeclareMathSymbol{\typecolon}{\mathbin}{symbols2}{"25}

\title[Boson-Fermion MOSVAs and Twisted Modules]{\textbf{Boson-Fermion
Meromorphic Open-string Vertex algebras and Their canonically twisted modules}}
\date{\today}
\author{Qixuan Fang, Yi-Zhi Huang, Fei Qi}
\begin{document}

\setlength{\oddsidemargin}{0cm} \setlength{\evensidemargin}{0cm}
\baselineskip=18pt

\begin{abstract} 
We construct a $\frac {\Z} 2$-graded meromorphic open-string vertex algebra from a 
finite-dimensional vector space with a nondegenerate symmetric bilinear form, together with its canonically twisted module. 
This algebra is generated by suitable noncommutative generalizations 
of bosonic and fermionic fields and is a noncommutative  generalization of the free
boson-fermion vertex operator superalgebra. Similarly to the purely bosonic and purely  fermionic cases 
in the early works by the second author in \cite{H-mosva}, by Fiordalisi and the third author
in \cite{FQ-Fermion-1}, 
and by the third author in \cite{Q-Fermion-2}, 
the usual super-commutatitive relation between creation and annihilation operators still 
hold  while no relations exist among creation operators. In particular,  normal-ordering 
remains well-defined. As in \cite{H-mosva}, \cite{FQ-Fermion-1}, and \cite{Q-Fermion-2}, 
we prove a generalized  Wick's theorem in this case, which gives a formula for 
a product of two normal-ordered products of bosonic and fermionic generating 
fields. Using this generalized Wick's theorem,
we construct the $\frac {\Z} 2$-graded meromorphic open-string vertex algebra and 
its canonically twisted module in this case. The construction in this paper 
 is the algebraic part of our construction  of suitable Dirac-like operators from spin manifolds. 
\end{abstract}

\maketitle
\tableofcontents

\section{Introduction}

Meromorphic open-string vertex algebras (MOSVAs hereafter) were introduced by the second author 
in \cite{H-mosva} as a natural noncommutative generalization of vertex operator algebras. 
Unlike ordinary vertex algebras, MOSVAs do not require the commutator 
formula, skew-symmetry, or the full Jacobi identity to hold. Note that chiral parts of 
conformal field theories always have suitable commutativity properties. 
The noncommutativity of MOSVAs means that they might correspond to some substructures of suitable
nonconformal field theories.

The initial construction in \cite{H-mosva} gives a $\mathbb{Z}$-graded MOSVA 
from a finite-dimensional inner product space (a Euclidean space)
and modules for this 
MOSVA. This algebra and its modules are noncommutative 
generalizations of the Heisenberg vertex operator algebra for free bosons
from this  inner product space and their modules.
The construction in \cite{H-mosva} in fact works for a finite-dimensional complex vector space over
with a nondegenerate symmetric bilinear form. 
Jointly with 
Fiordalisi in \cite{FQ-Fermion-1}, the third author constructed a $\frac{\Z}2$-graded 
MOSVA from an even-dimensional inner product space (the construction also works 
for a complex vector space over with a nondegenerate symmetric bilinear form).  The third author also 
constructed its canonically twisted module in \cite{Q-Fermion-2}. These constructions
are in fact non-super-commutative generalizations of the constructions of the free fermion vertex 
operator superalgebra  (also known as the Clifford vertex operator algebras) and its canonically twisted modules given in \cite{FFR} 
(see also \cite{Barron-Review}). 

The introduction and construction of MOSVAs and their modules in \cite{H-mosva} was motivated 
by the study in physics of  nonlinear sigma models with Riemannian manifolds as targets. 
The Heisenberg vertex operator algebras, Clifford vertex operator algebras and their 
modules  are in fact substructures of 
 (linear) sigma models with Euclidean spaces as targets. But the construction
of these vertex operator algebras and, especially, their modules cannot be generalized to the case where the targets
are non-flat Riemannian manifolds. On the other hand, in \cite{H-mosva-rm}, using the 
MOSVAs and their modules constructed in \cite{H-mosva}, the second author successfully constructed 
MOSVAs and modules generated by eigenfunctions on Riemannian manifolds. 
It is also shown in \cite{H-mosva-rm} that the zero mode of a special vertex operator obtained from the Riemannian metric acts on smooth functions on the Riemannian manifold 
as the Laplacian. From this construction,
it is clear that natural generalizations of Heisenberg vertex operator algebras and their modules 
are suitable quotients of the  MOSVAs and modules in \cite{H-mosva-rm}. 

In physics, supersymmetric  nonlinear sigma models with Riemannian manifolds as targets play an
important role. Many mathematical conjectures from string theory are in fact obtained based on 
such supersymmetric nonlinear sigma models, which still do not have a mathematical construction.
To develop a mathematical theory of such supersymmetric nonlinear sigma models,
we need to generalize the construction of the  MOSVAs and modules in \cite{H-mosva-rm}
for Riemannian manifolds to a construction of MOSVAs and canonically twisted modules
generated by spinor fields (fermions). We also expect that if these structures are constructed, 
we will obtain a suitable Dirac-like operator as a vertex operator such that its suitable component
acts as the Dirac operator on spinor fields. 

The first step of this generalization is 
to construct MOSVAs and their canonically twisted modules involving both bosons and fermions
on a finite-dimensional complex 
vector space with a nondegenerate symmetric bilinear form. In this paper, we carry out this first step. In a paper in preparation, we shall 
give a construction of sheaves of MOSVAs and their canonically twisted modules on spin manifolds and a Dirac-like operator acting on modules generated 
from spinor fields. 

More precisely, we construct in this paper a $\frac{\Z} 2$-graded MOSVA $V$ associated with a finite-dimensional complex 
vector space $\h$ with a nondegenerate symmetric bilinear form, whose dimension needs not to be even. The algebra is 
generated by fields from two copies of $\h$, one copy $\h^\B$ corresponding to bosons and 
another copy $\h^\F$ corresponding to fermions. Instead of the symmetric algebra of the 
affinization $\widehat{\h}^\B_\Z$, the exterior algebra of the twisted affinization $\widehat{\h}^\F_{\Z+1/2}$, 
and a mere tensor product of these vertex operator algebras, we use the tensor algebra
$T(\widehat{\h}^\B_\Z \oplus \widehat{\h}^\F_{\Z+1/2})$ and take  the quotient by an ideal
$N(\widehat{\h}^\B_{\Z}\oplus \widehat{\h}^\F_{\Z+1/2})$. We then take $V$ to be the $T(\widehat{\h}^\B_\Z \oplus \widehat{\h}^\F_{\Z+1/2})$-module
induced from the $N(\widehat{\h}^\B_{\Z}\oplus \widehat{\h}^\F_{\Z+1/2})$-module given by a one-dimensional 
space, and show that it admits a MOSVA structure. We want to emphasize again that 
$V$ is \textit{not} the tensor product of the MOSVAs constructed in \cite{H-mosva} and \cite{FQ-Fermion-1}.

The MOSVA $V$ admits a natural parity grading by defining
 the bosonic generating fields to be even, and the fermionic generating fields to be odd. 
The parity grading induces a parity involution $\theta$ on the MOSVA $V$. For any 
$T(\h^\B\oplus \h^\F)$-module $M$, following a similar procedure as in 
\cite{FLM}, \cite{FFR}, and \cite{Q-Fermion-2}, we construct a $\Z$-graded $\theta$-twisted 
$V$-module $W$ generated by $M$, called the ``canonically twisted modules.''
As in the case of vertex operator algebras, the twisted vertex operators on the canonically twisted modules
are series with half-integer powers. In particular, the correlation functions are suitable  algebraic functions. 

The main technical part of proving the axioms for the MOSVA structure on $V$ and for the module structure on $W$ 
is the proof of the weak associativity. As in \cite{H-mosva},  \cite{FQ-Fermion-1}, and \cite{Q-Fermion-2} , 
the main difficult step is the proof of a formula expresses the product of two normal-ordered 
products as a linear combination of normal-ordered products. Such
formulas  proved in  \cite{H-mosva},  \cite{FQ-Fermion-1},  \cite{Q-Fermion-2}, and this paper, 
are in fact generalizations of a suitable version of Wick's theorem \cite{W}
to the noncommutative settings. As in the case of 
Wick's theorem, the coefficients appearing in the linear combination are 
products of permanents and determinants, representing contractions of bosonic 
and fermionic fields, respectively. For the twisted module, we modify the 
$\exp(\Delta(x))$-correction described in \cite{FLM}, \cite{FFR}, and \cite{Q-Fermion-2}. 
We extend this correction in  \cite{FFR} and \cite{Q-Fermion-2} to the case when $\dim\h$ is odd. It should also be pointed 
out that the total contraction number defined in \cite{Q-Fermion-2} are in fact Pfaffians 
of certain skew-symmetric matrices. The proofs involve careful combinatorial manipulations 
with 2-shuffles and 3-shuffles, extending techniques from \cite{FQ-Fermion-1} and \cite{Q-Fermion-2}.

The paper is organized as follows. In Section 2, we review the definitions of $\frac{\Z} 2$-graded 
MOSVA and its twisted modules. This section also contains a discussion on the relation between the generalized rationality 
and the weak associativity with pole-order condition. In Section 3, we give the construction of the 
boson-fermion MOSVA $V$ but leave the proof of Theorem \ref{It-is-a-mosva} to Section 6. 
In Section 4, we give  the construction of the canonically twisted modules but leave the proof of 
Theorem \ref{it-is-a-module} to Section 7. In Section 5, we introduce some notations and prove 
some technical lemmas to prepare for the proofs of the theorems
metioned above, including notations for indices, shuffles, and lemmas for changing orders of summations. 
Theorem \ref{It-is-a-mosva} is proved in Section 6, with subsections proving a 
 generalized Wick theorem and the weak 
associativity. Theorem \ref{it-is-a-module} is proved in Section 7, by proving a generalized Wick's theorem in this case, 
analyzing products and iterates of the naive 
vertex operator, proving the properties of $\exp(\Delta(x))$, and finally establishing the weak associativity 
for the correct twisted vertex operator.

\noindent\textbf{Acknowledgements:} The third author is supported by Sun Yat-Sen University Research Start-up Grant 74120-12266016. 

\section{$\frac{\Z}{2}$-graded meromorphic open-string vertex algebras 
and $\theta$-twisted modules}

% In \cite{H-mosva}, a meromorphic open-string vertex algebra
% is defined to be a $\Z$-graded vector space $V$ equipped with a vertex operator map and 
% a vacuum satisfying suitable axioms. This definition
% can be generalized without change any other things to the case that 
% the underlying space is $\frac{\Z}{2}$-graded. Here we give the definition of 
% $\frac{\Z}{2}$-graded meromorphic 
% open-string vertex algebra. 

\begin{defn}\label{Def-mosva}
A \textit{$\frac{\Z}{2}$-graded meromorphic 
open-string vertex algebra} (MOSVA hereafter) is a $\frac{\Z}{2}$-graded 
vector space $V=\coprod_{n \in \frac{\Z}{2}} V_{(n)} $ 
(graded by \textit{weights}) equipped with a \textit{vertex operator map}
\begin{align*}
Y: V \otimes V &\, \rightarrow\, V[[x, x^{-1}]]\\
u \otimes v &\, \mapsto\, Y(u, x) v
\end{align*}
and a \textit{vacuum} $\one \in V$, satisfying the following conditions:

\begin{enumerate}

\item Axioms for the grading: 
\begin{enumerate}

\item \textit{Lower bound condition}: When $n$ is sufficiently 
negative, $V_{(n)}=0$.

\item \textit{$\mathbf{d}$-commutator formula}: 
Let $\d: V \rightarrow V$ be defined by 
$\d v=n v$ for $v \in V_{(n)} $. Then, for every $v \in V$
$$
\left[\d, Y(v, x)\right]=x \frac{d}{d x} Y(v, x)+Y\left(\d v, x\right)
$$

    \end{enumerate}
    
\item  Axioms for the vacuum:

\begin{enumerate}

    \item \textit{Identity property}: $Y(\mathbf{1}, x)=1_V$, where $1_V$ is the identity operator on $V$.
    
    \item  \textit{Creation property}: For $u \in V, Y(u, x) \mathbf{1} \in V[[x]]$ and $\lim _{x \rightarrow 0} Y(u, x) \mathbf{1}=$ u.
    
\end{enumerate}

\item \textit{$D$-derivative property} and \textit{$D$-commutator formula}: 
Define $D_V$ : $V \rightarrow V$ by
$$
D_V v=\lim _{x \rightarrow 0} \frac{d}{d x} Y(v, x) \mathbf{1}
$$
for $v \in V$. Then for $v \in V$,
$$
\frac{d}{d x} Y(v, x)=Y\left(D_V v, x\right)=\left[D_V, Y(v, x)\right].
$$

\item {\it Rationality}: For $u_{1}, \dots, u_{n}, v\in V$
and $v'\in V'$, the series 
\begin{eqnarray}\label{product}
\lefteqn{\langle v', Y(u_{1}, z_1)\cdots Y(u_{n}, z_n)v\rangle}\nn
&&=\langle v', Y(u_{1}, x_1)\cdots Y(u_{n}, x_n)v\rangle\lbar_{x_{1}=z_{1},
\dots, x_{n}=z_{n}}
\end{eqnarray}
converges absolutely 
when $|z_1|>\cdots >|z_n|>0$ to a rational function in $z_{1}, \dots, z_{n}$
with the only possible poles at $z_{i}=0$ for $i=1, \dots, n$ and $z_{i}=z_{j}$ 
for $i\ne j$. For $u_{1}, u_{2}, v\in V$
and $v'\in V'$, the series 
\begin{eqnarray}\label{iterate}
\lefteqn{\langle v', Y(Y(u_{1}, z_1-z_{2})u_{2}, z_2)v\rangle}\nn
&&=\langle v', Y(Y(u_{1}, x_{0})u_{2}, x_2)v\rangle
\lbar_{x_{0}=z_{1}-z_{2},\;
x_{2}=z_{2}}
\end{eqnarray}
converges absolutely when $|z_{2}|>|z_{1}-z_{2}|>0$ to a rational function
with the only possible poles at $z_{1}=0$, $z_{2}=0$ and $z_{1}=z_{2}$. 

\item {\it Associativity}: For $u_{1}, u_{2}, v\in V$, 
$v'\in V'$, we have
\begin{equation}\label{associativity-cplx}
\langle v', 
Y(u_{1},z_1)Y(u_{2},z_2)v\rangle
=
\langle v', 
Y(Y(u_{1},z_{1}-z_{2})u_{2},z_2)v\rangle
\end{equation}
when  $|z_{1}|>|z_{2}|>|z_{1}-z_{2}|>0$. 
\end{enumerate}
If $\dim V_{(n)}<\infty$ for every $n\in \frac{\Z}{2}$, we say that $V$ is  
{\it grading-restricted}.
\end{defn}

% Let
% $$V^{[0]} = \coprod\limits_{n\in \Z} V_{(n)}, V^{[1]}
% = \coprod\limits_{n\in \Z+\frac 1 2} V_{(n)}.$$

\begin{rmk}
The notion of MOSVA is a natural noncommutative generalization of vertex superalgebra. Recall that if $V$ is a vertex superalgebra, then there exists an additional $\Z_2$-grading $V = V^0 \oplus V^1$,  called \textit{parity grading}. The subspace $V^0$ consists of \textit{even} elements, $V^1$ \textit{odd} elements. For a $\Z_2$-homogeneous element $u$, let $|u| = (-1)^i$ if $u\in V^i, i=0,1$. Then the vertex operator of two $\Z_2$-homogeneous element $u$ and $v$ satisfies the super skew-symmetry 
$$Y(u, x)v = (-1)^{|u|\cdot |v|} e^{xD}Y(v,-x)u. $$
Such relations certainly do not hold if $V$ is a MOSVA. 
\end{rmk}

\begin{defn}\label{defn-of-theta}
Let
$$V^0 = \coprod_{n\in \Z} V_{(n)}, V^1 = \coprod_{n\in \Z} V_{(n+1/2)}, $$
and call them the \textit{even part}, \textit{odd part} of $V$, respectively. We similarly define the \textit{parity} $|u| = (-1)^i$ if $u\in V^i, i=0, 1$. Let $\theta: V \to V$ be a linear isomorphism defined by
$$\theta(v)=e^{i\pi}v = (-1)^i v,\ v\in V^{i},\ i=0, 1.$$
Clearly, $\theta$ is an involution of the MOSVA $V$, i.e., $\theta^2(v) = v$, and for all $u, v\in V$,
$$\theta(Y(u,x)v)
=Y(\theta(u), x)\theta(v). $$
We call $\theta$ the \textit{parity involution}. Physically, the even part corresponds to all the \textit{bosonic states}; the odd part corresponds to all the \textit{fermionic states}. 
\end{defn}

\begin{rmk}\label{equiv-defn-mosva}
Note that many results from $\Z$-graded MOSVAs still hold for $\frac{\Z}{2}$-graded MOSVAs. In particular, if the rationality axiom holds for $n=2$, and order of the pole at $z_{1}=0$ of the rational function given by the sum of the series 
$$\langle v', Y(u_{1},z_1)Y(u_{2},z_2)v\rangle$$
is bounded above by a positive integer independent of $u_{2}$ and $v'$, then, together with other axioms, the rationality holds for any $n>2$. See Proposition 2.1.12 in \cite{Q-Thesis}, Theorem 3.4 in \cite{Q-Mod}, and Remark 2.3 in \cite{FQ-Fermion-1}.
This pole-order condition is satisfied by all the current-existing examples of MOSVA. Therefore, it suffices to check the following \textit{weak associativity with pole-order condition}:
For every $u_1, v \in V$, there exists $p \in \mathbb{N}$
 such that for every $u_2 \in V$,
$$(x_0+x_2)^p Y(u_1, x_0+x_2)Y(u_2, x_2)v 
= (x_0+x_2)^p Y(Y(u_1, x_0)u_2, x_2)v.$$
\end{rmk}

To define $g$-twisted $V$-modules for an automorphism $g$ of $V$ of finite order, we introduce the following notation: for $z\in \C^\times$, define 
$$l_p(z) = \log |z| + i (\arg z + 2\pi p),\ \arg z \in [0, 2\pi)$$
as the $p$-th branch of the multi-valued logarithmic function $\Log z$. 
\begin{defn}\label{twisted-mod}
{\rm Let $(V, Y, \one)$ be a $\frac{\Z}{2}$-graded 
meromorphic open-string vertex algebra and 
$g$ an automorphism of $(V, Y, \one)$ of finite order $k$.	
A \textit{lower-bounded $g$-twisted $V$-module} is a $\mathbb{C}$-graded 
vector space $W=\coprod_{n\in\C, \alpha\in \C/\Z} W_{(n)}^{[\alpha]}$ (graded by {\it weights} and {\it $g$-weights}) equipped with a {\it twisted vertex operator map}
\begin{eqnarray*}
Y_W^g:  V\otimes W &\to & 
W[[x^{1/k},x^{-1/k}]]\\
v\otimes w &\mapsto& Y_W^g(v,x)w=\sum_{n\in \Z/k}\left(Y_W^g\right)_n(v) x^{-n-1},
\end{eqnarray*}
or equivalently, a map 
\begin{align*}
\left(Y_W^g\right)^{p}: \C^\times &\to \hom_{\C}(V\otimes W, \overline{W}),\nn
z&\mapsto \left(Y_W^g\right)^p(v, z) 
= \sum_{n\in \Z/k}\left(Y_W^g\right)_n(v) e^{(-n-1)l_p(z)}
\end{align*}
for each $p\in \Z$, where 
$\overline{W} = \prod_{n\in \C, \alpha\in \C/\Z} W_{(n)}^{[\alpha]}$
is the algebraic completion of $W$,
together with
an operator $\d_{W}$,
%and an operator $D_{W}$
 satisfying the following axioms:
\begin{enumerate}			

\item Axioms for the grading: 

\begin{enumerate}

\item \textit{Lower bound condition}:  When $\text{Re}{(m)}$ is sufficiently negative, $W_{(m)}^{[\alpha]}=0$ for every $\alpha\in \C/\Z$. 
				
\item  \textit{$\mathbf{d}$-grading and $g$-grading condition}: 
for every $w\in W_{(m)}^{[\alpha]}$, $\d_W w = m w$, $gw = e^{\frac{2\pi i \alpha }k}w$. Moreover, for $v\in V, w\in W$, 
$$g Y^g_W(v, x) w = Y_W^g(gv, x) gw. $$
				
\item  \textit{$\mathbf{d}$-commutator formula}: For $u\in V$, 
$$[\mathbf{d}_{W}, Y_W^g(u,x)]
= Y_W^g(\mathbf{d}_{V}u,x)+x\frac{d}{dx}Y_W^g(u,x).$$

\end{enumerate}
			
\item \textit{Identity property}: $Y_W^g(\one,x)=1_{W}$ (the identity operator on $W$).

\item {\it Generalized rationality}: 
For $0\le l<k$, let $V^{[l]}$ be the eigenspace of 
$g$ with eigenvalue $e^{\frac{2\pi i l}{k}}$. 
For $n\in \N$, 
$0\le k_{1}, \dots, k_{n}<k$,
 $v_1\in V^{[k_{1}]}, \dots, v_n\in V^{[k_{n}]}$,
$w\in W,\ w'\in W'$, there exists a generalized rational function 
$$f(z_1, \dots, z_n) = g(z_1, \dots, z_n)\prod_{i=1}^n z_i^{\frac{k_{i}}{k}}z_i^{-q_i}
\prod_{1\le i < j \le n}(z_i-z_j)^{-q_{ij}}, $$
where $q_{i}, q_{ij}\in \N$ and $g(z_1, \dots, z_n)$ is
a polynomial in $z_{1}, \dots, z_{n}$, such that for every $p\in \Z$,
$$\langle w', \left(Y_W^g\right)^p(v_1, z_1) \cdots 
\left(Y_W^g\right)^p(v_n, z_n)w\rangle$$
converges absolutely in the region $|z_1|>\cdots > |z_n|>0$
to the $p$-th branch 
$$f^{p, \dots, p}(z_1, \dots, z_n) =g(z_1, \dots, z_n) \prod_{i=1}^n e^{\frac{k_{i}}{k}l_{p}(z_{i})}
z_i^{-q_i}
\prod_{1\le i < j \le n}(z_i-z_j)^{-q_{ij}}$$
of $f(z_1, \dots, z_n)$. 
For $v_{1}\in V^{[k_{1}]}, v_{2}\in V^{[k_{2}]}$, $w\in W$
and $w'\in W'$, there exists an generalized rational function 
$$f(z_{1}, z_{2})=g(z_1, z_2)z_1^{\frac{k_{1}}{k}} z_1^{-q_1}
z_2^{\frac{k_{2}}{k}} z_2^{-q_2}
 (z_1-z_2)^{-t}, $$
where $q_{1}, q_{2}, t\in \N$ and 
$g(z_1, z_2)$ is
a polynomial in $z_{1}, z_{2}$ such that 
for every $p\in \Z$,
$$\langle w', \left(Y_W^g\right)^p(Y(v_{1}, z_1-z_{2})v_{2}, z_2)w\rangle$$
converges absolutely when $|z_{2}|>|z_{1}-z_{2}|>0$ and its sum 
is equal to  the branch 
$$f^{p, p}(z_{1}, z_{2})=g(z_1, z_2)e^{\frac{k_{i}}{k}l_{p}(z_{1})} z_1^{-q_1}
e^{\frac{k_{2}}{k}l_{p}(z_{2})} z_2^{-q_2}
 (z_1-z_2)^{-t}$$
on the region given by $|z_{2}|>|z_{1}-z_{2}|>0$ and 
$|\arg z_{1}-\arg z_{2}|<\frac{\pi}{2}$. 

\item {\it Associativity}: For $v_{1}, v_{2}\in V$, $w\in W$ and
$w'\in W'$, we have
\begin{equation}\label{associativity}
\langle w',  \left(Y_W^g\right)^p(v_{1},z_1)\left(Y_W^g\right)^p(v_{2},z_2)w\rangle
=\langle w', 
\left(Y_W^g\right)^p(Y(v_{1},z_{1}-z_{2})v_{2},z_2)w\rangle
\end{equation}
when  $|z_{1}|>|z_{2}|>|z_{1}-z_{2}|>0$
 and $|\arg z_{1}-\arg z_{2}|<\frac{\pi}{2}$. 

% \item The {\it $D$-derivative property} and the  {\it $D$-commutator formula}: 
% For $u\in V$,
% \begin{eqnarray}\label{L-1-W}
% \frac{d}{dx}Y_{W}(u, x)
% &=&Y_{W}(D_{V}u, x) \nn
% &=&[D_{W}, Y_{W}(u, x)].
% \end{eqnarray}

\end{enumerate}  
If $\dim W_{(n)}<\infty$ for $n\in \C$, we say that $W$ is a {\it 
grading-restricted 
$g$-twisted $V$-module}, or, simply a {\it $g$-twisted module}. }
\end{defn}

\begin{prop}\label{alg-func}
A lower-bounded $g$-twisted $V$-module satisfies the equivariance property: 
For all $v\in V$, $p\in \Z$, 
$$\left(Y_W^g\right)^{p+1}(g v, z)
=\left(Y_W^g\right)^{p}(v, z).$$
\end{prop}
\begin{proof}
This can be easily verified using the generalized
rationality in the case of $n=1$. 
\end{proof}

\begin{rmk}\label{equiv-defn-mod}
If the generalized rationality holds for $n=2$, and the number $q_{1}$ in the generalized rational function given by the sum of the series
$$\langle w', Y(u_{1},z_1)Y(u_{2},z_2)w\rangle$$
is bounded above by a positive integer independent of $u_{2}$ and $w'$, then, together with other axioms, the generalized rationality holds for any $n>2$. The proof is a straightforward generalization of Proposition 2.5 and 2.8 in \cite{Q-Fermion-2} and thus shall be omitted. Therefore, it suffices to check the following \textit{generalized weak associativity with pole-order condition}:
For $v_1\in V^{k_{1}}$ and $w\in W$, there exists 
$P\in \mathbb{N}$ such that for $v_2\in V^{k_{2}}$, 
\begin{align*}
& (x_0+x_2)^{P + \frac{k_{1}}{k}} x_2^{\frac{k_{2}}{k}} 
Y_W^g(v_1, x_0+x_2)Y_W^g(v_2, x_2)w \\
&\quad =  (x_2+x_0)^{P+\frac{k_{1}}{k}}x_2^{\frac{k_{2}}{k}} 
Y_W^g(Y(v_1, x_0)v_2, x_2)w 
\end{align*}
as series in $W((x_0, x_2))$. 
\end{rmk}
We shall restrict our attention to the particular case with $g=\theta$. For brevity, hereafter we shall omit $\theta$ and simply use the notation $Y_W$ for the 
parity-involution-twisted vertex operator.

\section{A construction of $\frac{\Z}{2}$-graded meromorphic open-string vertex algebras}\label{sec-algebra}

In this section, we construct a $\frac{\Z}{2}$-graded meromorphic open-string vertex algebra generated by both the bosonic and fermionic part from a finite-dimensional complex vector space \(\h\) with a complex nondegenerate symmetric bilinear form \((\cdot, \cdot)\). For an inner product space $\mathfrak{h}_{\mathbb{R}}$ over $\mathbb{R}$, we take $\mathfrak{h}$ to be the 
complexification of $\mathfrak{h}_{\mathbb{R}}$ and $(\cdot, \cdot)$ the symmetric bilinear from 
on $\mathfrak{h}$ obtained from the inner product on $\mathfrak{h}_{\mathbb{R}}$. The construction is a nontrivial combination of the bosonic construction in \cite{H-mosva} and the fermionic construction in \cite{FQ-Fermion-1}. 

\subsection{The algebra $N\left(\widehat{\h}_{\frac {\Z}{2}}\right)$ and its induced module}
Let $\h$ be a vector space 
with a nondegenerate symmetric bilinear form $(\cdot, \cdot)$. Let $\h^\B$ and $\h^\F$ be two copies of $\h$. We impose that $\h^\B \perp \h^\F$ in $\h^\B\oplus \h^\F$ with respect to $(\cdot, \cdot)$. 
Let
\begin{align*}
\widehat{\mathfrak{h}}_{\mathbb{Z}}
&=\mathfrak{h}^\B \otimes \mathbb{C}\left[t, t^{-1}\right] \oplus \mathbb{C} \mathbf{k}_\B,\\
\widehat{\mathfrak{h}}_{\Z+1/2}
&=\mathfrak{h}^\F \otimes t^{1/2} \mathbb{C}\left[t, t^{-1}\right] \oplus \C \mathbf{k}_{\F} .
\end{align*}
As vector spaces,  
$\widehat{\mathfrak{h}}_{\mathbb{Z}}
=\widehat{\mathfrak{h}}_{\mathbb{Z}}^{-}
 \oplus \widehat{\mathfrak{h}}_{\mathbb{Z}}^{0}
 \oplus \widehat{\mathfrak{h}}_{\mathbb{Z}}^{+}$ and  
 $\widehat{\mathfrak{h}}_{\Z+1/2}
 =\widehat{\mathfrak{h}}_{\Z+1/2}^{-} 
  \oplus \widehat{\h}_{\Z+1/2}^{0} \oplus \widehat{\mathfrak{h}}_{\Z+1/2}^{+}$, 
where
$$
    \widehat{\mathfrak{h}}_{\mathbb{Z}}^{\pm}=\h^\B\otimes t^{ \pm 1} 
\mathbb{C}\left[t^{ \pm 1}\right], \ 
\widehat{\mathfrak{h}}_{\mathbb{Z}}^{0}
=\mathfrak{h}^\B \otimes \C t^0 \oplus \mathbb{C} 
\mathbf{k} \cong \h^\B\oplus \C\k_\B,
$$
$$
  \widehat{\mathfrak{h}}_{\Z+1/2}^{\pm}=\h^\F\otimes 
t^{\pm 1/2} \mathbb{C}\left[t^{\pm 1}\right], \widehat{\h}_{\Z+1/2}^{0} = \C \k_\F.
$$
Let 
$$\widehat{\h}_{\frac{\Z}{2}}=\widehat{\mathfrak{h}}_{\mathbb{Z}}\oplus 
\widehat{\mathfrak{h}}_{\Z+1/2},\;
\widehat{\h}_{\frac{\Z}{2}}^{\pm}=\widehat{\mathfrak{h}}_{\mathbb{Z}}^{\pm}\oplus 
\widehat{\mathfrak{h}}_{\Z+1/2}^{\pm}, \;
\widehat{\h}_{\frac{\Z}{2}}^{0}=\widehat{\mathfrak{h}}_{\mathbb{Z}}^{0}\oplus 
\widehat{\mathfrak{h}}_{\Z+1/2}^{0}.$$
We define the parity $|u| = 0$ for $u\in \h^\B$, and $|u|=1$ for $u\in \h^\F$. Natually, the parity extends to
$\widehat{\mathfrak{h}}_{\mathbb{Z}}$ and 
$\widehat{\mathfrak{h}}_{\Z+1/2}$
as $0$ and $1$, respectively. 
Then $\widehat{\h}_{\frac{\Z}{2}}$ is $\Z_{2}$-graded.
Consider the tensor algebra 
$$T\left(\widehat{\h}_{\frac{\Z}{2}}\right)=
T\left(\widehat{\mathfrak{h}}_{\mathbb{Z}}\oplus 
\widehat{\mathfrak{h}}_{\Z+1/2}\right).$$ 
The $\Z_{2}$-grading on $\widehat{\h}_{\frac{\Z}{2}}$ induces a 
$\Z_{2}$-grading on $T\left(\widehat{\h}_{\frac{\Z}{2}}\right)$, also called the \textit{parity grading}. We will continue to use $|u|$ to denote the \textit{parity} of 
an $\Z_2$-homogeneous element $u\in T\left(\widehat{\h}_{\frac{\Z}{2}}\right)$

For two $\Z_2$-homogeneous elements $u_{1}$ and $u_{2}$ of $T\left(\widehat{\h}_{\frac{\Z}{2}}\right)$, we use $[u_{1}, u_{2}]$ to denote the \textit{$\Z_{2}$-graded commutator} of $u_{1}$ and $u_{2}$, that is,
\begin{equation}\label{Braket-definition}
    [u_{1}, u_{2}]= u_{1}u_{2}-(-1)^{|u_{1}| |u_{2}|} u_{2}u_{1}.
\end{equation}

Let \(N\left(\widehat{\h}_{\frac{\Z}{2}}\right)\) be
the quotient of the tensor algebra 
\(T\left(\widehat{\h}_{\frac{\Z}{2}}\right)\) 
by the two-sided ideal generated by elements of the following forms: 
\begin{equation}\label{Commutator to quotient}
\begin{aligned}
 &\left[\k_\B, \k_\F\right], \\
&\left[a\otimes t^k, \mathbf{k}_\B\right],\; \quad \text{ for }a\in \h^\B, k\in \Z,\\
&\left[a\otimes t^k, \mathbf{k}_\F\right],\; \quad \text{ for }a\in \h^\B, k\in \Z,\\
&\left[a\otimes t^{k+1/2}, \mathbf{k}_\B\right], \; \quad \text{ for }a\in \h^\F, k\in \Z,\\
&\left[a\otimes t^{k+1/2}, \mathbf{k}_\F\right], \; \quad \text{ for }a\in \h^\F, k\in \Z,\\
&\left[a \otimes  t^k , b \otimes  t^0\right],\; \quad \text{ for }a,b \in \h^\B, k\in \Z\setminus \{0\},\\
&\left[a \otimes  t^{k+1/2} , b \otimes  t^0\right],\; \quad \text{ for }a\in \h^\F ,b \in \h^\B, k\in \Z,\\
&\left[a \otimes t^m, b \otimes t^n\right]-m(
a, b) \delta_{m+n, 0} \mathbf{k}_\B,\; \quad \text{ for }a, b\in \h^\B, m\in\Z_+, n\in -\Z_+, \\
&\left[a \otimes t^{m+1/2}, 
b \otimes t^{n+1/2}\right]-(a, 
b) \delta_{m+n+1, 0} \mathbf{k}_\F, \;\quad \text{ for }a,b\in \h^\F, m\in\N, n\in -\Z_+,\\
&\left[a \otimes t^m, b \otimes 
t^{n+1/2}\right],\; \quad \text{ for }a\in \h^\B, b\in \h^\F, m\in\Z_+, n\in -\Z_+, \;
\text{or} \;m\in -\Z_+, n\in \Z_+ .
\end{aligned}
\end{equation}
Compared to the free boson-fermion vertex operator superalgebra (see for example, Section 3.3 in \cite{Barron-Review}), 
\begin{itemize}
    \item For $a, b\in \h^\B$, the elements $a\otimes t^m$ and $b \otimes t^n$ admit no relations if $m$ and $n$ are both positive, or both negative, or both zero. 
    \item For $a, b\in \h^\F$, the elements $a\otimes t^{m+1/2}$ and $b\otimes t^{n+1/2}$ admit no relation if $m+1/2$ and $n+1/2$ are both positive, or both negative. 
    \item For $a\in \h^\B$, $b\in \h^\F$, the elements $a\otimes t^{m}$ and $b\otimes t^{n+1/2}$ admit no relations if $m$ and $n+1/2$ are both positive, or both negative. 
\end{itemize} 

\begin{prop}[Poincaré-Birkhoff-Witt type theorem] \label{PBW1}
As a vector space, the quotient algebra $N\left(\widehat{\h}_{\frac{\Z}{2}}\right)$
is isomorphic to
$$T\left(\widehat{\h}_{\frac{\Z}{2}}^{-}\right)\otimes 
T\left(\widehat{\h}_{\frac{\Z}{2}}^{+}\right)
\otimes T\left(\h \right)\otimes T\left(\C\k_\B\right) \otimes T\left(\C\k_\F\right).$$
\end{prop}
\begin{proof}
The proof is a trivial combination of those of Proposition 3.1 in \cite{H-mosva} and Proposition 3.2 in \cite{FQ-Fermion-1}. We shall omit the details. 
\end{proof}

Consider now the one-dimensional complex vector space $\C\one$. Let
$\widehat{\h}_{\frac{\Z}{2}}^{+}$ and $\h^\B$ (as in $\widehat{\h}_{\frac{\Z}{2}}^{0})$ act on $\C\one$ 
trivially, $\k_\B$ act on $\C\one$ as a scalar $\ell_\B\in \C^\times$, and $\k_\F$ act on $\C\one$ as a scalar $\ell_\F\in \C^\times$. 
Then $\C\one$ is a module for the 
subalgebra 
$N\left(\widehat{\h}_{\frac{\Z}{2}}^{+}\oplus \widehat{\h}_{\frac{\Z}{2}}^{0}\right)$ of 
$N\left(\widehat{\h}_{\frac{\Z}{2}}\right)$ generated by 
$\widehat{\h}_{\frac{\Z}{2}}^{+}\oplus \widehat{\h}_{\frac{\Z}{2}}^{0}\cong \widehat{\h}_{\frac{\Z}{2}}^{+}\oplus 
\h^\B\oplus \C\k_\B \oplus \C \k_{\F}$. 
We then consider the induced module 
\begin{equation}
V=N\left(\widehat{\h}_{\frac{\Z}{2}}\right)\otimes_{N\left(\widehat{\h}_{\frac{\Z}{2}}^{+}\oplus \widehat{\h}_{\frac{\Z}{2}}^{0}\right)} \C\one. 
\label{vector space algebra}
\end{equation}
By Proposition \ref{PBW1}, $V$  as a vector space is isomorphic to 
$T\left(\widehat{\h}_{\frac{\Z}{2}}^{-}\right)\otimes\C\one\cong 
T\left(\widehat{\h}_{\frac{\Z}{2}}^{-}\right)$. We shall identify $V$ with 
$T\left(\widehat{\h}_{\frac{\Z}{2}}^{-}\right)$.
For $a\in \h^\B\cup \h^\F$, $m\in \Z$,
we denote the action of $a\otimes t^{m+|a|/2}$ on $V=T\left(\widehat{\h}_{\frac{\Z}{2}}^{-}\right)$
by $a(m+|a|/2)$. For $a\in \h^\B$, we call $a(m)$  
\textit{bosonic positive modes} if $m\in \Z_{+}$, 
\textit{bosonic negative modes} if $m\in -\Z_{+}$, 
\textit{bosonic zero modes} if $m=0$. For $a\in \h^\F$, we call $a(m+1/2)$ \textit{fermionic positive modes} if $m\in \N$, \textit{fermionic negative modes} if $m\in \Z_-$. Then it follows from Proposition 
\ref{PBW1} that $V$
is spanned by elements of the form
\begin{equation}
a_1(-m_1+|a_1|/2) \cdots a_r(-m_r+|a_r|/2) \one \label{generator}
\end{equation}
where $a_1, \ldots, a_r\in\h^\B \cup \h^\F$, $m_1, \ldots, m_r \in \Z_{+}$, $r\in \N$. We define the {\it weight} of
(\ref{generator}) to be $$m_1-|a_1|/2+\cdots+m_r-|a_r|/2,$$
and define the weight operator $\d$ on $V$ accordingly. 
Then $V=T\left(\widehat{\h}_{\frac{\Z}{2}}^{-}\right)$ becomes a 
$\frac{\Z}{2}$-graded vector space, that is, 
$$V=T\left(\widehat{\h}_{\frac{\Z}{2}}^{-}\right)=\coprod_{n \in \tfrac{\Z}{2}}\left(T\left(\widehat{\h}_{\frac{\Z}{2}}^{-}\right)\right)_{(n)},$$
where each $\left(T\left(\widehat{\h}_{\frac{\Z}{2}}^{-}\right)\right)_{(n)}$ 
denotes the homogeneous subspace of weight $n$. Clearly, for every $a\in \h^\B \cup \h^\F$ 
and $n\in \Z$, the mode $a(n+|a|/2)$ is a homogeneous operator with 
weight $-n-|a|/2$. 

\subsection{Generating functions of the modes and a normal-ordering $:\cdot:$}\label{gen-func-alg}

For every $a\in\h^\B \cup \h^\F$, we consider the series
\begin{equation}
a(x)=\sum_{n\in \Z} a(n+|a|/2) x^{-n-1}.
\end{equation}
For $a\in \h^\B$, let
\begin{align*}
a(x)^+ = \sum_{n\in \Z_+} a(n)x^{-n-1},\ a(x)^- = \sum_{n\in -\Z_+} a(n)x^{-n-1},\ a(x)^0 = a(0) x^{-1}.
\end{align*}
For $a\in \h^\F$, let 
\begin{align*}
a(x)^+ = \sum_{n\in \N} a\left(n+\tfrac{1}{2}\right)x^{-n-1},\  a(x)^- 
&= \sum_{n\in -\Z_+} a\left(n+\tfrac{1}{2}\right)x^{-n-1},\ a(x)^0 = 0. 
\end{align*}
Note that the $\pm$-notation stands for the positive/negative modes of the coefficients, instead of the positive/negative powers of $x$. For simplicity, we also introduce the notation 
\[a^{(m)}(x):=\frac{1}{m!}\frac{\partial^{m}}{ \partial x^{m}} a(x),\]
for $a\in \h^\B \cup \h^\F$. Clearly, taking derivatives commutes with taking positive or negative modes. 

\begin{prop} \label{algebra-commutator}
Let 
\begin{align*}
    F^\B_{mn}(x, y) &=\frac{1}{(m-1)!(n-1)!} \frac{\partial^{m+n-2}}{\partial x^{m-1} \partial y^{n-1}}(x-y)^{-2} =n\binom{-n-1}{m-1}(x-y)^{-m-n},\\
    F^\F_{mn}(x,y) &= \frac{1}{(m-1)!(n-1)!} \frac{\partial^{m+n-2}}{\partial x^{m-1} \partial y^{n-1}}(x-y)^{-1} = \binom{-n}{m-1}(x-y)^{-m-n+1}. 
\end{align*}
In particular, denote \[F^\B(x,y)=F^\B_{11}(x,y)=(x-y)^{-2},\ F^\F(x,y)=F^\F_{11}(x,y)=(x-y)^{-1}.\]
We have the following commutator formulas:
\begin{align*}
     \left[a^{(m-1)}(x)^+, b^{(n-1)}(y)^-\right]
     &=\ell_\B (a,b) \iota_{xy}F^\B_{mn}(x,y), &  a, b\in \h^\B\\
     \left[a^{(m-1)}(x)^+,b^{(n-1)}(y)^-\right]
    &=0, & a\in \h^\B, b\in h^\F\\
    \left[a^{(m-1)}(x)^+, b^{(n-1)}(y)^-\right]
    &=\ell_\F(a,b) \iota_{xy}F^\F_{mn}(x,y), & a, b\in \h^\F
\end{align*}
for all $m,n \in \Z_+$. Here, $\iota_{xy}$ expands negative powers of $x-y$ to its the binomial expansion 
in the nonnegative powers of $y$, i.e., for all $t\in \N$, 
$$\iota_{xy}{(x-y)^{-t}}=\sum_{i=0}^{\infty}\binom{-t}{i} x^{t-i}(-y)^i
=\sum_{i=0}^{\infty}\binom{t+i-1}{i} x^{t-i} y^i.$$
\end{prop}
\begin{proof}
    These all follow from direct calculations. 
\end{proof}

Although we did not assume any commutator relations between positive modes, surprisingly we still have all the commutator relations between all of them. More specifically, we have:

\begin{prop}\label{Pos-modes-commute}
    For \(a, b\in \h^\B \cup \h^\F\), \(m, n \in \N \), we have 
    \begin{equation}
        [a(m+|a|/2), b(n+|b|/2)]=0,
    \end{equation}
    where the commutator is defined by \eqref{Braket-definition}.
\end{prop}
\begin{proof} 
    We have the following three cases to consider: (1) \(a, b \in \h^\B\); (2) \(a,b \in \h^\F\); (3) one of \(a, b\) is in \(\h^\B\), and the other is in \(\h^\F\). Case (2) is proved in Proposition 3.3 of \cite{FQ-Fermion-1}. Case (1) can be proved by a straightforward modification, which we shall omit. Below, we elaborate Case (3).
    
    Without loss of generality, assume that \(a\in \h^\B\) and \(b\in\h^\F\). We act \([a(m), b(n+1/2)]\) on a basis element \(h_1(-k_1+|h_1|/2) \cdots h_r(-k_r+|h_r|/2) \one\), \(k_1, \ldots, k_r \in \Z_+\) and show that the action results in \(0\). We prove by induction on \(r\). Since \([a(m), b(n+1/2)]\one=0\), the conclusion clearly holds when \(r=0\). Assume the identity holds for all smaller \(r\). Then for $v=h_2(-k_2+|h_2|/2) \cdots h_r(-k_r + |h_r|/2)\one$, $[a(m), b(n+1/2)]v = 0$. 

    Now if $h_1\in \h^\B$, we calculate $[a(m), b(n+1/2)]  h_1(-k_1) v$:
    \begin{align*}
        &[a(m), b(n+1/2)]  h_1(-k_1) v \\
        &\quad = a(m) b(n+1/2) h_1(-k_1) v - b(n+1/2) a(m) h_1(-k_1) v\\
       &\quad= \left(a, h_1\right) \delta_{m, k_1} b(n+1/2) v + h_1(-k_1) a(m)b(n+1/2) v\\
        &\quad\quad  -(a, h_1) \delta_{m, k_1} b(n+1/2) v-h_1(-k_1) b(n+1/2)a(m) v\\
       &\quad =h_1(-k_1) [a(m), b(n+1/2)] v = 0.
    \end{align*}
    
    If $h_1\in \h^\F$, we calculate $[a(m), b(n+1/2)]  h_1(-k_1+1/2) v$:
    \begin{align*}
        &[a(m), b(n+1/2)]  h_1(-k_1+1/2) v\\
        &\quad=  a(m) b(n+1/2) h_1(-k_1+1/2) v - b(n+1/2) a(m) h_1(-k_1+1/2)v\\
      &\quad =(b, h_1) \delta_{n+1, k_1}  a(m) v -h_1(-k_1+1/2) a(m)b(n+1/2) v \\
       &\quad  -(b, h_1) \delta_{n+1, k_1}  a(m) v +h_1(-k_1+1/2) b(n+1/2)a(m) v\\
      &\quad = -h_1(-k_1+1/2)[a(m), b(n+1/2)]v = 0.
    \end{align*}
    Thus follows the conclusion. 
\end{proof}

For $a_1, \ldots, a_r\in\h^\B \cup \h^\F$, we partition $\{1, \dots, r\}$ as the disjoint union of
$$A^\F = \{i\in \{1, \dots, r\}: a_i\in \h^\F\}$$
and 
$$A^\B = \{i\in \{1, \dots, r\}: a_i\in \h^\B\}$$
Let $1\le f_1< \cdots < f_{r^\F}\le r$ be all the elements of $A^\F$.
Then for 
$m_1, \ldots, m_r \in \Z$, we define the {\it normal-ordered product of $a_1(m_1+|a_1|/2), \dots, a_r(m_r+|a_r|/2)$} by
$$: a_{1}(m_{1}+|a_1|/2)\cdots a_{r}(m_{r}+|a_r|/2):
=(-1)^{\sigma^{\F}}a_{\sigma(1)}(m_{\sigma(1)}+|a_{\sigma(1)}|/2)\cdots 
a_{\sigma(r)}(m_{\sigma(r)}+|a_{\sigma(r)}|/2),$$
where $\sigma\in \text{Sym}\{1, \dots, r\}$ is the unique permutation such that
\begin{gather*}
m_{\sigma(1)} + |a_{\sigma(1)}|/2 < 0, ..., m_{\sigma(\alpha)} + |a_{\sigma(\alpha)}|/2 < 0,\\
m_{\sigma(\alpha+1)} + |a_{\sigma(\alpha+1)}|/2 >0 , ..., m_{\sigma(\beta)} + |a_{\sigma(\beta)}|/2 > 0, \\
m_{\sigma(\beta+1)} + |a_{\sigma(\beta+1)}|/2 = 0, ..., m_{\sigma(r)} + |a_{\sigma(r)}|/2 = 0, \\
1\leq \sigma(1) <  \cdots < \sigma(\alpha) \leq r,\\ 
1\leq \sigma(\alpha+1) <  \cdots < \sigma(\beta)\leq r,\\
1 \leq \sigma(\beta+1)<  \cdots < \sigma(r) \leq r, 
\end{gather*}
and $\sigma^{\F}\in \text{Sym}\{f_1, \dots, f_{r^\F}\}$ is the permutation determined uniquely by $\sigma$. More precisely, let 
$$1\leq i_1 < \cdots < i_{r^\F} \leq r$$
such that 
$$\sigma(i_1), \dots, \sigma(i_{r^\F})\in A^\F. $$
Then $\sigma^\F$ is the unique permutation in $\text{Sym}\{f_1, \dots, f_{r^\F}\}$ defined by 
\begin{align}
    \sigma^\F(f_1) = \sigma(i_1), \dots, \sigma^\F(f_{r^\F}) = \sigma(i_{r^\F}). \label{sigma-F-defn}
\end{align}

    The definition can be intuitively understood by the following procedure: starting with the ordered sequence $(\sigma(1), \dots, \sigma(r))$, we filter out a subsequence $(\sigma(i_1), \dots, \sigma(i_{r^\F}))$ of elements in $A^\F$, then $\sigma^\F$ sends the sequence $(f_1, \dots, f_{r^\F})$ to the sequence $(\sigma(i_1), \dots, \sigma(i_{r^\F}))$. Note that $\sigma^\F$ is independent of the ordering of elements in $A^\B$.  

Roughly speaking, normal ordering arranges all negative modes, bosonic and
fermionic, to the left; all positive modes, bosonic and fermionic, to the
middle; and all bosonic zero modes to the right, with the corresponding sign
change accounting only for the fermionic modes, and maintains the relative
order among the negative modes, among the positive modes, and among the zero
modes. For example, 
\begin{align*}
 & : a_{1}(-\tfrac{3}{2}) a_{2}(5) a_{3}(0) a_{4}
  (\tfrac{5}{2}) a_5(-4) a_6(-7) a_7(0) a_8(-\tfrac{1}{2})a_9(2)a_{10}(\tfrac{9}{2}) :\nn
&\quad =-a_{1}(-\tfrac{3}{2})a_{5}(-4)a_{6}(-7)
a_{8}(-\tfrac{1}{2})a_2(5) a_4(\tfrac{5}{2}) 
a_{9}(2) a_{10}(\tfrac{9}{2})a_{3}(0)a_{7}(0), 
\end{align*}
where $\sigma^\F = (4\ 8)$.

We extend the normal ordering of products of operators above to 
normal ordering of products of series whose terms are products of 
such operators, by applying normal ordering to the coefficient of each 
power in the series obtained from the products. Then for $a_1, \dots, a_r\in \h^\B \cup \h^\F$, 
$$:a_1(x_1)\cdots a_r(x_r): = \sum_{m_1, \dots ,m_r \in \Z} :a_1(m_1+|a_1|/2)\cdots a_r(m_r+|a_r|/2): x_1^{-m_1-1} \cdots x_r^{-m_r-1}.$$
For $\mu, \nu = 0, \dots, r$, let $J_{\mu\nu}\{1, \dots, r\}$ be the set of \textit{3-shuffles} with size $\mu$ and $\nu$. Explicitly, $J_{\mu\nu}\{1, \dots, r\}$ is the collection of $\sigma\in \text{Sym}\{1, \dots, r\}$, such that
\begin{align*}
 \sigma(1) < \cdots < \sigma(\mu),\ \sigma(\mu+1)<\cdots <\sigma(\mu+\nu),\  \sigma(\mu+\nu+1)< \cdots < \sigma(r).
\end{align*}
Clearly, 
\begin{align*}
    :a_1(x_1)\cdots a_r(x_r): = \sum_{\mu=0}^r \sum_{\nu=0}^{r-\mu} \sum_{\sigma\in J_{\mu\nu}\{1, \dots, r\}}(-1)^{\sigma^\F}&  a_{\sigma(1)}(x_{\sigma(1)})^-\cdots a_{\sigma(\mu)}(x_{\sigma(\mu)})^-\\ 
    & \cdot a_{\sigma(\mu+1)}(x_{\sigma(\mu+1)})^+\cdots a_{\sigma(\mu+\nu)}(x_{\sigma(\mu+\nu)})^+\\
    & \cdot a_{\sigma(\mu+\nu+1)}(x_{\sigma(\mu+\nu+1)})^0\cdots a_{\sigma(r)}(x_{\sigma(r)})^0.
\end{align*}

\begin{rmk} \label{Nord-commutes-w-derivation-rmk}
   We have the following simple but important observation: The normal ordering operation commutes with partial derivation: for every \(i=1, \ldots, r\),
    \begin{equation}\label{Nord-commutes-w-derivation-eq}
    \frac{\partial}{\partial x_i} \left(: a_1^{(m_1-1)} (x_1)\cdots  a_r^{(m_r-1)}(x_r) :\right)=:\frac{\partial}{\partial x_i}\left(a_1^{(m_1-1)} (x_1)\cdots a_r^{(m_r-1)}(x_r)\right):\, ,
    \end{equation}
    therefore for many propositions in this paper, it suffices to prove the case for \(m_1=\cdots =m_r=0\). 
\end{rmk}

\subsection{The vertex operator map and the main theorem}

We now define a vertex operator map
$$Y: V \otimes V \rightarrow V((x))$$
as follows. First, we define
\begin{equation}
    Y(\mathbf{1}, x)=1_V.
\end{equation}
For $a\in \h^\B\cup \h^\F$ and $m\in \Z_{+}$, we define 
\begin{align}
    Y(a(-m+|a|/2)\one, x) =  a^{(m-1)}(x)=\frac 1 {(m-1)!} \frac{d^{m-1}}{dx^{m-1}}a(x). \label{vo-generator}
\end{align}
In greater detail, if $a\in \h^\B$, then for $m\in \Z_+$, 
$$Y(a(-m)\one, x) = \sum_{n\in \Z} a(n) (x^{-n-1})^{(m-1)}= \sum_{n\in \Z} a(n)\binom{-n-1}{m-1} x^{-n-m}; $$
if $a\in \h^\F$, then for $m\in \Z_+$, 
\begin{align}
    Y(a(-m+1/2)\one, x) = \sum_{n\in \Z} a(n+1/2) (x^{-n-1})^{(m-1)}= \sum_{n\in \Z} a(n+1/2)\binom{-n-1}{m-1} x^{-n-m},\label{vo-generator-fermion}
\end{align}
where we have used the notation
\begin{equation}\label{x-p-q-notation}
(x^{p})^{q}=\binom{p}{q}x^{p-q}
\end{equation}
for $p\in \Z$ and $q\in \N$. 
Note that the formula (\ref{vo-generator-fermion}) is indeed consistent with that in \cite{FQ-Fermion-1}, where $Y(a(-m-1/2)\one, x) = a^{(m)}(x).$ We modified the convention and used $a(-m+1/2)\one$ instead to allow a unified formula (\ref{vo-generator}) for the action of both the bosonic and the fermionic fields. 

For $a_1, \ldots, a_n\in\h^\B\cup\h^\F$, $m_1, \ldots, m_n \in \Z_+$,
we define
$$ Y\left( a_{1}(-m_{1}+|a_1|/2)\cdots a_{n}(-m_{n}+|a_n|/2)\one, x\right) 
=: a_1^{(m_1-1)}(x) \cdots a_n^{(m_n-1)}(x):.$$

\begin{thm}\label{It-is-a-mosva}
 The triple $(V, Y, \one)$ defined above
 is a $\frac{\Z}{2}$-graded meromorphic open-string vertex algebra. 
\end{thm}

Theorem \ref{It-is-a-mosva} will be proved in Section 6.

\section{A construction of canonically twisted $V$-modules}

In this section, we construct $\theta$-twisted modules for the MOSVA $V$ constructed in the preceding section, where $\theta$ is given in Definition \ref{defn-of-theta}. Such a twisted module is also called a \textit{canonically twisted module}. 

\subsection{The algebra $N\left(\widehat{\h}^{\B, \F}_{\Z}\right)$ and its induced module}
Consider the affinization of $\h^\B$ and $\h^\F$, namely, 
$$\widehat{\mathfrak{h}}_{\mathbb{Z}}^\B = \h^\B \otimes  \C[t, t^{-1}]\oplus \k_\B, \qquad \widehat{\mathfrak{h}}_{\mathbb{Z}}^\F = \h^\F \otimes  \C[t, t^{-1}]\oplus \k_\F. $$ 
Let 
$$\widehat{\h}_{\Z}^{\B,\F}=\widehat{\mathfrak{h}}_{\mathbb{Z}}^\B
\oplus \widehat{\mathfrak{h}}_{\mathbb{Z}}^\F.$$
Denote
$$\widehat{\mathfrak{h}}_{\mathbb{Z}}^\B=\widehat{\mathfrak{h}}_{\mathbb{Z}}^{\B,-} \oplus \widehat{\mathfrak{h}}_{\mathbb{Z}}^{\B,0}
 \oplus \widehat{\mathfrak{h}}_{\mathbb{Z}}^{\B,+},\  
\widehat{\mathfrak{h}}_{\mathbb{Z}}^\F=\widehat{\mathfrak{h}}_{\mathbb{Z}}^{\F,-} \oplus \widehat{\mathfrak{h}}_{\mathbb{Z}}^{\F,0}
 \oplus \widehat{\mathfrak{h}}_{\mathbb{Z}}^{\F,+},\ 
\widehat{\h}_{\Z}^{\B,\F}=\widehat{\h}_{\Z}^{\B,\F, -}\oplus 
\widehat{\h}_{\Z}^{\B,\F, 0}\oplus \widehat{\h}_{\Z}^{\B,\F, +},$$
where
\begin{align*}
 & \widehat{\mathfrak{h}}_{\mathbb{Z}}^{\B, \pm}=\mathfrak{h}^\B
\otimes t^{ \pm 1} \mathbb{C}\left[t^{ \pm 1}\right], & &
\widehat{\mathfrak{h}}_{\mathbb{Z}}^{\B,0}=\mathfrak{h}^\B\otimes \C t^0 \oplus \C \k_\B\cong \h^\B \oplus \C \k_\B, \\
 &\widehat{\mathfrak{h}}_{\mathbb{Z}}^{\F, \pm}=\mathfrak{h}^\F
\otimes t^{ \pm 1} \mathbb{C}\left[t^{ \pm 1}\right], && \widehat{\mathfrak{h}}_{\mathbb{Z}}^{\F,0}=\mathfrak{h}^\F \otimes \C t^0 \oplus \C \k_\F\cong \h^\F \oplus \C \k_\F,\\
&\widehat{\h}_{\Z}^{\B,\F, \pm}=\widehat{\mathfrak{h}}_{\mathbb{Z}}^{\B, \pm}
\oplus \widehat{\mathfrak{h}}_{\mathbb{Z}}^{\F, \pm}, & &
\widehat{\h}_{\Z}^{\B,\F, 0}=\widehat{\mathfrak{h}}_{\mathbb{Z}}^{\B,0}
\oplus  \widehat{\mathfrak{h}}_{\mathbb{Z}}^{\F,0}\cong \h^\B\oplus \h^\F\oplus \C\k_\B \oplus \C \k_\F.
\end{align*}
Consider the quotient $N\left(\widehat{\h}_{\Z}^{\B,\F}\right)$ of the tensor algebra
$T\left(\widehat{\h}_{\Z}^{\B,\F}\right)$
and its two-sided ideal generated by elements of the following forms:
\begin{equation}\label{relations module}
\begin{aligned}
    &\left[\k_\B, \k_\F\right], \\
    &\left[a\otimes t^m, \mathbf{k}_\B\right],  \quad \text{ for } a\in \h^\B\cup \h^\F, \;m\in \Z,\\
    &\left[a\otimes t^{m}, \mathbf{k}_\F\right],  \quad \text{ for } a\in \h^\B\cup \h^\F,\;m\in \Z,\\
    &\left[a \otimes  t^m , b \otimes  t^0\right],  \quad \text{ for } a, b\in \h^\B\cup h^\F, \; m\in \Z\setminus \{0\},\\
    &\left[a \otimes t^m, b\otimes t^{n} \right], \quad  \text{ for } a\in \h^\B, b\in \h^\F, \; m\in\Z_+,\; n\in -\Z_+,  \text{ or } m\in -\Z_+, \;n\in \Z_+, \\
    &\left[a \otimes t^m, b \otimes t^n\right]-m( a, b) \delta_{m+n, 0} \mathbf{k}_\B,   \quad \text{ for } a, b\in \h^\B,  \; m\in\Z_+, \; n\in -\Z_+,\\
    &\left[ a \otimes t^{m}, b\otimes t^{n}\right]-(a, b) \delta_{m+n, 0} 
    \mathbf{k}_\F, \quad  \text{ for } a, b\in \h^\F,  \; m\in\Z_+,\; n\in -\Z_+.
\end{aligned}
\end{equation}
Compared to the free boson-fermion vertex superalgebra (see again for example, Section 3.3 in \cite{Barron-Review}), 
\begin{itemize}
    \item For $a, b\in \h^\B$, the elements $a\otimes t^m$ and $b\otimes t^n$ admit no relations if $m$ and $n$ are both positive, or both negative, or both zero. 
    \item For $a, b\in \h^\F$, the elements $a \otimes t^m$ and $b\otimes t^n$ admit no relations if $m$ and $n$ are both positive, or both negative, or both zero. 
    \item For $a\in \h^\B$, $b\in \h^\F$, the elements $a \otimes t^m$ and $b\otimes t^n$ admit no relations if $m$ and $n$ are both positive, or both negative, or both zero. 
\end{itemize}

\begin{prop}[Poincaré–Birkhoff–Witt type theorem] \label{PBW2}
As a vector space, $N\left(\widehat{\h}_{\Z}^{\B,\F} \right)$ is isomorphic to 
$$T\left(\widehat{\h}_{\Z}^{\B,\F, -} \right)\otimes T\left(\widehat{\h}_{\Z}^{\B,\F, +} \right)
\otimes T\left(\h^\B \oplus \h^\F\right)\otimes T\left(\C\k_\B \oplus \C \k_\F \right).$$
\end{prop}

\begin{proof}
The proof is a trivial combination of those of Proposition 3.1 in \cite{H-mosva} and Proposition 3.2 in \cite{Q-Fermion-2}. We shall omit the details. 
\end{proof}

Let $M$ be a $T\left(\h^\B \oplus \h^\F\right)$-module.  
Assume that $M=\coprod_{n\in \C}M_{[n]}$ is $\C$-graded (graded by \textit{weights}) such that the action of 
$T\left(\h^\B \oplus \h^\F\right)$ preserves
the grading and the $\C$-grading is lower bounded. 
Let $T\left(\widehat{\h}_{\Z}^{\B,\F, +} \right)$ act on $M$ 
trivially, $\k_\B$ act on $M$ as a scalar $\ell_\B\in \C$, and $\k_\F$ act on $M$ as a scalar $\ell_\F\in \C$. 
Then $M$ is a module for the subalgebra $N\left(\widehat{\h}_{\Z}^{\B,\F, +} \oplus \widehat{\h}_{\Z}^{\B,\F, 0} \right)$ of $N\left(\widehat{\h}_{\Z}^{\B,\F} \right)$. 
We then consider the induced module 
\begin{equation}
W=N\left(\widehat{\h}_{\Z}^{\B,\F} \right)\otimes_{N\left(\widehat{\h}_{\Z}^{\B,\F, +} \oplus 
\widehat{\h}_{\Z}^{\B,\F, 0} \right)} M. 
\label{vector space module}
\end{equation}
By Proposition \ref{PBW2}, as vector spaces, $W$ and  
$T\left(\widehat{\h}_{\Z}^{\B,\F, -} \right)\otimes M$ are isomorphic. We shall identify $W$ with 
$T\left(\widehat{\h}_{\Z}^{\B,\F, -} \right)\otimes M$.
For $a\in \h^\B \cup \h^\F$, $m\in \Z$, we denote the action of $a\otimes t^m$ on $W=T\left(\widehat{\h}_{\Z}^{\B,\F, -} \right)\otimes M$ by $a(m)$. 
For $a\in\h^\B$ (resp. $a\in\h^\F$), we call $a(m)$ the \textit{bosonic} (resp. \textit{fermionic}) \textit{positive modes} if $m\in \Z_{+}$, \textit{bosonic} (resp. \textit{fermionic}) \textit{negative modes} if $m\in -\Z_{+}$, \textit{bosonic} (resp. \textit{fermionic}) \textit{zero modes} if $m=0$. 
Then it follows from Proposition 
\ref{PBW2} that $W$
is spanned by elements of the form
\begin{equation}
 a_1(-m_1) \cdots a_r(-m_r) w \label{generator-mod}
\end{equation}
where $r\in \N$, $a_1, \ldots, a_r\in\h^\B \cup \h^\F$, $m_1, \ldots, m_r \in \Z_+$, and 
$w\in M$ is a homogeneous module element. We define the {\it weight} of
(\ref{generator-mod}) to be $m_1+\cdots+m_r+\wt w$,
and define the weight operator $\d_{W}$ accordingly. 
Then $W$ is a $\C$-graded vector space, that is, 
$W=\coprod_{n \in \C}W_{[n]},$
where $W_{[n]}$ for each $n\in \C$
is the homogeneous subspace of weight $n$. Clearly, for every $a\in \h^\B \cup \h^\F$ 
and $n\in \Z$, the mode $a(n)$ is homogeneous of weight $-n$.

\subsection{Generating functions of the modes and a normal-ordering $\typecolon \cdot\typecolon$}\label{gen-func-mod}

For $a\in\h^\B \cup \h^\F$, we denote
\begin{align*}
a(x)&=\sum_{n\in \Z} a(n) x^{-n-1+|a|}.
\end{align*}
In particular, for $a\in \h^\B$, we have $a(x) = \sum_{n\in \Z} a(n)x^{-n-1}$, and
$$
a(x)^+ = \sum_{n\in \Z_+} a(n)x^{-n-1},\  a(x)^- = \sum_{n\in -\Z_+} a(n)x^{-n-1},\  a(x)^0 = a(0)x^{-1};
$$
for $a\in \h^\F$, we have $a(x) = \sum_{n\in \Z} a(n)x^{-n-\frac 1 2}$, and 
$$a(x)^+ = \sum_{n\in \Z_+} a(n)x^{-n-\frac 1 2},\  a(x)^- = \sum_{n\in -\Z_+} a(n)x^{-n-\frac 1 2},\  a(x)^0 = a(0)x^{-\frac 1 2}.$$
\begin{prop}\label{mod-basic-modes-commutator}
    Denote \(F^\B_{mn}(x,y)\) as previously in Proposition \ref{algebra-commutator}, and \[\overline{G}_{mn}(x,y):=\frac{1}{(m-1)!(n-1)!} \frac{\partial^{m+n-2}}{\partial x^{m-1} \partial y^{n-1}} 
\left(x^{-1/2}y^{1/2} (x-y)^{-1}\right).\]
Then we have the following commutator formulas:
\begin{enumerate}
\item For $a, b\in \h^\B$, 
\[\left[a^{(m-1)}(x)^+, b^{(n-1)}(y)^- \right]
=(a,b)F^\B_{mn}(x, y).\] 
    \item For $a\in \h^\B, b\in \h^\F$, or $a\in \h^\F, b\in \h^\B$, 
    \[\left[a^{(m-1)}(x)^+, b^{(n-1)}(y)^-\right]=0.\]
    \item For $a, b\in \h^\F$, 
    \[\left[a^{(m-1)}(x)^+, b^{(n-1)}(y)^-\right]
=(a,b)\overline{G}_{mn}(x,y).\]
\end{enumerate}
\end{prop}

\begin{proof}
The proof is follows from a straightforward calculation. We omit the details here.
\end{proof}

Fix $a_1, \ldots, a_r\in\h^\B \cup \h^\F$ and 
$m_1, \ldots, m_r \in \Z$, again let 
$$A^\F = \{i \in \{1, \dots, r\}: a_i \in \h^\F\} = \{f_1, \dots, f_{r^\F}\}$$
where $1\leq f_1 < \cdots < f_{r^\F}\leq r$. We define the \textit{normal ordering}  of 
$a_1(m_1) \cdots a_r(m_r) $ by 
\begin{align*}
\typecolon a_1(m_1) \cdots a_r(m_r) \typecolon
&=(-1)^{\sigma^\F} a_{\sigma(1)}(m_{\sigma(1)}) 
\cdots a_{\sigma(\alpha)}(m_{\sigma(\alpha)})\cdot a_{\sigma(\alpha+1)}(m_{\sigma(\alpha+1)})\cdots a_{\sigma(\beta)}(m_{\sigma(\beta)})\\
&\hspace{4em} \cdot \nord a_{\sigma(\beta+1)}(m_{\sigma(\beta+1)})
\cdots a_{\sigma(r)}(m_{\sigma(r)}) \nord, 
\end{align*}
where $\sigma$ is the unique permutation of $\{1, \ldots, r\}$ such that
\begin{align*}
&m_{\sigma(1)}<0, \ldots, m_{\sigma(\alpha)}<0,\  m_{\sigma(\alpha+1)}>0, 
\ldots, m_{\sigma(\beta)}>0,\  m_{\sigma(\beta+1)}= \cdots = m_{\sigma(r)}=0,\\
&1\le \sigma(1)<\cdots <\sigma(\alpha)\le r,\ 1 \le\sigma(\alpha+1)<\cdots <\sigma(\beta)\le r,\ 
1\le \sigma(\beta+1)<\dots< \sigma(r)\le r,
\end{align*}
$\sigma^{\F}\in \text{Sym}\{f_1,\dots, f_{r^\F}\}$ is the permutation determined uniquely by $\sigma$ as in (\ref{sigma-F-defn}), and the normal ordering
for a product of zero modes is defined recursively as follows: For $i=1, \dots, r$, we denote 
$$r^{\mathcal{F}}(i)=\#\{j\in \{1, \dots, i\}: a_j\in \mathfrak{h}^\mathcal{F}\}.$$
\begin{itemize}
    \item If \(a_1 \in \h^\B\), then
     \[\typecolon a_1(0) \cdots a_r(0)\typecolon=a_1(0)\typecolon a_2(0) \cdots a_r(0)\typecolon;\]
    \item If \(a_1 \in \h^\F\), then
\begin{align}
    \typecolon a_1(0) \cdots a_r(0)\typecolon = \ & a_1(0)\typecolon a_2(0) \cdots a_r(0)\typecolon \nonumber\\
    & +\sum_{\alpha=2}^{r^\F} \frac{(-1)^{\alpha-1}}{2}\left(a_1, a_{f_{\alpha}}\right)\typecolon a_2(0) \cdots \widehat{a_{f_{\alpha}}(0)} \cdots a_r(0)\typecolon \nonumber\\
    = \ & a_1(0)\typecolon a_2(0) \cdots a_r(0)\typecolon \nonumber\\
    & +\sum_{i\in A^\F\setminus\{1\}} \frac{(-1)^{r^\F(i)-1}}{2}\left(a_1, a_i\right)\typecolon a_2(0) \cdots \widehat{a_{i}(0)} \cdots a_r(0)\typecolon. \label{fermionic zero mode NO}
\end{align}
\end{itemize}
For example, say $a_1, \dots, a_{14}\in \h^\B \cup \h^\F$ with $A^\F = \{1, 4, 7, 8, 11, 13, 14\}$.
Then 
\begin{align*}
& \nord a_1(-3)a_2(5)a_3(0)a_4(5)a_5(-4)a_6(-7)a_7(0) a_8(0) a_9(0) a_{10}(0)a_{11}(-1) a_{12}(2)a_{13}(11) a_{14}(0) \nord\\
&\quad = - a_1(-3)a_5(-4)a_6(-7)a_{11}(-1)a_2(5)a_4(5)a_{12}(2)a_{13}(11) \nord a_3(0)a_7(0)a_8(0)a_9(0)a_{10}(0)a_{14}(0)\nord 
\end{align*}
where 
$$\sigma^\F = \begin{pmatrix}
1 & 4 & 7 & 8 & 11 & 13 & 14\\
1 & 11 & 4 & 13 & 7 & 8 & 14 
\end{pmatrix} = (4, 11, 7) (8, 13). $$
So in this case, $(-1)^{\sigma^\F} = -1$. We may further expand the normal-ordering of zero-modes using the recursion:  
\begin{align*}
& \nord a_3(0)a_7(0)a_8(0)a_9(0)a_{10}(0)a_{14}(0)\nord \\
&\quad=  a_3(0)a_7(0)a_8(0)a_9(0)a_{10}(0)a_{14}(0) - \frac{1}{2}(a_7, a_8) a_3(0)a_9(0)a_{10}(0)a_{14}(0) \\
& \quad\quad+ \frac{1}{2} (a_7, a_{14}) a_3(0)a_8(0)a_9(0)a_{10}(0) - \frac{1}{2}(a_8, a_{14}) a_3(0)a_7(0)a_9(0)a_{10}(0).
\end{align*}
% \begin{rmk}
% In case the usual Clifford relation exists among fermionic zero modes. Then from the recursion, if we ignore the existence of bosonic modes, for the leftover fermionic modes, we can further obtain that
% \begin{align*}
% \typecolon a^{\F}_1(0) \cdots a^{\F}_r(0)\typecolon
% =\frac{1}{r!} \sum_{\sigma \in \operatorname{Sym}\{1, \ldots, r\}}(-1)^\sigma a^{\F}_{\sigma(1)}(0) 
% \cdots a^{\F}_{\sigma(r)}(0),
% \end{align*}
% as in \cite{FFR}.
% \end{rmk}

\subsection{The twisted vertex operator map and the main theorem}

We define first a naive twisted vertex operator map
\[
\bar{Y}_W: V \otimes W \rightarrow W((x^{1/2}))
\]
as follows:
For $\mathbf{1} \in V$, we define
\begin{equation}
    \bar{Y}_W(\mathbf{1}, x)=1_W.
\end{equation}
For $a\in \h^\B \cup \h^\F$ and $m\in \Z_+$, we define 
\begin{equation}
\bar{Y}_W(a(-m+|a|/2)\one, x)= a^{(m-1)}(x).
\end{equation}
For $a_1, \ldots, a_n\in\h^\B\cup \h^\F$, $m_1, \ldots, m_n \in \Z_+$,
we define
$$\bar{Y}_W\left( a_{1}(-m_{1}+|a_1|/2)\cdots a_{n}(-m_{n}+|a_n|/2)\one, x\right) 
=\typecolon a_1^{(m_1-1)}(x) \cdots a_n^{(m_n-1)}(x)\typecolon.$$

As in  \cite{FLM}, \cite{FFR} and \cite{Q-Fermion-2}, to give the correct twisted vertex operator map,
we need to introduce a formal series of operators on $V$. Let $e_1, \dots, e_d$ be an orthonormal basis of $\h^\F$ with respect to \((\cdot, \cdot)\). Recall that by an orthonormal basis with respect to the nondegenerate symmetric bilinear 
form \((\cdot, \cdot)\) over $\C$, we mean
a basis $\{e_1, \dots, e_d\}$ such that $(e_{i}, e_{j})=\delta_{ij}$ for $i, j=1, \dots, d$. 
By the Takagi decomposition of complex symmetric matrices, such an orthonormal basis exists. 
Define
\begin{equation}
\Delta(x)=\sum_{i=1}^d \sum_{m,n\geqslant 1} \frac 1 2 C_{mn}e_i(m-\tfrac{1}{2}) e_i(n-\tfrac{1}{2})x^{-m-n+1}
\end{equation}
where 
\[C_{m n}=\frac{1}{2} \frac{m-n}{m+n-1}\binom{-\tfrac{1}{2}}{m-1}\binom{-\tfrac{1}{2}}{n-1}.\]
Note that $C_{mn}=-C_{nm}$ and $\Delta(x)$ is independent of the choice of the orthonomal basis. Then we define
\begin{equation}
\exp(\Delta(x)):=\sum_{t\in \N} \frac{\left(\Delta(x)\right)^{t}}{t!}.
\end{equation}
We define the twisted vertex operator map 
\[
Y_W: V \otimes W \rightarrow W((x^{1/2}))
\]
by
\begin{equation}
    Y_W(u,x)= \bar{Y}_W\left(\exp(\Delta(x))u,x\right)
\end{equation}
for $u\in V$. 

\begin{thm}\label{it-is-a-module}
The $\C$-graded  vector space $W$ together with the twisted vertex operator map \(Y_W\) 
is a canonically twisted module for the  $\frac{\Z}{2}$-graded meromorphic 
open-string vertex algebra $V$ constructed in Section \ref{sec-algebra}.
\end{thm}

Theorem \ref{it-is-a-module} will be proved in Section 7.
\section{Some notations and lemmas}

Our strategy of proving Theorem \ref{It-is-a-mosva} and Theorem \ref{it-is-a-module} is to obtain explicit formulas for products of normal-ordered products. The proofs also involve manipulation of shuffles. In this section, we shall introduce necessary notations and lemmas to prepare for the proofs.

\subsection{Notations of indices} Let $r, s\in \Z_+,\ a_1, \dots, a_r,\ b_1, \dots, b_s\in \h^\B\cup \h^\F.$ We introduce the following notations:

\begin{itemize}[leftmargin=*]
\item Let \begin{align*}
        A^\B =  \{i  \in \{1, \dots, r\}: a_i \in \h^\B\} & ,\ A^\F = \{i \in \{1, \dots, r\}: a_i \in \h^\F\}, \\
        B^\B =  \{i  \in \{1, \dots, s\}: b_i \in \h^\B\} & ,\ B^\F = \{i \in \{1, \dots, s\}: b_i \in \h^\F\}. 
    \end{align*}
    Let 
    \begin{align*}
        r^\B = \#A^\B,\ r^\F = \#A^\F,\ s^\B = \# B^{\B},\ s^\F = \#B^{\F}. 
    \end{align*}
\item For $\rho\in \N$, let 
\begin{align*}
    A^\B_\rho = \ & \{(i_1, \dots, i_\rho): 1\leq i_1 < \cdots < i_\rho \leq r, i_1, \dots, i_\rho \in A^\B\},\\
    A^\F_\rho = \ & \{(i_1, \dots, i_\rho): 1\leq i_1 < \cdots < i_\rho \leq r, i_1, \dots, i_\rho  \in A^\F\},\\
    B^\B_\rho = \ & \{(i_1, \dots, i_\rho): 1\leq i_1 < \cdots < i_\rho \leq s, i_1, \dots, i_\rho  \in B^\B\},\\     
    B^\F_\rho = \ & \{(i_1, \dots, i_\rho): 1\leq i_1 < \cdots < i_\rho \leq s, i_1, \dots, i_\rho  \in B^\F\}.
\end{align*}
Clearly, 
\begin{align*}
    \#A^\B_\rho = \binom{r^\B}{\rho}, \#A^\F_\rho = \binom{r^\F}{\rho}, \#B^\B_\rho = \binom{s^\B}{\rho}, \#B^\F_\rho = \binom{s^\F}{\rho}.
\end{align*}
Every $I=(i_{1}, \dots, i_{\rho})\in A^\B_\rho$ uniquely determines an element 
$I^c=(i_{1}^{c}, \dots, i_{r^\B-\rho}^c)\in A^\B_{r^\B-\rho}$ such that 
$$1 \leq i_1^c < \cdots < i_{r^\B-\rho}^c \leq r,\ \{i_1, \dots, i_\rho\} \sqcup \{i_1^c, \dots, i_{r^\B-\rho}^c\} = A^\B. $$
Similarly, we define $I^c \in A^\F_{r^\F-\rho}$ for $I\in A^\F_\rho$, $I^c \in B^\B_{s^\B-\rho}$ for $I\in B^\B_\rho$, $I^c \in B^\F_{s^\F-\rho}$ for $I\in B^\F_\rho$. 
\item For $\sigma\in \text{Sym}\{1, \dots, r\}$, let 
\begin{align*}
    A^\B_{\sigma} =  \{i\in \{1, \dots, r\}: \sigma(i) \in A^\B\}&,\
    A^\F_{\sigma} =  \{i\in \{1, \dots, r\}: \sigma(i) \in A^\F\},\\
    B^\B_{\sigma} =  \{i\in \{1, \dots, s\}: \sigma(i) \in B^\B\}&,\
    B^\F_{\sigma} =  \{i\in \{1, \dots, s\}: \sigma(i) \in B^\F\}.
\end{align*}
The slight abuse of notation is justified by the convention that $\sigma$ always means a permutation. We also introduce 
\begin{align*}
r^\B_\sigma(i) & = \#\{j \in \{1, \dots, i\}: j \in A^\B_\sigma\} = \#\{j\in \{1, \dots, i\}: \sigma(j)\in A^\B\}, \\
r^\F_\sigma(i) & = \#\{j \in \{1, \dots, i\}: j \in A^\F_\sigma\} = \#\{j\in \{1, \dots, i\}: \sigma(j)\in A^\F\}, \\
s^\B_\sigma(i) & = \#\{j \in \{1, \dots, i\}: j \in B^\B_\sigma\} = \#\{j\in \{1, \dots, i\}: \sigma(j)\in B^\B\}, \\
s^\F_\sigma(i) & = \#\{j \in \{1, \dots, i\}: j \in B^\F_\sigma\} = \#\{j\in \{1, \dots, i\}: \sigma(j)\in B^\F\}. 
\end{align*}
When $\sigma=(1)$, we should also abbreviate it and use $r^\B(i), r^\F(i), s^\B(i), s^\F(i)$. Note that we have already used this notation in the definition of the normal ordering. 
\item For $p \leq q$, we will also need the notation
\begin{align*}
    A^\B[p,q] =  \{i  \in \{p, \dots, q\}: a_i \in \h^\B\} & , A^\F[p,q] = \{i \in \{p, \dots, q\}: a_i \in \h^\F\}, \\
    B^\B[p,q] =  \{i  \in \{p, \dots, q\}: b_i \in \h^\B\} & , B^\F[p,q] = \{i \in \{p, \dots, q\}: b_i \in \h^\F\}, 
\end{align*}
and we define \(A^\B_{\sigma}[p, q]\), \(A^\F_{\sigma}[p, q]\), \(B^\B_{\sigma}[p, q]\), \(B^\F_{\sigma}[p, q]\) analogously.
\item For $\mu, \nu\in \N$, let 
\begin{align*}
    A^\F_{\mu, \nu} = \ & \{(i_1, \dots, i_\mu, i_{\mu+1}, \dots, i_{\mu+\nu}): 1\leq i_1 < \cdots < i_\mu \leq r,  1\leq i_{\mu+1} < \cdots < i_{\mu+\nu} \leq r, \\
    & \hspace{10.5em} i_1, \dots, i_\mu, i_{\mu+1}, \dots, i_{\mu+\nu}\in A^\F \text{ mutually distinct}\},\\    
    B^\F_{\mu, \nu} = \ & \{(i_1, \dots, i_\mu, i_{\mu+1}, \dots, i_{\mu+\nu}): 1\leq i_1 < \cdots < i_\mu \leq s,  1\leq i_{\mu+1} < \cdots < i_{\mu+\nu} \leq s, \\
    & \hspace{10.5em} i_1, \dots, i_\mu, i_{\mu+1}, \dots, i_{\mu+\nu}\in B^\F \text{ mutually distinct}\}.
\end{align*}
\end{itemize}

\subsection{Products with omitted terms}\label{sec-omit-def} Since $A^\B$ and $A^\F$ may be mixed and scattered within $\{1, \dots, r\}$, and the same holds for $B^\B$ and $B^\F$, the conventional hat-notation for omission cannot be easily adapted. Instead, for $\rho\in \Z_+, I^b\in A_\rho^\B, I^f \in A_\rho^\F, J^b\in B_\rho^\B, J^f\in B_\rho^\F$, we introduce the following:
\begin{itemize}[leftmargin=*]
\item The notation 
$$\bigg(a_1(x_1)\cdots a_r(x_r) b_1(y_1) \cdots b_s(y_s)\bigg)_{O(I^b, I^f, J^b, J^f)}$$
means 
\begin{align}
\prod_{\substack{1\leq p \leq r\\ p\neq i^b_1, \dots, i^b_\rho, i^f_1, \dots, i^f_\rho}}a_p(x_p) \prod_{\substack{1\leq q \leq s\\ q\neq j^b_1, \dots, j^b_\rho, j^f_1, \dots, j^f_\rho}} b_q(x_q)\label{omit-def}
\end{align}
\item We use the notation 
\begin{align}
:a_1(x_1)\cdots a_r(x_r) b_1(y_1) \cdots b_s(y_s):_{O(I^b, I^f, J^b, J^f)}\label{formula-omit-1}
\end{align}
and 
$$\nord a_1(x_1)\cdots a_r(x_r) b_1(y_1) \cdots b_s(y_s)\nord_{O(I^b, I^f, J^b, J^f)}$$
for the normal-ordered product of (\ref{omit-def}) in the context of MOSVA $V$ and the twisted module $W$, respectively. To emphasize the operations to be carried out for exposition purposes, we sometimes add in non-existential terms into such a product with the hat-notation. For example, for every $1\leq k \leq \eta$, the normal-ordered product 
$$:a_1(x_1) \cdots \reallywidehat{a_{i_k^b}(x_{i^b_k})}\cdots a_r(x_r) b_1(y_1) \cdots b_s(y_s):_{O(I^b, I^f, J^b, J^f)}$$
means exactly the same thing as (\ref{formula-omit-1}). 
\item 
In particular, we use the notation 
$$:a_1(x)\cdots a_r(x):_{O(I^b, I^f)}$$
for the normal-ordered product 
$$:\prod_{\substack{1\leq p \leq s, \\p\neq i^b_1, \dots, i^b_\rho, i^f_1, \dots, i^f_\rho}}a_p(x_p):\ , $$
and the notation 
$$\bigg(b_1(-n_1+|b_1|/2) \cdots b_s(-n_s+|b_s|/2)\one\bigg)_{O(J^b, J^f)}$$
for the product 
$$\left(\prod_{\substack{1\leq q \leq s\\ q\neq j^b_1, \dots, j^b_\rho, j^f_1, \dots, j^f_\rho}} b_q(-n_q + |b_q|/2) \right) \one.$$
\end{itemize}
For $\mu, \nu\in \Z_+, I^f\in A^\F_{\mu,\nu}, J^f\in B^\F_{\mu, \nu}$, we introduce similar notations.

\subsection{2-shuffles, 3-shuffles, and their signs} Many calculations are essentially manipulations of 2-shuffles and 3-shuffles. We recall their definitions here. Fix $r\in \Z_+$. For each $\rho\in \N$, let 
\begin{align*}
J_\rho(1, \dots, r) = \{\sigma\in \text{Sym}\{1, \dots, r\}: 1\leq \sigma(1) < \cdots < \sigma(\rho)\leq r, 1\leq \sigma(\rho+1) < \cdots < \sigma(r)\leq r\}. 
\end{align*}
Elements in $J_\rho(1, \dots, r)$ is called a \textit{2-shuffle}. Recall that
$A^\F = \{f_1, \dots, f_{r^\F}\}$. Clearly, an element $I^f =(i_1,\dots, i_\rho)\in A_\rho^\F$ corresponds to a 2-shuffle $\sigma\in J_\rho(1, \dots, r^\F)$ via 
$$i_1 = f_{\sigma(1)}, \dots, i_\rho = f_{\sigma(\rho)},\ i_1^c = f_{\sigma(\rho+1)}, \dots, i^c_{r^\F-\rho} = f_{\sigma(r^\F)}. $$
We shall use the notation $(-1)^{\sigma(I^{f})}$ for the sign of the corresponding shuffle 
as an element in $\text{Sym}\{1, \dots, r^\F\}.$
Similar correspondences exist between $B^\F_\rho$ and $J_\rho(1, \dots, s^\F)$. We likewise introduce $(-1)^{\sigma(J^f)}$ for $J^f\in B^\F_{\mu,\nu}$. 

We recall Lemma 4.3 in \cite{Q-Fermion-2} regarding the sign of a 2-shuffle. 
\begin{lemma}[\cite{Q-Fermion-2}]\label{2-shuffle-lemma}
    Let $\sigma\in J_\mu(1, \dots, r)$ be a 2-shuffle. Let 
    $$p_1 = \sigma(1), \dots, p_\mu = \sigma(\mu).$$
    Then $\sigma$ can be expressed as a product of cycles
    $$\sigma = (p_1, p_1-1, \dots, 2, 1) (p_2, p_2-1,\dots, 3, 2)\cdots (p_\mu, p_{\mu-1}, \dots, \mu+1, \mu)$$
    In particular, we have 
    $$(-1)^\sigma = (-1)^{p_1-1+p_2-2+\cdots + p_\mu-\mu} = (-1)^{p_1+\cdots + p_\mu} (-1)^{\frac{\mu(\mu+1)}2}. $$
\end{lemma}

For each $\mu, \nu\in \N$, let
\begin{align*}
J_{\mu,\nu}(1, \dots, r) = \{\sigma\in \text{Sym}\{1, \dots, r\}: \ & 1\leq \sigma(1) < \cdots < \sigma(\mu)\leq r, \\
& 1 \leq \sigma(\mu+1) < \cdots < \sigma(\mu+\nu) \leq r,\\
& 1 \leq \sigma(\mu+\nu+1)<\cdots < \sigma(r) \leq r \}.
\end{align*}
Elements in $J_{\mu,\nu}(1, \dots, r)$ is called a 3-shuffle. Clearly, an element $I^f \in A_{\mu,\nu}^\F$ corresponds to a 3-shuffle $\sigma\in J_{\mu,\nu}(1, \dots, r^\F)$ via 
$$i_1 = f_{\sigma(1)}, \dots, i_\mu = f_{\sigma(\mu)},\ i_{\mu+1} = f_{\sigma(\mu+1)}, \dots, i_{\mu+\nu} = f_{\sigma(\mu+\nu)},$$
and
$$i_1^c = f_{\sigma(\mu+\nu+1)}, \dots, i_{r^\F-\mu-\nu}^c = f_{\sigma(r^\F)}. $$
We shall use the notation $(-1)^{\sigma(I^f)}$ for the sign of the corresponding shuffle as an element in $\text{Sym}\{1, \dots, r^\F\}.$
Similar correspondences exists between $B^\F_{\mu,\nu}$ and $J_{\mu,\nu}(1, \dots, s^\F)$. We likewise introduce $(-1)^{\sigma(J^f)}$ for $J^f\in B^\F_{\mu,\nu}$.

The easiest way to understand the sign of a 3-shuffle is to view it as an iteration of 2-shuffles. 
More precisely, let $\sigma\in J_{\mu,\nu}(1, \dots, r)$ be a 3-shuffle. Let 
$$i_1 = \sigma(1), \dots, i_\mu = \sigma(\mu), $$
and 
$$1\leq i_1^c < \cdots < i_{r-\mu}^c \leq r$$
such that 
$$\{\sigma(\mu+1), \dots, \sigma(r)\} = \{i_1^c, \dots, i_{r-\mu}^c\}.$$
Then the permutation $\sigma_1$ defined by
$$\sigma_1 (1) = i_1, \dots, \sigma_1(\mu) = i_\mu, \sigma_1(\mu+1) = i_1^c, \dots, \sigma_1(r) = i_{r-\mu}^c$$
is an element in $J_\mu(1, \dots, r)$. The permutation defined by 
$$\sigma(\mu+1) = \sigma_2(i_1^c), \dots, \sigma(r) = \sigma_2(i_{r-\mu}^c)$$
is then an element in $J_\nu(i_1^c, \dots, i_{r-\mu}^c)$. $\sigma_2$ extends to an element in Sym$\{1, \dots, r\}$, which we shall denote also by $\sigma_2$. By a straightforward comparison of images, we have 
$$\sigma = \sigma_2 \circ \sigma_1$$
As a result, 
$$(-1)^\sigma = (-1)^{\sigma_2} (-1)^{\sigma_1}. $$
If we set 
$$\sigma_2(i_1^c) = i_{j_1}^c, \dots, \sigma_2(i_{r-\mu}^c) = i_{j_{r-\mu}}^c,$$
then by Lemma \ref{2-shuffle-lemma}
\begin{align}
(-1)^\sigma = (-1)^{i_1+\cdots + i_\mu}(-1)^{\frac{\mu(\mu+1)}{2}}\cdot (-1)^{j_1 + \cdots + j_\nu} (-1)^{\frac{\nu(\nu+1)}{2}}. \label{sign-3-shuffle}
\end{align}

\subsection{Change of the orders of summations} 

We give generalizations of Lemma 4.1 and Remark 4.2 in \cite{Q-Fermion-2}. 
\begin{lemma}\label{Comb-Id-Lemma} 
Let $r,\mu,\nu\in \Z_+$ with $\mu+\nu\leq r-1$. Then for any algebraic expression \(\Phi\) with one argument and any algebraic expression $\Psi$ with $r-1$ arguments,
\begin{align}
    & \sum_{\sigma\in J_{\mu,\nu}(1, \dots, r)} \sum_{i\in A_\sigma^\F[1,\mu]} (-1)^{r^\F_\sigma(i)}(-1)^{\sigma^\F} \Phi(\sigma(i)) \Psi(\sigma(1), \dots, \reallywidehat{\sigma(i)}, \dots,
    \sigma(\mu), \sigma(\mu+1), \dots, \sigma(r)) \nonumber\\
    & \qquad \quad= \sum_{j\in A^\F} (-1)^{r^\F(j)}  \Phi(j) \sum_{\tau \in J_{\mu-1, \nu}(1,\dots,\widehat{j},\dots,r)} (-1)^{\tau^\F} \Psi(\tau(1), \dots, \reallywidehat{\tau(j)}, \dots, \tau(r)), \label{Comb-Id-1}\\
    & \sum_{\sigma\in J_{\mu, \nu}(1, \dots, r)} \sum_{i \in A_\sigma^\F[\mu+1, \mu+\nu]} (-1)^{r^\F_\sigma(i)-r^\F_\sigma(\mu)}(-1)^{\sigma^\F}\nonumber \\
&\quad\quad\quad\quad\quad\quad\quad\quad\quad\quad\quad \cdot 
\Phi(\sigma(i))  \Psi(\sigma(1), \dots, \sigma(\mu+1), \dots, \reallywidehat{\sigma(i)},\dots,\sigma(\mu+\nu) \dots,
    \sigma(r))\nonumber\\
    &\qquad  \quad= \sum_{j\in A^\F} (-1)^{r^\F(j)} \Phi(j) \sum_{\tau \in J_{\mu, \nu-1}(1,\dots,\widehat{j},\dots,r)} (-1)^{\tau^\F} \Psi(\tau(1), \dots, \reallywidehat{\tau(j)}, \dots, \tau(r)) , \label{Comb-Id-2}\\
   & \sum_{\sigma\in J_{\mu, \nu}(1, \dots, r)} \sum_{i\in A_\sigma^\F [\mu+\nu+1,r]} (-1)^{r_\sigma^\F(i)-r_\sigma^\F(\mu+\nu)}(-1)^{\sigma^\F} \nonumber \\
&\quad\quad\quad\quad\quad\quad\quad\quad\quad\quad\quad \cdot \Phi(\sigma(i)) \Psi(\sigma(1), \dots, 
    \sigma(\mu+\nu+1), \dots, \reallywidehat{\sigma(i)}, \dots,\sigma(r)) \nonumber\\
    & \qquad \quad= \sum_{j\in A^\F} (-1)^{r^\F(j)}  \Phi(j) \sum_{\tau \in J_{\mu, \nu}(1,\dots,\widehat{j},\dots,r)} (-1)^{\tau^\F} \Psi(\tau(1),\dots,  \reallywidehat{\tau(j)}, \dots, \tau(r)). \label{Comb-Id-3}
\end{align}
\end{lemma}

\begin{proof}
We only elaborate (\ref{Comb-Id-1}) here. Rewrite the left-hand side of (\ref{Comb-Id-1}) as 
\begin{align*}
    & \sum_{j=1}^r \sum_{\substack{\sigma\in J_{\mu,\nu}(1, \dots, r)\\ \sigma(1) = j\in A^\F}}  (-1)^{r^\F_\sigma(1)}(-1)^{\sigma^\F} \Phi(\sigma(1)) \Psi(\reallywidehat{\sigma(1)}, \sigma(2), \dots, 
    \sigma(\mu), \sigma(\mu+1), \dots, \sigma(r))\\
    &\quad\quad  + \cdots + \sum_{j=1}^r \sum_{\substack{\sigma\in J_{\mu,\nu}(1, \dots, r)\\ \sigma(i) = j\in A^\F}}  (-1)^{r^\F_\sigma(i)}(-1)^{\sigma^\F} \nonumber \\
&\quad\quad\quad\quad\quad\quad\quad\quad\quad\quad\quad \cdot\Phi(\sigma(i)) \Psi(\sigma(1), \dots, \reallywidehat{\sigma(i)}, \dots,
    \sigma(\mu), \sigma(\mu+1), \dots, \sigma(r))\\
    &\quad\quad   + \cdots + \sum_{j=1}^r \sum_{\substack{\sigma\in J_{\mu,\nu}(1, \dots, r)\\ \sigma(\mu) = j\in A^\F}}  (-1)^{r^\F_\sigma(\mu)}(-1)^{\sigma^\F} \Phi(\sigma(\mu)) \Psi(\sigma(1), \dots,
    \reallywidehat{\sigma(\mu)}, \sigma(\mu+1), \dots, \sigma(r)).
\end{align*}
    Fix $i=1, \ldots, \mu$ and $\sigma\in J_{\mu\nu}(1, \ldots, r)$. If we regard $\sigma$ as three increasing sequences of numbers in $\{1, \ldots, r\}$ respectively of sizes $\mu, \nu, s-\mu-\nu$ with $j$ sitting in the $i$-th position of the first sequence, then removing $j$ results in three increasing sequences of numbers in $\{1, \ldots, \widehat{j}, \ldots, r\}$ repsectively of sizes $\mu-1, \nu, s-\mu-\nu$. Let $\tau$ be the corresponding 3-shuffle in $J_{\mu-1, \nu}(1, \ldots,\widehat{j}, \ldots, s)$. Then in case $i<j$, we know 
    \begin{align*}
        & \tau(1) = \sigma(1), \ldots , \tau(i-1) =\sigma(i-1)<j <\tau(i) = \sigma(i+1),\\ 
        & \tau(i+1)=\sigma(i+2), \ldots, \tau(j-1)=\sigma(j), \tau(j+1)=\sigma(j+1), \ldots, \tau(r)=\sigma(r)
    \end{align*}
    In case $i = j$, we know
    \begin{align*}
        & \tau(1) = \sigma(1), \ldots, \tau(i-1) =\sigma(i-1)\\ 
        & \tau(i+1)=\sigma(i+1), \ldots, \tau(r)=\sigma(r).
    \end{align*}
    Notice that $i>j$ is impossible. Note that from $\tau$ one can recover $\sigma$ by first defining an extension
    $$\tau^e(k) = \left\{\begin{array}{ll}
    \tau(k) & k\neq j\\
    j & k = j
    \end{array}\right.,$$
    then performing transpositions to move $j$ consecutively to the $i$-th position. 

    Having understood $\sigma$ and $\tau$, let us proceed to understand the relation between $\sigma^\F$ and $\tau^\F$. We start by recalling the definition of $\sigma^\F$. Again 
    let $1\leq f_1< \cdots < f_{r^\F} \leq r$ such that $A^\F=\{f_1, \dots, f_{r^\F}\}$. Let $1\leq \alpha_1 < \cdots < \alpha_{r^\F} \leq r$ such that $A^\F_\sigma=\{\alpha_1, \dots, \alpha_{r^\F}\}$. $\sigma^\F$ is then defined by 
    $$\sigma^\F (f_m) = \sigma(\alpha_m),\ m = 1, \ldots, r^\F.$$
    Since $j\in A^\F$, there exists $p, q\in [1, r^\F]$ such that 
    $$j= f_p = \sigma(\alpha_q). $$
    Since $j = \sigma(i)$, it is necessary that 
    $$\alpha_q = i,\ 1 \leq \alpha_1 < \cdots < \alpha_{q-1} < i.$$
    In particular, 
    $$r^\F_\sigma(i) = \#\{1\leq t \leq i: \sigma(t)\in A^\F\} = q.$$
    Without loss of generality, let's assume that $q < p$. Then $\sigma^\F$ is the following permutation
    \[\begin{pmatrix}
        f_1 & \cdots & f_{q-1} & f_q & f_{q+1} & \cdots & f_{p-1} & f_p = j & f_{p+1} & \cdots & f_k\\
        \sigma(\alpha_1) & \cdots & \sigma(\alpha_{q-1}) & \sigma(\alpha_q) = j & \sigma(\alpha_{q+1}) & \cdots & \sigma(\alpha_{p-1}) & \sigma(\alpha_p) & \sigma(\alpha_{p+1}) & \cdots & \sigma(\alpha_k).
    \end{pmatrix}\]
    We now analyze $\tau^\F$. Since for every $m<i$, $\tau(m) = \sigma(m)$, we know that $$\tau^\F(f_n) = \sigma(\alpha_n),\ 1\leq n \leq q-1,$$ 
    and also for every $i \leq m < j$, $\tau(m) = \sigma(m+1)$, so $\tau(m)\in A^\F$ means $\sigma(m+1)\in A^\F$, which implies that $m+1 \in \{\alpha_{q+1}, \alpha_{q+2}, \ldots, \alpha_{p-1}\}$. Thus 
    $$\tau^\F(f_n) = \sigma(\alpha_{n+1}),\ q \leq n \leq p-1.$$
    Note that $j = f_p$ is not in the domain of $\tau$, thus not in the domain of $\tau^\F$. After skipping $f_p$, we have 
    $$\tau^\F(f_{n}) = \sigma(\alpha_n),\ p+1 \leq n \leq k.$$
    Now consider $(\tau^e)^\F$ of the extension $\tau^e$. Clearly, 
    $$(-1)^{(\tau^e)^\F} = (-1)^{\tau^\F}.$$
    With the analysis above, $(\tau^e)^\F$ is the permutation 
    $$\begin{pmatrix}
        f_1 & \cdots & f_{q-1} & f_q & f_{q+1} & \cdots & f_{p-1} & f_p=j & f_{p+1} & \cdots & f_k\\
        \sigma(\alpha_1) & \cdots & \sigma(\alpha_{q-1}) & \sigma(\alpha_{q+1}) & \sigma(\alpha_{q+2}) & \cdots & \sigma(\alpha_p) & \sigma(\alpha_q)=j & \sigma(\alpha_{p+1}) & \cdots & \sigma(\alpha_k)
    \end{pmatrix}.$$
    Then clearly, $\sigma^\F$ and $(\tau^e)^\F$ differs by a permutation, namely, 
    $$ (\tau^e)^\F = \sigma^\F (f_qf_{q+1} \ldots f_{p}).$$
    In particular, 
    $$(-1)^{(\tau^e)^\F} = (-1)^{\sigma^\F} (-1)^{q-p}.$$
    Recall that $r^\F_\sigma(i)= q$ and $r^\F(j) = p$. So 
    $$(-1)^{\sigma^\F}(-1)^{r^\F_\sigma(i)} = (-1)^{\tau^\F} (-1)^{r^\F(j)}.$$
    When $p=q$ and $p>q$, we obtain the same identity with a similar analysis. We shall not go over the details here. 

    Therefore, the left-hand side of (\ref{Comb-Id-1}) may be expressed as
    \begin{align*}
    & \sum_{j\in A^\F} \sum_{\substack{\tau\in J_{\mu-1, \nu}(1, \dots, \widehat{j},\dots, r)\\ \tau(1)>j} }(-1)^{\tau^\F}(-1)^{r^\F(j)} \Phi(j) \Psi(\widehat{j}, \tau(1), \dots,
    \tau(\mu-1), \tau(\mu), \dots, \tau(r))\\
    &\quad\quad +\cdots +\sum_{j\in A^\F}\sum_{\substack{\tau\in J_{\mu-1, \nu}(1, \dots, \widehat{j},\dots, r)\\ \tau(i-1)<j<\tau(i)}}  (-1)^{\tau^\F}(-1)^{r^\F(j)}\nonumber \\
&\quad\quad\quad\quad\quad\quad\quad\quad\quad\quad\quad \cdot 
\Phi(j) \Psi(\tau(1), \dots, \widehat{j}, \dots,
    \tau(\mu-1), \tau(\mu), \dots, \tau(r))\\
    &\quad\quad +\cdots +\sum_{j=1}^r\sum_{\substack{\tau\in J_{\mu-1, \nu}(1, \dots, \widehat{j},\dots, r)\\ \tau(\mu-1)<j}}  (-1)^{\tau^\F}(-1)^{r^\F(j)}\Phi(j)\Psi(\tau(1), \dots, \tau(\mu-1), \widehat{j}, \tau(\mu), \dots, \tau(r))\\
    &\quad= \sum_{j\in A^\F}(-1)^{r^\F(j)}\sum_{\tau\in J_{\mu-1, \nu}(1, \dots, \widehat{j},\dots, r)}  (-1)^{\tau^\F} \Phi(j)\Psi(\tau(1), \dots,
    \tau(\mu-1), \tau(\mu), \dots, \tau(r)),
    \end{align*}
    which is exact the right-hand side after the modification of the nonexistential $\reallywidehat{\tau(j)}$. 
\end{proof}

\begin{rmk}
    While the arguments for (\ref{Comb-Id-2}) and (\ref{Comb-Id-3}) are similar, the parity analysis is more complicated. To obtain the identity $(-1)^{\sigma^\F} (-1)^{r^\F_\sigma(i)-r^\F_\sigma(\mu)} = (-1)^{r^\F(j)}(-1)^{\tau^\F}$ in the context of (\ref{Comb-Id-2}), one has to consider all possible arrangements of the positions of $i, j, \mu$ and $i-\mu$ (there are at most $\binom{5}2+5=15$ cases). To obtain the identity $(-1)^{\sigma^\F} (-1)^{r^\F_\sigma(i)-r^\F_\sigma(\mu+\nu)} = (-1)^{r^\F(j)}(-1)^{\tau^\F}$ in the context of (\ref{Comb-Id-3}), one has to consider all possible arrangements of the positions of $i, j, \mu, \nu$ and $i-\mu-\nu$ (there are at most $\binom{7}2+7 = 28$ cases). We shall not elaborate the details here. 
%   \textcolor{red}{I forgot why there are so many cases. Hopefully I can find it in the archive.}
\end{rmk}

\begin{lemma} Let $r, \mu\in \Z_+$ with $\mu\leq r-1$. Then for any algebraic expression \(\Phi\) with one argument and any algebraic expression $\Psi$ with $r-1$ arguments,
\begin{align}
    & \sum_{\sigma\in J_\mu(1, \dots, r)} \sum_{i\in A_\sigma^\F} (-1)^{r^\F_\sigma(i)-1}(-1)^{\sigma^\F} \Phi(\sigma(i))\Psi(\sigma(1), \dots, \reallywidehat{\sigma(i)}, \dots, \sigma(\mu), \sigma(\mu+1), \dots, \sigma(r))
    \nonumber\\
    &\quad = \sum_{j\in A^\F} (-1)^{r^\F(j)-1} \Phi(j) \sum_{\tau\in J_{\mu-1}(1, \dots, \widehat{j}, \dots, r)} (-1)^{\tau^\F} \Psi(\tau(1), \dots, \reallywidehat{\tau(j)},  \dots, \tau(r)), \label{Comb-Id-4}\\
    & \sum_{\sigma\in J_\mu(1, \dots, r)} \sum_{i\in A_\sigma^\F} (-1)^{r^\F(\mu+i)-1}(-1)^{\sigma^\F} \Phi(\sigma(i)) \Psi(\sigma(1), \dots, \reallywidehat{\sigma(i)}, \dots, \sigma(\mu), \sigma(\mu+1), \dots, \sigma(r))
    \nonumber\\
    & \quad = \sum_{j\in A^\F} (-1)^{r^\F(j)-1} \Phi(j) \sum_{\tau\in J_\mu(1, \dots, \widehat{j}, \dots, r)} (-1)^{\tau^\F} \Psi(\tau(1), \dots, \reallywidehat{\tau(j)}, \dots, \tau(r)). \label{Comb-Id-5}
\end{align}
\end{lemma}

\begin{proof}
    Both identities follow from a straightforward modification of the proof of (\ref{Comb-Id-2}). 
\end{proof}

\begin{lemma}\label{Comb-Id-Lemma-B} 
Let $r,\mu,\nu\in \Z_+$ with $\mu+\nu\leq r-1$. Then for any algebraic expression \(\Phi\) with one argument and any algebraic expression $\Psi$ with $r-1$ arguments,
\begin{align}
    & \sum_{\sigma\in J_{\mu,\nu}(1, \dots, r)} \sum_{i\in A_\sigma^\B[1,\mu]} (-1)^{\sigma^\F} \Phi(\sigma(i)) \Psi(\sigma(1), \dots, \reallywidehat{\sigma(i)}, \dots,
    \sigma(\mu), \sigma(\mu+1), \dots, \sigma(r)) \nonumber\\
    & \quad= \sum_{j\in A^\B}  \Phi(j)  \sum_{\tau \in J_{\mu-1, \nu}(1,\dots,\widehat{j},\dots,r)} (-1)^{\tau^\F} \Psi(\tau(1), \dots, \reallywidehat{\tau(j)}, \dots, \tau(r)),\label{Comb-Id-1-B}\\
    & \sum_{\sigma\in J_{\mu, \nu}(1, \dots, r)} \sum_{i \in A_\sigma^\B[\mu+1, \mu+\nu]} (-1)^{\sigma^\F} \Phi(\sigma(i)) \Psi(\sigma(1), \dots, \sigma(\mu+1), \dots, \reallywidehat{\sigma(i)},\dots,\sigma(\mu+\nu) \dots,
    \sigma(r)) \nonumber\\
    & \quad= \sum_{j\in A^\B} \Phi(j) \sum_{\tau \in J_{\mu, \nu-1}(1,\dots,\widehat{j},\dots,r)} (-1)^{\tau^\F} \Psi(\tau(1), \dots, \reallywidehat{\tau(j)}, \dots, \tau(r)), \label{Comb-Id-2-B}\\
   & \sum_{\sigma\in J_{\mu, \nu}(1, \dots, r)} \sum_{i\in A_\sigma^\B [\mu+\nu+1,r]} (-1)^{\sigma^\F} \Phi(\sigma(i))\Psi(\sigma(1), \dots, 
    \sigma(\mu+\nu+1), \dots, \reallywidehat{\sigma(i)}, \dots,\sigma(r)) \nonumber\\
    &\quad = \sum_{j\in A^\B} \Phi(j) \sum_{\tau \in J_{\mu, \nu}(1,\dots,\widehat{j},\dots,r)} (-1)^{\tau^\F} \Psi(\tau(1),\dots,  \reallywidehat{\tau(j)}, \dots, \tau(r)). \label{Comb-Id-3-B}
\end{align}
\end{lemma}

\begin{proof}
Proceed with the same procedure as in the proof of Lemma \ref{Comb-Id-Lemma}. The only necessary modification is that since $j\in A^\B$, the ordering $\sigma(f_1), \dots, \sigma(f_{r^\F})$ is the same as $\tau(f_1), \dots, \tau(f_{r^\F})$. Therefore, 
$\sigma^\F = \tau^\F$, So the parity analysis is not necessary. The rest of the procedure is almost the same with straightforward modifications. 
\end{proof}
We also clearly have
\begin{lemma} Let $r, \mu\in \Z_+$ with $\mu\leq r-1$. Then for any algebraic expression \(\Phi\) with one argument and any algebraic expression $\Psi$ with $r-1$ arguments,
\begin{align}
    & \sum_{\sigma\in J_\mu(1, \dots, r)} \sum_{i\in A_\sigma^\B} (-1)^{\sigma^\F} \Phi(\sigma(i)) \Psi(\sigma(1), \dots, \reallywidehat{\sigma(i)}, \dots, \sigma(\mu), \sigma(\mu+1), \dots, \sigma(r))
    \nonumber\\
    &\quad = \sum_{j\in A^\B} \Phi(j) \sum_{\tau\in J_{\mu-1}(1, \dots, \widehat{j}, \dots, r)} (-1)^{\tau^\F} \Psi(\tau(1), \dots, \reallywidehat{\tau(j)},  \dots, \tau(r)), \label{Comb-Id-4-B}\\
    & \sum_{\sigma\in J_\mu(1, \dots, r)} \sum_{i\in A_\sigma^\B} (-1)^{\sigma^\F} \Phi(\sigma(i)) \Psi(\sigma(1), \dots, \reallywidehat{\sigma(i)}, \dots, \sigma(\mu), \sigma(\mu+1), \dots, \sigma(r))
    \nonumber\\
    & \quad= \sum_{j\in A^\B} \Phi(j) \sum_{\tau\in J_\mu(1, \dots, \widehat{j}, \dots, r)} (-1)^{\tau^\F} \Psi(\tau(1), \dots, \reallywidehat{\tau(j)}, \dots, \tau(r)). \label{Comb-Id-5-B}
\end{align}
\end{lemma}

\section{Proof of Theorem \ref{It-is-a-mosva}}\label{Proof-it-is-a-mosva}

We prove Theorem \ref{It-is-a-mosva} in this section. We need to show that 
 $(V, Y, \one)$ given in Section 3 satisfies the axioms of a 
 $\frac{\Z}{2}$-graded meromorphic open-string vertex algebra. 

The axioms for the grading, axioms for the vacuum, $D$-derivative property and $D$-commutator formula 
may be proved by a straightforward modification of the arguments in Section 5 of 
\cite{H-mosva} and  Section 3 of \cite{FQ-Fermion-1}. We shall omit the details. 
From Remark \ref{equiv-defn-mosva}, to prove the rationality and associativity, we need only prove 
the weak associativity with pole-order condition, which is the main technical part of the proof of 
Theorem \ref{It-is-a-mosva}.

\subsection{A generalized Wick's Theorem for the normal ordering $:\cdot:$ on $V$} \label{sec-Wick-thm-simple-case}

As in \cite{H-mosva} and \cite{FQ-Fermion-1}, to prove the weak associativity, we need a formula to 
express the product of two normal-ordered products of derivatives of the generating fields of the form 
$a(x)$, as a linear combination of the normal-ordered products of some of these generating fields. 
The formulas obtained in  \cite{H-mosva}, \cite{FQ-Fermion-1}, and this section can in fact be viewed as 
generalizations of Wick's theorem (see \cite{W}).  We shall call these results \textit{generalized Wick's theorems}.
The generalized Wick's theorem in \cite{H-mosva}, \cite{FQ-Fermion-1}, and this paper
do not need to assume any relations among creation operators,  among zero mode operators,
and among annihilation operators, though it has been proved that annihilation operators
commute or anti-commute with each other. 
 
We first formulate and prove some simple versions of our general generalized Wick's theorem. For brevity, for $a\in \h^\B\cup \h^\F$ and $m\in \Z_+$, $a^{(m-1)}(x)$ 
means the generating function defined in Section \ref{gen-func-alg}. 

We start with the following recursion of normal-ordered product. 
\begin{prop}\label{Nord-recursion-left-boson-prop}
    For $a\in \h^\B$, $b_1, \ldots, b_s\in \h^\B \cup \h^\F$, $m, n_1, \dots, n_s\in \Z_+$,  
    \begin{align*}
        & a^{(m-1)}(x) :b_1^{(n_1-1)}(y_1)\cdots b_s^{(n_s-1)}(y_s): - :a^{(m-1)}(x)b_1^{(n_1-1)}(y_1)\cdots b_s^{(n_s-1)}(y_s):\\
        &\quad =    \sum_{j\in A^\B} (a, b_j) \ell_\B F^\B_{mn_j}(x,y_j) :b_1^{(n_1-1)}(y_1)\cdots \reallywidehat{b_j^{(n_j-1)}(y_i)}\cdots b_s^{(n_s-1)}(y_s):. 
    \end{align*}
\end{prop}
\begin{proof} It suffices to prove the identity for $m = n_1 = \cdots = n_s = 1$ (cf. Remark \ref{Nord-commutes-w-derivation-rmk}). 
    Write $a(x) = a(x)^- + a(x)^0 + a(x)^+$. Note that since 
    \begin{align*}
        & a(x)^- :b_1(y_1)\cdots b_s(y_s): = :a(x)^- b_1(y_1)\cdots b_s(y_s):\\
        & a(x)^0 :b_1(y_1)\cdots b_s(y_s): = :a(x)^0 b_1(y_1)\cdots b_s(y_s):, 
    \end{align*}
    it suffices to handle the difference with $a(x)^+$:
    \begin{align*}
        & a(x)^+ :b_1(y_1)\cdots b_s(y_s): - :a(x)^+ b_1(y_1)\cdots b_s(y_s):\\
       &\quad = \sum_{\mu=0}^s \sum_{\nu = 0}^{s-\mu} \sum_{\sigma\in J_{\mu\nu}(1, \ldots, s)} (-1)^{\sigma^\F} a(x)^+ b_{\sigma(1)}(y_{\sigma(1)})^- \cdots b_{\sigma(\mu)}(y_{\sigma(\mu)})^-\\ 
& \hspace{5em} \cdot b_{\sigma(\mu+1)}(y_{\sigma(\mu+1)})^+ \cdots b_{\sigma(\mu+\nu)}(y_{\sigma(\mu+\nu)})^+ 
         b_{\sigma(\mu+\nu+1)}(y_{\sigma(\mu+\nu+1)})^0 \cdots b_{\sigma(s)}(y_{\sigma(s)})^0 \\
        &\quad\quad  - \sum_{\mu=0}^s \sum_{\nu = 0}^{s-\mu} \sum_{\sigma\in J_{\mu\nu}(1, \ldots, s)} (-1)^{\sigma^\F}  b_{\sigma(1)}(y_{\sigma(1)})^- \cdots b_{\sigma(\mu)}(y_{\sigma(\mu)})^- a(x)^+\\ 
& \hspace{5em} \cdot b_{\sigma(\mu+1)}(y_{\sigma(\mu+1)})^+ \cdots b_{\sigma(\mu+\nu)}(y_{\sigma(\mu+\nu)})^+
b_{\sigma(\mu+\nu+1)}(y_{\sigma(\mu+\nu+1)})^0 \cdots b_{\sigma(s)}(y_{\sigma(s)})^0 \\
    &\quad    =  \sum_{\mu=1}^s \sum_{\nu = 0}^{s-\mu} \sum_{\sigma\in J_{\mu\nu}(1, \ldots, s)} (-1)^{\sigma^\F} \sum_{i=1}^{\mu} b_{\sigma(1)}(y_{\sigma(1)})^- \cdots [a(x)^+, b_{\sigma(i)}(y_{\sigma(i)})^-]\cdots b_{\sigma(\mu)}(y_{\sigma(\mu)})^- \\ 
& \hspace{5em} \cdot b_{\sigma(\mu+1)}(y_{\sigma(\mu+1)})^+ \cdots b_{\sigma(\mu+\nu)}(y_{\sigma(\mu+\nu)})^+ b_{\sigma(\mu+\nu+1)}(y_{\sigma(\mu+\nu+1)})^0 \cdots b_{\sigma(s)}(y_{\sigma(s)})^0. 
    \end{align*}
    Note that for each $i = 1, \ldots, \mu$, 
    $$[a(x)^+, b_{\sigma(i)}(y_{\sigma(i)})] = \begin{cases}
        0 & \text{ if } b_{\sigma(i)\in \h^\F}, \text{ i.e. } i \in B^\F_\sigma\\
        \left(a, b_{\sigma(i)}\right) \ell_\B F^\B(x, y_{\sigma(i)})  & \text{ if } b_{\sigma(i)\in \h^\B}, \text{ i.e. } i\in B^\B_\sigma. 
    \end{cases}$$
    Then the summand for each fixed $\mu, \nu$ is equal to  
    \begin{align}
        & \sum_{\sigma\in J_{\mu\nu}(1, \ldots, s)}(-1)^{\sigma^\F} \sum_{i\in B_\sigma^\B[1, \mu]}\left(a, b_{\sigma(i)}\right) \ell_\B F^\B(x, y_{\sigma(i)})\nonumber \\
        & \hspace{11.5em} \cdot b_{\sigma(1)}(y_{\sigma(1)}) \cdots \reallywidehat{b_{\sigma(i)}(y_i)^-} \cdots b_{\sigma(\mu)}(y_{\sigma(\mu)})^-\nonumber \\
        & \hspace{11.5em} \cdot b_{\sigma(\mu+1)}(y_{\sigma(\mu+1)})^+ \cdots b_{\sigma(\mu+\nu)}(y_{\sigma(\mu+\nu)})^+ \nonumber\\
        & \hspace{11.5em} \cdot b_{\sigma(\mu+\nu+1)}(y_{\sigma(\mu+\nu+1)})^0 \cdots b_{\sigma(s)}(y_{\sigma(s)})^0. \nonumber
    \end{align}
    Using Lemma \ref{Comb-Id-Lemma}, we change the order of summation, to see that the summand is equal to
    \begin{align*}
        & \sum_{j\in B^\B} (a, b_j) \ell_\B  F^\B(x, y_j) \sum_{\tau \in J_{\mu-1,\nu}(1, \ldots, \widehat{j}, \ldots, s)}(-1)^{\tau_\F}b_{\tau(1)}(y_{\tau(1)})^- \cdots b_{\tau(\mu-1)}(y_{\tau(\mu-1)})^-\nonumber \\
        & \hspace{17em} \cdot b_{\tau(\mu)}(y_{\tau(\mu)})^+ \cdots b_{\tau(\mu-1+\nu)}(y_{\tau(\mu-1+\nu)})^+\\
        & \hspace{17em} \cdot b_{\tau(\mu+\nu)}(y_{\tau(\mu+\nu)})^0 \cdots b_{\tau(s)}(y_{\tau(s)})^0 \ .
    \end{align*}
    Summing up $\mu$ and $\nu$, we obtain 
    $$\sum_{j\in B^\B} (a, b_j) \ell_\B F^\B(x, y_j) :b_1(y_1) \cdots \reallywidehat{b_j(y_j)} \cdots b_s(y_s):.$$
\end{proof}

\begin{prop}\label{Nord-recursion-left-fermion-prop}
    For $a\in \h^\F, b_1, \ldots, b_s \in \h^\B \cup \h^\F$, $m, n_1, \dots, n_s\in \Z_+$, 
    \begin{align*}
        & a^{(m-1)}(x) :b_1^{(n_1-1)}(y_1) \cdots b_s^{(n_s-1)}(y_s): - :a^{(m-1)}(x) b_1^{(n_1-1)}(y_1)\cdots b_s^{(n_s-1)}(y_s):\\
        &\quad=  \sum_{j\in B^\F} (-1)^{s^\F(j)-1}(a, b_j) \ell_\F F_{mn_j}^\F(x, y_j) :b_1^{(n_1-1)}(y_1) \cdots \reallywidehat{b_j^{(n_j-1)}(y_j)}\cdots b_s^{(n_s-1)}(y_s): . 
    \end{align*}
\end{prop}

\begin{proof}
    Write $a(x) = a(x)^- + a(x)^+$. Note that 
    \begin{align*}
        & a(x)^- :b_1(y_1)\cdots b_s(y_s): = :a(x)^- b_1(y_1)\cdots b_s(y_s): \ , 
    \end{align*}
    it suffices to handle the difference with $a(x)^+$:    
    \begin{align*}
        & a(x)^+ :b_1(y_1)\cdots b_s(y_s) - :a(x)^+ b_1(y_1)\cdots b_s(y_s):\\
    &\quad     = \sum_{\mu=0}^s \sum_{\nu = 0}^{s-\mu} \sum_{\sigma\in J_{\mu\nu}(1, \ldots, s)} (-1)^{\sigma^\F} a(x)^+ b_{\sigma(1)}(y_{\sigma(1)})^- \cdots b_{\sigma(\mu)}(y_{\sigma(\mu)})^-\\
        & \hspace{10em}\cdot  b_{\sigma(\mu+1)}(y_{\sigma(\mu+1)})^+ \cdots b_{\sigma(\mu+\nu)}(y_{\sigma(\mu+\nu)})^+ \\
        & \hspace{10em} \cdot b_{\sigma(\mu+\nu+1)}(y_{\sigma(\mu+\nu+1)})^0 \cdots b_{\sigma(s)}(y_{\sigma(s)})^0 \\
        & \quad\quad - \sum_{\mu=0}^s \sum_{\nu = 0}^{s-\mu} \sum_{\sigma\in J_{\mu\nu}(1, \ldots, s)} (-1)^{\sigma^\F}(-1)^{s_\sigma^\F(\mu)}  b_{\sigma(1)}(y_{\sigma(1)})^- \cdots b_{\sigma(\mu)}(y_{\sigma(\mu)})^-  \\
        & \hspace{10em} \cdot a(x)^+b_{\sigma(\mu+1)}(y_{\sigma(\mu+1)})^+ \cdots b_{\sigma(\mu+\nu)}(y_{\sigma(\mu+\nu)})^+ \\ 
        & \hspace{10em} \cdot b_{\sigma(\mu+\nu+1)}(y_{\sigma(\mu+\nu+1)})^0 \cdots b_{\sigma(s)}(y_{\sigma(s)})^0 \\
       &\quad  =  \sum_{\mu=1}^s \sum_{\nu = 0}^{s-\mu} \sum_{\sigma\in J_{\mu\nu}(1, \ldots, s)} (-1)^{\sigma^\F} \sum_{i=1}^{\mu} (-1)^{s_\sigma^\F(i)-1}\\
        & \hspace{10em} \cdot  b_{\sigma(1)}(y_{\sigma(1)})^- \cdots [a(x)^+, b_{\sigma(i)}(y_{\sigma(i)})^-]\cdots b_{\sigma(\mu)}(y_{\sigma(\mu)})^- \\
        & \hspace{10em} \cdot b_{\sigma(\mu+1)}(y_{\sigma(\mu+1)})^+ \cdots b_{\sigma(\mu+\nu)}(y_{\sigma(\mu+\nu)})^+\\
        & \hspace{10em} \cdot b_{\sigma(\mu+\nu+1)}(y_{\sigma(\mu+\nu+1)})^0 \cdots b_{\sigma(s)}(y_{\sigma(s)})^0. 
    \end{align*}
    Note that for each $i = 1, \ldots, \mu$, 
    $$[a(x)^+, b_{\sigma(i)}(y_{\sigma(i)})] = \begin{cases}
        0 & \text{ if } b_{\sigma(i)} \in \h^\B ,\\
        \left(a, b_{\sigma(i)}\right) \ell_\F F^\F(x, y_{\sigma(i)})  & \text{ if } b_{\sigma(i)}\in \h^\F. 
    \end{cases}$$
    Then the summand for each fixed $\mu, \nu$ is equal to  
    \begin{align}
        & \sum_{\sigma\in J_{\mu\nu}(1, \ldots, s)}(-1)^{\sigma^\F} \sum_{i\in B^\F_\sigma[1,\mu]}(-1)^{s^\F_\sigma(i)-1}\left(a, b_{\sigma(i)}\right) \ell_\F F^\F(x, y_{\sigma(i)}) \nonumber \\
        & \hspace{11.5em} \cdot b_{\sigma(1)}(y_{\sigma(1)}) \cdots \reallywidehat{b_{\sigma(i)}(y_i)^-} \cdots b_{\sigma(\mu)}(y_{\sigma(\mu)})^-\nonumber \\
        & \hspace{11.5em} \cdot b_{\sigma(\mu+1)}(y_{\sigma(\mu+1)})^+ \cdots b_{\sigma(\mu+\nu)}(y_{\sigma(\mu+\nu)})^+ \nonumber\\
        & \hspace{11.5em} \cdot b_{\sigma(\mu+\nu+1)}(y_{\sigma(\mu+\nu+1)})^0 \cdots b_{\sigma(s)}(y_{\sigma(s)})^0. \nonumber
    \end{align}
    Using Lemma \ref{Comb-Id-Lemma}, we change the order of summation, to see that the summand is equal to
    \begin{align*}
        & \sum_{j\in B^\F} (-1)^{s^\F(j)-1}(a, b_j) \ell_\F  F^\F(x, y_j) \sum_{\tau \in J_{\mu-1,\nu}(1, \ldots, \widehat{j}, \ldots, s)}(-1)^{\tau_\F}b_{\tau(1)}(y_{\tau(1)})^- \cdots b_{\tau(\mu-1)}(y_{\tau(\mu-1)})^-\nonumber \\
        & \hspace{22.2em} \cdot b_{\tau(\mu)}(y_{\tau(\mu)})^+ \cdots b_{\tau(\mu-1+\nu)}(y_{\tau(\mu-1+\nu)})^+ \\
        & \hspace{22.2em} \cdot b_{\tau(\mu+\nu)}(y_{\tau(\mu+\nu)})^0 \cdots b_{\tau(s)}(y_{\tau(s)}).^0
    \end{align*}
    Summing up $\mu$ and $\nu$, we obtain 
    $$\sum_{j\in B^\B} (a, b_j) (-1)^{s^\F(j)-1}\ell_\F F^\F(x, y_j) :b_1(y_1) \cdots \reallywidehat{b_j(y_j)} \cdots b_s(y_s):.$$

\end{proof}

\begin{prop}\label{Nord-recursion-right-boson-prop}
    For $a_1, \ldots, a_r\in \h^\B \cup \h^\F$, $b\in \h_\B$, $m_1, \dots, m_r, n\in \Z_+$, 
    \begin{align*}
        & :a_1^{(m_1-1)}(x_1)\cdots a_r^{(m_r-1)}(x_r):b^{(n-1)}(y) - :a_1^{(m_1-1)}(x_1)\cdots a_r^{(m_r-1)}(x_r) b^{(n-1)}(y): \\
        &\quad = \sum_{\substack{i\in A^\B}} (a_i, b)\ell_\B F^\B_{m_i n}(x_i, y) :a_1^{(m_1-1)}(x_1) \cdots \reallywidehat{a_i^{(m_i-1)}(x_i)} \cdots a_r^{(m_r-1)}(x_r): . 
    \end{align*}
\end{prop}

\begin{proof}
    One may either proceed with shuffles as in Proposition \ref{Nord-recursion-left-boson-prop}, or apply induction similarly to the proof of Proposition 4.2 in \cite{FQ-Fermion-1}. We shall omit the details for brevity. 
\end{proof}

\begin{prop}\label{Nord-recursion-right-fermion-prop}
    For $a_1, \ldots, a_r\in \h^\B \cup \h^\F$, $b\in \h^\F$, $m_1, \dots, m_r, n\in \Z_+$, 
    \begin{align*}
        & :a_1^{(m_1-1)}(x_1)\cdots a_r^{(m_r-1)}(x_r):b^{(n-1)}(y) - :a_1^{(m_1-1)}(x_1)\cdots a_r^{(m_r-1)}(x_r) b^{(n-1)}(y): \\
        %= \ & \sum_{1\leq p \leq k}(-1)^{k-p} (a_{\alpha_p}, b)\ell_\F F(x_{\alpha_p}, y) :a_1(x_1) \cdots \reallywidehat{a_{\alpha_p}(x_{\alpha_p})} \cdots a_r(x_r):\\
     &\quad   =  \sum_{\substack{j\in A^\F}}(-1)^{r^\F-r^\F(j)} (a_j, b)\ell_\F F^\F_{m_jn}(x_j, y) :a_1^{(m_1-1)}(x_1) \cdots \reallywidehat{a_{j}^{(m_j-1)}(x_{j})} \cdots a_r^{(m_r-1)}(x_r): . 
    \end{align*}
\end{prop}

\begin{proof}
Same as Proposition \ref{Nord-recursion-right-boson-prop}. 
\end{proof}

%We introduce one more proposition about normal-ordered products. 
%\begin{enumerate}
 %       \item If $a\in \h^\B$, then 
  %      \begin{align*}
   %     :a^{(m-1)}(x) b_1^{(n_1-1)}{y_1} \cdots b_s^{(n_s-1)}(y_s): = &\left(a^{(m-1)}(x)^- +a^{(m-1)}(x)^0\right) : b_1^{(n_1-1)}{y_1} \cdots b_s^{(n_s-1)}(y_s) :\\
    %    &+ : b_1^{(n_1-1)}{y_1} \cdots b_s^{(n_s-1)}(y_s) : a^{(m-1)}(x)^+;
     %   \end{align*}
      %  \item If $a\in \h^\F$, then 
       % \begin{align*}
        %:a^{(m-1)}(x) b_1^{(n_1-1)}{y_1} \cdots b_s^{(n_s-1)}(y_s): = & a^{(m-1)}(x)^-  : b_1^{(n_1-1)}{y_1} \cdots b_s^{(n_s-1)}(y_s) :\\
        %&+ (-1)^{r^\F-1}: b_1^{(n_1-1)}{y_1} \cdots b_s^{(n_s-1)}(y_s) : a^{(m-1)}(x)^+.
        %\end{align*}
   % \end{enumerate}
%\end{prop}

%\begin{proof}
   % The proof follows directly from the definition of the normal ordering \(: \cdot :\). Just note that in Case (2), for \(a^{(m-1)}(x)^+\) to pass to the rightmost, it has to anti-commute with \(r^\F-1\) positive fermionic fields (itself excluded), hence the sign \((-1)^{r^\F-1}\). 
%\end{proof}

%begin{rmk}
 %   Proposition \ref{Nord-recursion-left-right-prop} does not really touch on the \(\Z_2\)-graded commutators between positive and negative modes, hence essentially much simpler than Propositions \ref{Nord-recursion-left-boson-prop} to \ref{Nord-recursion-right-fermion-prop}. We will need it only in Subsection \ref{subsec-properties-of-exp}. 
%\end{rmk}

\begin{nota}
We use \(\mat(a_{ij})_{i,j=1}^n\) to denote an \((n\times n)\)-matrix whose \((i,j)\)-th component is \(a_{ij}\). Then we denote the permanant of a matrix \(A\) by \(\perm A\), and the determinant of \(A\) by \(\det A\). We adopt the convention that the permanent and the determinant of the empty matrix is \(1\). 
\end{nota}

\begin{nota}\label{Index-Notation-1}
To avoid the clumsy iterated subscripts, for $\eta, \eta'\in \Z_+, I^b=(i_{1}^{b}, \dots, i_{\eta}^{b})
 \in A^\B_\eta, J^b=(j_{1}^{b}, \dots, j_{\eta'}^{b})\in B^\B_{\eta'}, \alpha = 1, \dots, \eta, \beta = 1, \dots, \eta'$, we use the notation 
$$\langle i^b_\alpha, j^b_\beta\rangle_\B = \ell_\B (a_{i^b_\alpha}, b_{j^b_\beta})F^\B(x_{i^b_\alpha}, y_{j^b_\beta}).$$
Similarly, for $\eta, \eta'\in \Z_+, I^f=(i_{1}^{f}, \dots, i_{\eta}^{f})\in A^\F_\eta, J^f=
(j_{1}^{f}, \dots, j_{\eta'}^{f})\in B^\F_{\eta'}, \alpha = 1, \dots, \eta, \beta = 1, \dots, \eta'$, we use the notation 
$$\langle i^f_\alpha, j^f_\beta\rangle_\F = \ell_\F (a_{i^f_\alpha}, b_{j^f_\beta})F^\F(x_{i^f_\alpha}, y_{j^f_\beta}).$$
In case it is clear from the contexts, we may also use the notations
$$\langle p,q\rangle_\B = \begin{cases}\ell_\B (a_p, a_q)F^\B(x_p, x_q) & \text{ if }p, q \in A^\B ,\\
    \ell_\B (a_p, a_q)F^\B(y_p, y_q) & \text{ if }p, q \in B^\B ,\\
\end{cases}$$ 
and
$$\langle p,q\rangle_\F = \begin{cases}\ell_\F (a_p, a_q)F^\F(x_p, x_q) & \text{ if }p, q \in A^\F ,\\
    \ell_\F (a_p, a_q)F^\F(y_p, y_q) & \text{ if }p, q \in B^\F .\\
\end{cases}$$
We will only use these notations in the related computations in the proofs. We will not use them in the statements of definitions, lemmas, propositions, and theorems.
\end{nota}

We now state and prove our generalized Wick's theorem. 

\begin{thm}\label{Wick-thm-general-case}
    Let $r, s\in \Z_+$, $a_1, \dots, a_r, b_1, \dots, b_s\in \h^\B \cup \h^\F$, $m_1, \dots, m_r, n_1, \dots, n_s\in \Z_+$. Then, 
    \begin{align}
     & :a_1^{(m_1-1)}(x_1) \cdots a_r^{(m_r-1)}(x_r): \cdot :b_1^{(n_1-1)}(y_1) \cdots b_s^{(n_s-1)}(y_s):  \nonumber\\
 &\quad 
    = \sum_{\rho, \eta=0}^{\min(r,s)} \sum_{I^b\in A^\B_\eta, J^b\in B^\B_\eta}  \perm \mat \left(\ell_\B\left(a_{i_\alpha^b},\, b_{j_\beta^b}\right) F^\B_{m_{i_\alpha^b}n_{i_\beta^b}}\left(x_{i_{\alpha}^b},\, y_{j_\beta^b}\right)\right)_{\alpha, \beta=1}^\eta \nonumber\\
     & \hspace{3em} \cdot\sum_{I^f\in A^\F_\rho, J^f\in B^\F_\rho} (-1)^{r^\F(i_1^f)+\cdots +r^\F(i_\rho^f)+s^\F(j_1^f)+\cdots + s^\F(j_\rho^f)}(-1)^{r^\F\rho + \frac{\rho(\rho+1)}{2}} \nonumber\\
     &\hspace{8.5em} \cdot \det \mat \left(\ell_\F\left(a_{i_\gamma^f},\, b_{j_\delta^f}\right) F^\F_{m_{i_\delta^f}n_{j_\delta^f}}\left(x_{i_\gamma^f},\, y_{j_\delta^f}\right)\right)_{\gamma, \delta=1}^\rho \nonumber\\
     &\hspace{8.8em}\cdot : a_1^{(m_1-1)}(x_1) \cdots a_r^{(m_r-1)}(x_r) b_1^{(n_1-1)}(y_1) \cdots b_s^{(n_s-1)}(y_s):_{O(I^b, I^f, J^b, J^f)} \label{Eq-Wick-thm-general-case}
    \end{align}
\end{thm}

\begin{proof}
    From Remark \ref{Nord-commutes-w-derivation-rmk}, It suffices to prove the case where $m_1 = \cdots = m_r = n_1 = \cdots = n_s = 1$. We prove (\ref{Eq-Wick-thm-general-case}) by induction. 
    
    By Proposition \ref{Nord-recursion-left-boson-prop} and Proposition \ref{Nord-recursion-left-fermion-prop}, the conclusion holds for $r=1$ and for every $s\in \Z_+$. Assume the conclusion holds for smaller $r$. We discuss the following cases. 

\noindent\textbf{Case 1. $a_r \in \h^\F.$} By Proposition \ref{Nord-recursion-left-fermion-prop} and Proposition \ref{Nord-recursion-right-fermion-prop}, we have 
\begin{align}
    &:a_1(x_1)\cdots a_r(x_r): \cdot :b_1(y_1) \cdots b_s(y_s):\nonumber\\
    = &:a_1(x_1)\cdots a_{r-1}(x_{r-1}): \cdot :a_r(x_r) \cdot b_1(y_1) \cdots b_s(y_s) \label{Wick-F-1}\\
    &\quad + \sum_{j\in \B^\F} (-1)^{r^\F(j) - 1}\langle r, j\rangle_{\F} 
    :a_1(x_1)\cdots a_{r-1}(x_{r-1}):\cdot
    :b_1(y_1)\cdots \reallywidehat{b_j(y_j)}\cdots b_s(y_s): \label{Wick-F-2}\\
    &\quad  - \sum_{i\in A^\F\setminus\{r\} } (-1)^{r^\F - r^\F(i)}\langle i, r\rangle_\F :a_1(x_1)\cdots \reallywidehat{a_i(x_i)}\cdots a_{r-1}(x_{r-1}):\cdot:b_1(y_1)\cdots b_s(y_s):. \label{Wick-F-3}
\end{align}
    By the induction hypothesis, (\ref{Wick-F-1}) is equal to
    \begin{align}
        & \sum_{\rho, \eta=0}^{\min(r-1,s)} \sum_{I^b\in A^\B_\eta, J^b\in B^\B_\eta} \perm \mat \left(\ell_\B\left(a_{i_\alpha^b},\, b_{j_\beta^b}\right) F^\B\left(x_{i_{\alpha}^b},\, y_{j_\beta^b}\right)\right)_{\alpha, \beta=1}^\eta \nonumber\\
        & \hspace{3.1em} \cdot \sum_{I^f\in A^\F_\rho\setminus\{r\}, J^f\in B^\F_\rho} (-1)^{r^\F(i_1^f)+\cdots +r^\F(i_\rho^f)+s^\F(j_1^f)+\cdots + s^\F(j_\rho^f)}(-1)^{r^\F\rho + \frac{\rho(\rho+1)}{2}} \nonumber\\
        &\hspace{10.2em} \cdot \det \begin{pmatrix}
            \langle i_1^f, j_1^f\rangle_\F & \cdots & \langle i^f_1, j^f_{\eta}\rangle_\F\\
            \vdots &  & \vdots \\
            \langle i_\rho^f, j_1^f\rangle_\F & \cdots & \langle i^f_\rho, j^f_{\eta}\rangle_\F
        \end{pmatrix}\nonumber\\
        &\hspace{10.5em}\cdot :a_1(x_1) \cdots a_{r-1}(x_{r-1})\cdot a_r(x_r) \cdot b_1(y_1) \cdots b_s(y_s):_{O(I^b, I^f, J^b, J^f)} \label{Wick-F-4}\\
        & \quad\quad + \sum_{\rho = 0, \eta=1}^{\min(r-1,s)} \sum_{I^b\in A^\B_\eta, J^b\in B^\B_\eta}  \perm \mat \left(\ell_\B\left(a_{i_\alpha^b},\, b_{j_\beta^b}\right) F^\B\left(x_{i_{\alpha}^b},\, y_{j_\beta^b}\right)\right)_{\alpha, \beta=1}^\eta  
        \nonumber\\
        & \hspace{4.5em}\cdot \sum_{I^f\in A^\F_\rho\setminus\{r\},J^f\in B^\F_{\rho-1}} (-1)^{r^\F(i_1^f)+\cdots +r^\F(i_\rho^f)+1+s^\F(j_1^f)+1+\cdots + s^\F(j_{\rho-1}^f)+1}(-1)^{r^\F\rho + \frac{\rho(\rho+1)}{2}} \nonumber\\
        &\hspace{11.3em} \cdot \det \begin{pmatrix}
            \langle i_1^f, r\rangle_\F & \langle i_1^f, j_1^f\rangle_\F & \cdots & \langle i^f_1, j^f_{\eta-1}\rangle_\F\\
            \vdots & \vdots & & \vdots \\
            \langle i_\rho^f, r\rangle_\F & \langle i_\rho^f, j_1^f\rangle_\F & \cdots & \langle i^f_\rho, j^f_{\eta-1}\rangle_\F
        \end{pmatrix} \nonumber\\
        &\hspace{11.6em}\cdot :a_1(x_1) \cdots a_{r-1}(x_{r-1})\cdot \reallywidehat{a_r(x_r)} \cdot b_1(y_1) \cdots b_s(y_s):_{O(I^b, I^f, J^b, J^f)}\label{Wick-F-5} 
    \end{align}
We expand the determinant in (\ref{Wick-F-5}) by the first column, to obtain
\begin{align}
    & \sum_{\rho =0, \eta=1}^{\min(r-1,s)} \sum_{I^b\in A^\B_\eta, J^b\in B^\B_{\eta}}  
       \perm \mat \left(\ell_\B\left(a_{i_\alpha^b},\, b_{j_\beta^b}\right) F^\B\left(x_{i_{\alpha}^b},\, y_{j_\beta^b}\right)\right)_{\alpha, \beta=1}^\eta \nonumber
        \nonumber\\
        & \hspace{4.7em} \cdot\sum_{I^f\in A^\F_\rho\setminus\{r\}, J^f\in B^\F_{\rho-1}} (-1)^{r^\F(i_1^f)+\cdots +r^\F(i_\rho^f)+s^\F(j_1^f)+\cdots + s^\F(j_{\rho-1}^f)+\rho}(-1)^{r^\F\rho + \frac{\rho(\rho+1)}{2}} \nonumber\\
        &\hspace{10.2em} \cdot  \sum_{k=1}^{\eta} (-1)^{k+1} \langle i_k^f, r\rangle_\F
        \det \begin{pmatrix}
            \hcancel{\langle i_1^f, r\rangle_\F}& \langle i_1^f, j_1^f\rangle_\F & \cdots & \langle i^f_1, j^f_{\eta-1}\rangle_\F\\
            \vdots & \vdots & & \vdots \\
            \hcancel{\langle i_k^f, r\rangle_\F} & \hcancel{\langle i_k^f, j_1^f\rangle_\F} & \cdots & \hcancel{\langle i^f_k, j^f_{\eta-1}\rangle_\F}\\
            \vdots & \vdots & & \vdots \\
            \hcancel{\langle i_\eta^f, r\rangle_\F} & \langle i_\eta^f, j_1^f\rangle_\F & \cdots & \langle i^f_\eta, j^f_{\eta-1}\rangle_\F
        \end{pmatrix} \nonumber\\
        &\hspace{10.5em}\cdot :a_1(x_1) \cdots \reallywidehat{a_{i_k^f}(x_{i^f_k})} \cdots a_{r-1}(x_{r-1})\cdot \reallywidehat{a_r(x_r)} \cdot b_1(y_1) \cdots b_s(y_s):_{O(I^b, I^f, J^b, J^f)} .\nonumber
    \end{align}
Using an argument similar to that in the proof of Lemma \ref{Comb-Id-Lemma-B}, we may change the order of summation of $i^b$ and $k$. Set $i = i^f_k$, switch the order of summation, adjust the indices in $I^f$ by kicking the $k$-th term out and shifting the $(k+1)$-th term to the $\rho$-th term by 1, while noticing that in the adjusted sequence $I^f$, $r^\F(i^f_p)$ is shifted by 1 for $p = k+1, \dots, \rho$, so that 
\begin{align*}
    & (-1)^{r^\F(i^f_1) + \cdots + r^\F(i^f_k) + \cdots + r^\F(i^f_\rho)+\rho}(-1)^{r^\F\rho+\frac{\rho(\rho+1)}{2}}(-1)^{k+1} \\
    &\quad = (-1)^{r^\F(i)-r^\F} (-1)^{r^\F(i^f_1) + \cdots + r^\F(i^f_{k-1}) + r^\F(i^f_{k+1})-1 + \cdots + r^\F(i^f_{\rho})-1} (-1)^{(r^\F-1)(\rho-1)+\frac{\rho(\rho-1)}{2}},
\end{align*} 
then we obtain 
\begin{align}
    & \sum_{i\in A^\B} (-1)^{r^\F(i)-r^\F} \langle i, r\rangle_\F \nonumber\\
        & \hspace{3 em} \cdot\sum_{\rho =0, \eta=1}^{\min(r-1,s)} \sum_{I^b\in A^\B_{\eta-1}\setminus\{r,i\}, J^b\in B^\B_{\eta-1}}  
        \perm \mat \left(\ell_\B\left(a_{i_\alpha^b},\, b_{j_\beta^b}\right) F^\B\left(x_{i_{\alpha}^b},\, y_{j_\beta^b}\right)\right)_{\alpha, \beta=1}^\eta \nonumber\\
        & \hspace{3em} \cdot\sum_{I^f\in A^\F_\rho, J^f\in B^\F_\rho} (-1)^{r^\F(i_1^f)+\cdots +r^\F(i_{\rho-1}^f)+s^\F(j_1^f)+\cdots + s^\F(j_\rho^f)}(-1)^{(r^\F-1)(\rho-1) + \frac{\rho(\rho-1)}{2}} \nonumber\\
        &\hspace{5.8em} \cdot \det \mat \left(\langle i_\gamma^f,\, j_\delta^f\rangle_{\F}\right)_{\gamma, \delta=1}^{\rho-1} \nonumber\\
        &\hspace{6em}\cdot :a_1(x_1) \cdots \reallywidehat{a_{i}(x_{i})} \cdots a_{r-1}(x_{r-1})\cdot \reallywidehat{a_r(x_r)} \cdot b_1(y_1) \cdots b_s(y_s):_{O(I^b, I^f, J^b, J^f)} . \label{Wick-F-6}
    \end{align}
Clearly, from the induction hypothesis, $(\ref{Wick-F-6}) + (\ref{Wick-F-3}) = 0$. Thus $(\ref{Wick-F-1}) + (\ref{Wick-F-2}) + (\ref{Wick-F-3}) =(\ref{Wick-F-2})+ (\ref{Wick-F-4})$. 

To show $(\ref{Wick-F-2})+ (\ref{Wick-F-4}) = (\ref{Eq-Wick-thm-general-case})$, we analyze the sum of $I^f$ in (\ref{Eq-Wick-thm-general-case}). Clearly, summing $I^f$ over $A^\F_\rho$ is the same as summing over all possible $\rho$-tuples 
$I^f=(i^f_1, \dots, i^f_\rho)$ satisfying $i^f_\rho = r$, i.e., 
\begin{align}
1\leq i^f_1 < \cdots < i^f_{\rho-1} \leq r-1, i^f_\rho = r, \label{Decompose-A-B-rho-1}
\end{align}
together with those $I^f$ satisfying $i^f_{\rho} \leq r-1$, i.e., 
\begin{align}
    1\leq i^f_1 < \cdots < i^f_{\rho-1} < i^f_{\rho} \leq  r-1. \label{Decompose-A-B-rho-2}
\end{align}
Clearly, (\ref{Wick-F-4}) corresponds to the part in (\ref{Eq-Wick-thm-general-case}) where the sum is over $I^f$ satisfying (\ref{Decompose-A-B-rho-2}). What remains to show is that (\ref{Wick-F-2}) is equal to 
\begin{align}
    & \sum_{\rho, \eta=0}^{\min(r,s)} \sum_{I^b\in A^B_\eta} \sum_{J^b\in B^\B_\eta}\perm \mat \left(\ell_\B\left(a_{i_\alpha^b},\, b_{j_\beta^b}\right) F^\B\left(x_{i_{\alpha}^b},\, y_{j_\beta^b}\right)\right)_{\alpha, \beta=1}^\eta \nonumber\\
    & \hspace{3em} \cdot\sum_{I^f\in A^\F_\rho, i^f_\rho = r} \sum_{J^f\in B^\F_\rho} (-1)^{r^\F(i_1^f)+\cdots +r^\F(i_{\rho-1}^f)+r^\F +s^\F(j_1^f)+\cdots + s^\F(j_\rho^f)}(-1)^{r^\F\rho + \frac{\rho(\rho+1)}{2}} \nonumber\\
     &\hspace{8.5em} \cdot \det \mat \left( \langle i_\alpha^f, j_\beta^f \rangle \right)_{\alpha, \beta=1}^\rho \nonumber\\
     &\hspace{8.8em}\cdot :a_1(x_1) \cdots a_r(x_r) b_1(y_1) \cdots b_s(y_s):_{O(I^b, I^f, J^b, J^f)}. \label{Wick-F-7}
\end{align}
Expand the determinant in (\ref{Wick-F-7}) by its last row and organize it as
\begin{align}
    & \sum_{\rho, \eta=0}^{\min(r,s)} \sum_{I^b\in A^B_\eta} \sum_{J^b\in B^\B_\eta}\perm \mat \left(\ell_\B\left(a_{i_\alpha^b},\, b_{j_\beta^b}\right) F^\B\left(x_{i_{\alpha}^b},\, y_{j_\beta^b}\right)\right)_{\alpha, \beta=1}^\eta \nonumber\\
    & \hspace{3em} \cdot\sum_{I^f\in A^\F_\rho, i^f_\rho = r} \sum_{J^f\in B^\F_\rho} (-1)^{r^\F(i_1^f)+\cdots +r^\F(i_{\rho-1}^f)+r^\F +s^\F(j_1^f)+\cdots + s^\F(j_\rho^f)}(-1)^{r^\F\rho + \frac{\rho(\rho+1)}{2}} \nonumber\\
     &\hspace{8.5em} \cdot  \sum_{k=1}^\eta (-1)^{k+\rho}\langle r, j_k^f\rangle_\F 
    \det \begin{pmatrix}
    \langle i^f_1, j^f_1 \rangle_\F & \cdots & \hcancel{\langle i^f_1, j^f_k\rangle_\F} & \cdots & \langle i^f_1, j^f_\eta\rangle_\F\\ 
    \vdots & & \vdots & & \vdots \\
    \langle i^f_{\eta-1}, j^f_1\rangle_\F & \cdots & \hcancel{\langle i^f_{\eta-1}, j^f_k\rangle_\F} & \cdots & \langle i^f_{\eta-1}, j^f_\eta\rangle_\F\\
    \hcancel{\langle r, j^f_1\rangle_\F} & \cdots & \hcancel{\langle r, j^f_k\rangle_\F} & \cdots & \hcancel{\langle r, j^f_\eta\rangle_\F}\\
    \end{pmatrix} \nonumber\\
     &\hspace{8.8em}\cdot :a_1(x_1) \cdots a_r(x_r) b_1(y_1) \cdots b_s(y_s):_{O(I^b, I^f, J^b, J^f)} . \nonumber
\end{align}
Then the summation of $I^f$ is over $A^\F_{\rho-1}\setminus\{r\}$. Using an argument similar to that in the proof of Lemma \ref{Comb-Id-Lemma-B}, we may change the order of summation of $J^f$ and $k$. Set $j = j^f_k$, switch the order of summation, adjust the indices in $J^f$ by kicking the $k$-th term out and shifting the $(k+1)$-th term to the $\eta$-th term by 1, while noticing that in the adjusted sequence $J^f$, $s^\F(j^f_p)$ is shifted by 1 for $p = k+1, \dots, \rho$, so that
\begin{align*}
& (-1)^{r^\F(i_1^f) + \cdots + r^\F(i_{\rho-1}^f) + r^\F + s^\F(j_1^f) + \cdots + s^\F(j_k^f) + \cdots + s^\F(j_\rho^f)} (-1)^{r^\F\rho + \frac{\rho(\rho+1)}{2}}(-1)^{k+\rho}\\
&\quad = (-1)^{s^\F(j)-1} (-1)^{r^\F(i_1^f) + \cdots + r^\F(i_{\rho-1}^f) + s^\F(j_1^f) + \cdots + s^\F(j_{k-1}^f) + s^\F(j_{k+1}^f)-1 \cdots + s^\F(j_\rho^f)-1}\\
& \hspace{20em}\cdot  (-1)^{(r^\F-1)(\rho-1) + \frac{\rho(\rho-1)}{2}},
\end{align*}
Then we obtain 
\begin{align}
    & \sum_{j=1}^s (-1)^{s^\F(j)-1} \langle r, j\rangle_\F \sum_{\rho, \eta=0}^{\min(r,s)} \sum_{I^b\in A^\B_{\eta}} \sum_{J^b\in B^\B_{\eta}} 
    \perm \mat \left(\ell_\B\left(a_{i_\alpha^b},\, b_{j_\beta^b}\right) F^\B_{m_{i_\alpha^b}n_{i_\beta^b}}\left(x_{i_{\alpha}^b},\, y_{j_\beta^b}\right)\right)_{\alpha, \beta=1}^\eta \nonumber\\
    & \hspace{3em} \cdot\sum_{I^f\in A^\F_{\rho-1}\setminus\{r\}, J^f\in B^\F_{\rho-1}\setminus\{j\}} (-1)^{r^\F(i_1^f)+\cdots +r^\F(i_{\rho-1}^f)+s^\F(j_1^f)+\cdots + s^\F(j_{\rho-1}^f)}(-1)^{(r^\F-1)(\rho-1) + \frac{\rho(\rho-1)}{2}} \nonumber\\
     &\hspace{10.2em} \cdot \det \mat \left(\langle i_\alpha^f, j_\beta^f \rangle_\F\right)_{\gamma, \delta=1}^{\rho-1} \nonumber\\
     &\hspace{10.5em}\cdot :a_1(x_1) \cdots a_r(x_r) b_1(y_1) \cdots b_s(y_s):_{O(I^b, I^f, J^b, J^f)} . \label{Wick-F-8}
\end{align}
Clearly, (\ref{Wick-F-8}) = (\ref{Wick-F-2}) by the induction hypothesis. The conclusion then follows for this case. 

\noindent\textbf{Case 2. $a_{r}\in \h^\B.$} With Proposition \ref{Nord-recursion-left-boson-prop} and Proposition \ref{Nord-recursion-right-boson-prop}, we may argue in a similar fashion using the column and row expansions of permanents. Note also that there does not exist any sign issues to take care of, so the situation is simpler. We shall omit the details here. 
\end{proof}

\subsection{Products and iterates of vertex operators, and weak associativity with pole-order condition}

From Theorem \ref{Wick-thm-general-case}, we obtain immediately the following.

\newcommand{\wrapper}[1]{%
  \tikz[baseline=(math.base)]{
    \node[inner sep=0.85pt] (math) {$\scriptstyle #1$}; 
    \draw[line width=0.5pt] (math.north west) -- (math.south west) -- (math.south east) -- (math.north east);
  }%
}

\newcommand{\wrappernormal}[1]{%
  \tikz[baseline=(math.base)]{
    \node[inner sep=1.1pt] (math) {$\displaystyle #1$}; 
    \draw[line width=0.5pt] (math.north west) -- (math.south west) -- (math.south east) -- (math.north east);
  }%
}

\begin{cor}\label{alg-prod-general} 
    For \(r, s \in \N\), $a_1, \dots, a_r, b_1, \dots, b_s\in \h^\B \cup \h^\F, m_1, \ldots, m_r, n_1, \ldots, n_s \in \Z_+$, 
    \begin{align}
     & Y(a_1(-m_1+|a_1|/2)\cdots a_r(-m_r+|a_r|/2)\one, x)Y(b_1(-n_1+|b_1|/2)\cdots b_s(-n_s+|b_s|/2)\one, y) 
     \nonumber\\
&\quad     = \sum_{\rho, \eta=0}^{\min(r,s)} \sum_{I^b\in A^\B_\eta, J^b\in B^\B_\eta}  \perm \mat \left(\ell_\B\left(a_{i_\alpha^b},\, b_{j_\beta^b}\right) F^\B_{m_{i_\alpha^b}n_{i_\beta^b}}\left(x, y\right)\right)_{\alpha, \beta=1}^\eta \nonumber\\
     & \hspace{3em}\cdot \sum_{I^f\in A^\F_\rho, J^f\in B^\F_\rho} (-1)^{r^\F(i_1^f)+\cdots +r^\F(i_\rho^f)+s^\F(j_1^f)+\cdots + s^\F(j_\rho^f)}(-1)^{r^\F\rho + \frac{\rho(\rho+1)}{2}} \nonumber\\
     &\hspace{8.5em} \cdot \det \mat \left(\ell_\F\left(a_{i_\gamma^f},\, b_{j_\delta^f}\right) F^\F_{m_{i_\delta^f}n_{j_\delta^f}}\left(x, y\right)\right)_{\gamma, \delta=1}^\rho \nonumber\\
     &\hspace{8.8em}\cdot :a_1^{(m_1-1)}(x) \cdots a_r^{(m_r-1)}(x) b_1^{(n_1-1)}(y) \cdots b_s^{(n_s-1)}(y):_{O(I^b, I^f, J^b, J^f)}. \label{Product-vo-algebra}
    \end{align}
\end{cor}
\begin{proof}
    Substitute \(x_1=\cdots=x_r=x,\ y_1=\cdots=y_s=y\) in Theorem \ref{Wick-thm-general-case}.
\end{proof}
\begin{cor}\label{alg-iterate-first-step}
    For \(r, s \in \N\), $a_1, \dots, a_r, b_1, \dots, b_s\in \h^\B \cup \h^\F, m_1, \ldots, m_r, n_1, \ldots, n_s \in \Z_+$, 
    \begin{align}
    & Y(a_1(-m_1+|a_1|/2)\cdots a_r(-m_r+|a_r|/2)\one, x)b_1(-n_1+|b_1|/2)\cdots b_s(-n_s+|b_s|/2)\one
     \nonumber\\
    &\quad = \sum_{\rho, \eta=0}^{\min(r,s)} \sum_{I^b\in A^\B_\eta, J^b\in B^\B_\eta}  \perm \mat \left(\ell_\B\left(a_{i_\alpha^b},\, b_{j_\beta^b}\right) F^\B_{m_{i_\alpha^b}n_{i_\beta^b}}\left(x, 0\right)\right)_{\alpha, \beta=1}^\eta \nonumber\\
     & \hspace{3em}\cdot \sum_{I^f\in A^\F_\rho, J^f\in B^\F_\rho} (-1)^{r^\F(i_1^f)+\cdots +r^\F(i_\rho^f)+s^\F(j_1^f)+\cdots + s^\F(j_\rho^f)}(-1)^{r^\F\rho + \frac{\rho(\rho+1)}{2}} \nonumber\\
     &\hspace{8.5em} \cdot \det \mat \left(\ell_\F\left(a_{i_\gamma^f},\, b_{j_\delta^f}\right) F^\F_{m_{i_\delta^f}n_{j_\delta^f}}\left(x, 0\right)\right)_{\gamma, \delta=1}^\rho \nonumber\\
     &\hspace{8.5em}\cdot \bigg(a_1^{(m_1-1)}(x)^- \cdots a_r^{(m_r-1)}(x)^-\bigg)_{O(I^b,I^f)} \nonumber\\
     &\hspace{8.5em}\cdot \bigg(b_1(-n_1+|b_1|/2) \cdots b_s(-n_s+|b_s|/2)\one \bigg)_{O(J^b, J^f)} . \label{eq-alg-iterate-first-step}
    \end{align}
\end{cor}
\begin{proof}
    We first let both sides of (\ref{Product-vo-algebra}) act on the vaccum \(\one\). Then let \(y_1=\cdots=y_s=0\). Also note that 
    \[\begin{aligned}
& : a_1(x_1) \cdots a_r(x_r) b_1\left(-n_1\right) \cdots b_s\left(-n_s\right): \mathbf{1} \\
&\quad= a_1(x_1)^{-} \cdots a_r(x_r)^{-} b_1\left(-n_1\right) \cdots b_s\left(-n_s\right) \mathbf{1}
\end{aligned}\]
    since the positive modes annihilates the vacuum in the normal ordering. 
    %\[B_{mn}(x,0)=(n+1)\binom{-n-2}{m}x^{-m-n-2}=\]
    %\[F_{mn}(x,0)=\binom{-n-1}{m}x^{-m-n-1}=\left(x^{-n-1}\right)^{(m)}\]
\end{proof}

\begin{cor}\label{alg-iterate-second-step}
With the same notation as in Corollary \ref{alg-iterate-first-step}, we have 
\begin{align}
    & Y(Y(a_1(-m_1+|a_1|/2)\cdots a_r(-m_r+|a_r|/2)\one, x)b_1(-n_1+|b_1|/2)\cdots b_s(-n_s+|b_s|/2)\one, y)
     \nonumber\\
  &\quad   = \sum_{\rho, \eta=0}^{\min(r,s)} \sum_{I^b\in A^\B_\eta, J^b\in B^\B_\eta}  \perm \mat \left(\ell_\B\left(a_{i_\alpha^b},\, b_{j_\beta^b}\right) F^\B_{m_{i_\alpha^b}n_{i_\beta^b}}\left(x, 0\right)\right)_{\alpha, \beta=1}^\eta \nonumber\\
     & \hspace{3em} \cdot\sum_{I^f\in A^\F_\rho, J^f\in B^\F_\rho} (-1)^{r^\F(i_1^f)+\cdots +r^\F(i_\rho^f)+s^\F(j_1^f)+\cdots + s^\F(j_\rho^f)}(-1)^{r^\F\rho + \frac{\rho(\rho+1)}{2}} \nonumber\\
     &\hspace{8.5em} \cdot \det \mat \left(\ell_\F\left(a_{i_\gamma^f},\, b_{j_\delta^f}\right) F^\F_{m_{i_\delta^f}n_{j_\delta^f}}\left(x, 0\right)\right)_{\gamma, \delta=1}^\rho \nonumber\\
     &\hspace{8.8em}\cdot :a_1^{(m_1-1)}(y+x) \cdots a_r^{(m_r-1)}(y+x) b_1^{(n_1-1)}(x) \cdots b_s^{(n_s-1)}(x) :_{O(I^b, I^f, J^b, J^f)}.\label{eq-alg-iterate-second-step}
\end{align}
\end{cor}
\begin{proof}
    The conclusion follows from an argument similar to that for Corollary 4.7 in \cite{FQ-Fermion-1}. 
\end{proof}

We have the following weak associativity with pole-order condition: 

\begin{prop}\label{weak-ass-alg}
 For any $a_1, \dots, a_r,\ b_1, \dots, b_s\in \h,\ m_1, \dots, m_r,\ n_1, \dots, n_s\in \N$, we let 
$$u = a_1(-m_1-1/2)\cdots a_r(-m_r-1/2)\one,\ v = b_1(-n_1-1/2) \cdots b_s(-n_s-1/2)\one,$$    
    then for every homogeneous $w\in W$, the identity
    $$(x+y)^P Y(u, x)Y(v, y)w = (x+y)^P Y(Y(u, x)v, y)w$$
    holds, where
    $$P = \wt(w) + m_1 + \cdots + m_r + r$$
    that is independent of $v$. 
\end{prop}

\begin{proof}
    It suffices to notice that \[F^\B_{mn}(x+y,y)=F^\B_{mn}(x,0), F^\F_{mn}(x+y,y)=F^\F_{mn}(x,0).\] The conclusion then follows with an argument similar to that in Proposition 4.8 in \cite{FQ-Fermion-1}.
\end{proof}

As we have discussed in the beginning of this section, Proposition \ref{weak-ass-alg} 
finishes our proof of Theorem \ref{It-is-a-mosva}.

\section{Proof of Theorem \ref{it-is-a-module}} \label{Proof-it-is-a-module}

We prove Theorem \ref{it-is-a-module} in this section. We need to show that 
 $(W,  Y_W)$ given in Section 4 satisfies the axioms of a $\theta$-twisted $V$-module.

As in Section 6, the axioms for the grading, identity property, and $D$-derivative property may be proved 
by a straightforward modification of the arguments in Section 5.4 of \cite{Q-Fermion-2}. We shall not include the details here. 
From Remark \ref{equiv-defn-mod},  to prove the generalized rationality and associativity, we need only prove the weak associativity with pole-order condition, which is the main technical part of the proof of 
Theorem \ref{it-is-a-module}.

\subsection{A generalized Wick's Theorem for the normal ordering $\typecolon \cdot \typecolon$ 
on $W$}

As in Subsection \ref{sec-Wick-thm-simple-case}, to prove Theorem \ref{it-is-a-module},
we also need a generalized Wick's theorem. 
We formulate and prove a generalized Wick's theorem needed for the proof of Theorem \ref{It-is-a-mosva}
in this subsection. 

The entire analysis in this subsection is completely analogous to the what we have done in Subsection \ref{sec-Wick-thm-simple-case}, with the only difference being the existence of the zero modes, which is what we are going to focus on discussing.

For convenience, for $a\in \h^\B\cup \h^\F$ and $m\in \Z_+$, the notation $a^{(m-1)}(x)$ means the generating function defined in Section \ref{gen-func-mod}; the notation $a^{(m-1)}(x)_V$ means the generating function defined in Section \ref{gen-func-alg}.   

\begin{prop} \label{Nord-module-recursion-left-boson-prop}
    For $a\in \h^\B$, $b_1, \ldots, b_s\in \h^\B \cup \h^\F$, we have 
        \begin{align*}
        & a(x) \typecolon b_1(y_1)\cdots b_s(y_s)\typecolon - \typecolon a(x)b_1(y_1)\cdots b_s(y_s)\typecolon\\
        &\quad =    \sum_{i\in B^\B} (a, b_i) \ell_\B F^\B(x,y_i) \typecolon b_1(y_1)\cdots \reallywidehat{b_i(y_i)}\cdots b_s(y_s)\typecolon. 
    \end{align*}
\end{prop}

\begin{proof}
The proof is verbatim to that of Proposition \ref{Nord-recursion-left-boson-prop}.
\end{proof}

\begin{prop} \label{Nord-module-recursion-left-fermion-prop}
    For \(a\in \h^\F\), $b_1, \ldots, b_s\in \h^\B \cup \h^\F$, we have  
        \begin{align} 
        & a(x) \typecolon b_1(y_1)\cdots b_s(y_s)\typecolon - \typecolon a(x)b_1(y_1)\cdots b_s(y_s)\typecolon \nonumber\\
        &\quad=  \sum_{j\in B^\F} (-1)^{r^\F(j)}(a, b_j) \ell_\F \iota_{xy_j}\overline{G}(x, y_j) \typecolon b_1(y_1) \cdots \reallywidehat{b_j(y_j)}\cdots b_s(y_s)\typecolon \label{Nord-module-recursion-left-fermion-eQ-Thesis}\\
        &\qquad+ \sum_{j\in B^\F} \frac{(-1)^{r^\F(j)}}{2} (a,b_j) \ell_\F x^{-1/2}y_j^{-1/2} \typecolon b_1(y_1) \cdots \reallywidehat{b_j(y_j)}\cdots b_s(y_s)\typecolon \label{Nord-module-recursion-left-fermion-eQ-Mod}\\
&\quad = \sum_{j\in B^\F} (-1)^{r^\F(j)}(a, b_j) \ell_\F \iota_{xy_j}G(x, y_j)   \typecolon b_1(y_1) \cdots \reallywidehat{b_j(y_j)}\cdots b_s(y_s)\typecolon ,   \label{Nord-module-recursion-left-fermion-eQ-Fermion-2}
    \end{align}
    where 
    \[G(x, y) = \overline{G}(x,y) + \frac 1 2 x^{-1/2}y^{-1/2} =  \frac 1 2 x^{-1 / 2} y^{-1 / 2}(x+y)(x-y)^{-1}. \] 
\end{prop}

\begin{proof}
    Write \(a(x)=a(x)^-+a(x)^0+a(x)^+\).  Note that still \[a(x)^-\typecolon b_1(y_1) \cdots b_{s} (y_s)\typecolon = \typecolon a(x)^-  b_1(y_1) \cdots b_{s} (y_s)\typecolon,\] but the zero mode and positive modes need separate analysis. 

    For the positive modes, with a almost identical computation in Proposition \ref{Nord-recursion-left-fermion-prop}, we obtain 
    \begin{align}
        &  a(x)^+\typecolon b_1(y_1) \cdots b_{s} (y_s)\typecolon - \typecolon a(x)^+ b_1(y_1) \cdots b_{s} (y_s)\typecolon \nonumber \\
       &\quad   =  \sum_{p=1}^{s^\F} (-1)^{p} (a, b_{f_p}) \ell_\F \overline{G}(x, y_{f_p}) \typecolon b_1(y_1) \cdots \reallywidehat{b_{f_p}(y_{f_p})} \cdots b_s(y_s)\typecolon \nonumber \\
    &  \quad  =  \sum_{j\in B^\F} (-1)^{r^\F(j)}(a, b_j) \ell_\F \overline{G}(x, y_j) \typecolon b_1(y_1) \cdots \reallywidehat{b_j(y_j)}\cdots b_s(y_s)\typecolon .\label{Nord-mod-left-fermion-zero-mode-middle-step-2}
    \end{align}

    For the zero mode, by the definition of the normal ordering, we have
    \begin{align}
        &a(x)^0\typecolon b_1(y_1) \cdots b_{s} (y_s)\typecolon \nonumber \\
       &\quad = \sum_{\mu=0}^s \sum_{\nu = 0}^{s-\mu} \sum_{\sigma\in J_{\mu\nu}(1, \ldots, s)} (-1)^{\sigma^\F} a(x)^0 b_{\sigma(1)}(y_{\sigma(1)})^- \cdots b_{\sigma(\mu)}(y_{\sigma(\mu)})^- \nonumber\\
        & \hspace{8 em} \cdot  b_{\sigma(\mu+1)}(y_{\sigma(\mu+1)})^+ \cdots b_{\sigma(\mu+\nu)}(y_{\sigma(\mu+\nu)})^+ \nonumber\\
        & \hspace{8 em} \cdot \typecolon b_{\sigma(\mu+\nu+1)}(y_{\sigma(\mu+\nu+1)})^0 \cdots b_{\sigma(s)}(y_{\sigma(s)})^0 \typecolon \nonumber\\
        &\quad = \sum_{\mu=0}^s \sum_{\nu = 0}^{s-\mu} \sum_{\sigma\in J_{\mu\nu}(1, \ldots, s)} (-1)^{\sigma^\F} b_{\sigma(1)}(y_{\sigma(1)})^- \cdots b_{\sigma(\mu)}(y_{\sigma(\mu)})^- \nonumber\\
        & \hspace{8 em} \cdot  b_{\sigma(\mu+1)}(y_{\sigma(\mu+1)})^+ \cdots b_{\sigma(\mu+\nu)}(y_{\sigma(\mu+\nu)})^+ \nonumber\\
        & \hspace{8 em} \cdot  \typecolon a(x)^0 \cdot b_{\sigma(\mu+\nu+1)}(y_{\sigma(\mu+\nu+1)})^0 \cdots b_{\sigma(s)}(y_{\sigma(s)})^0\typecolon \nonumber\\
        &\quad\quad + \sum_{\mu=0}^s \sum_{\nu = 0}^{s-\mu} \sum_{\sigma\in J_{\mu\nu}(1, \ldots, s)} (-1)^{\sigma^\F} b_{\sigma(1)}(y_{\sigma(1)})^- \cdots b_{\sigma(\mu)}(y_{\sigma(\mu)})^- \nonumber\\
        & \hspace{9em} \cdot  b_{\sigma(\mu+1)}(y_{\sigma(\mu+1)})^+ \cdots b_{\sigma(\mu+\nu)}(y_{\sigma(\mu+\nu)})^+ \nonumber\\
        & \hspace{9 em} \cdot \bigg( \sum_{i\in B^\F_{\sigma}[\mu+\nu+1, r]} \frac{(-1)^{r^\F_\sigma(i)-r^\F_\sigma(\mu+\nu)}}{2}\left(a, b_{\sigma(i)}\right) x^{-1/2} y_{\sigma(i)}^{-1/2}  \nonumber\\
    & \hspace{9 em} \cdot \typecolon b_{\sigma(\mu+\nu+1)}(y_{\sigma(\mu+\nu+1)})^0 \cdots \reallywidehat{b_{\sigma(i)}(y_{\sigma(i)})^0} \cdots  b_{\sigma(s)}(y_{\sigma(s)})^0 \typecolon \bigg)\nonumber\\
&\quad    =  \nord a(x)^0 b_1(y_1) \cdots b_s(y_s) \nord\label{Nord-mod-left-fermion-zero-mode-middle-step-1} \\
    & \quad\quad+ \sum_{j\in B^\F} (-1)^{r^\F(j)} (a, b_j)x^{-1/2}y_j^{-1/2} \nonumber \\
    & \hspace{5 em}\cdot \sum_{\mu=0}^{s-1}\ \sum_{\nu=0}^{s-1-\mu}
        \sum_{\tau\in J_{\mu\nu}(1,\dots,\widehat{j},\dots,s)}
        (-1)^{\tau^\F} 
        b_{\tau(1)}(y_{\tau(1)})^-\cdots b_{\tau(\mu)}(y_{\tau(\mu)})^- \nonumber \\
        &\hspace{12em}\cdot b_{\tau(\mu+1)}(y_{\tau(\mu+1)})^+\cdots b_{\tau(\mu+\nu)}(y_{\tau(\mu+\nu)})^+
        \nonumber \\
        &\hspace{11.93em}\cdot
        \typecolon
        b_{\tau(\mu+\nu+1)}(y_{\tau(\mu+\nu+1)})^0\cdots
        b_{\tau(s-1)}(y_{\tau(s-1)})^0 
        \typecolon ,\label{Nord-mod-left-fermion-zero-mode-middle-step}
    \end{align} 
    where the last equality follows from (\ref{Comb-Id-3}) in Lemma \ref{Comb-Id-Lemma} together with a shift of indices. Notice that (\ref{Nord-mod-left-fermion-zero-mode-middle-step}) is precisely (\ref{Nord-module-recursion-left-fermion-eQ-Mod}). The conclusion then follows from combining (\ref{Nord-mod-left-fermion-zero-mode-middle-step-2}) and (\ref{Nord-mod-left-fermion-zero-mode-middle-step-1}). 
\end{proof}

Similar to Propsition \ref{Nord-recursion-right-boson-prop} and Proposition \ref{Nord-recursion-right-fermion-prop}, the following propositions hold. 

\begin{prop}
    \label{Nord-module-recursion-right-boson-prop}
    For \(a_1, \ldots, a_r \in \h^\B \cup \h^\F\), \(b\in \h^\B\), we have
        \begin{align*}
        & \typecolon a_1(x_1) \cdots a_r(x_r) \typecolon b(y) - \typecolon a_1(x_1)\cdots a_r(x_r) b(y)\typecolon \\
      &\quad   =  \sum_{i\in A^\B} (a_i, b) \ell_\B F^\B(x_i, y) \typecolon a_1(x_1) \cdots \reallywidehat{a_i(x_i)} \cdots a_r(x_r)\typecolon.
    \end{align*}
\end{prop}

 \begin{prop}
    \label{Nord-module-recursion-right-fermion-prop}
\(a_1, \ldots, a_r \in \h^\B \cup \h^\F\), \(b\in \h^\F\), we have 
        \begin{align} 
        & \typecolon a_1(x_1)\cdots a_r(x_r)\typecolon b(y) -\typecolon a_1(x_1)\cdots a_r(x_r) b(y)\typecolon \nonumber\\
   &\quad     =\sum_{j\in A^\F} (-1)^{r^\F - r^\F(j)}(a_j, b) \ell_\F \overline{G}(x_j, y) \typecolon a_1(x_1) \cdots \reallywidehat{a_j(x_j)}\cdots a_r(x_r)\typecolon \label{Nord-module-recursion-right-fermion-eQ-Thesis}\\
        &\quad\quad+ \sum_{j\in A^\F} 
\frac{(-1)^{r^\F - r^\F(j)}}{2} (a_j, b) \ell_\F x_j^{-1/2}y^{-1/2} \typecolon a_1(x_1) \cdots \reallywidehat{a_j(x_j)}\cdots a_r(x_r)\typecolon \label{Nord-module-recursion-right-fermion-eQ-Mod}\\
&\quad=  \sum_{j\in A^\F} (-1)^{r^\F - r^\F(j)}(a_j, b) \ell_\F G(x, y_j)   \typecolon a_1(x_1) \cdots \reallywidehat{a_j(x_j)}\cdots a_r(x_r)\typecolon.  \label{Nord-module-recursion-right-fermion-eQ-Fermion-2}
    \end{align}
\end{prop}

Based on Propositions \ref{Nord-module-recursion-left-boson-prop}--\ref{Nord-module-recursion-right-fermion-prop}, by an almost identical induction as in the proof of Theorem \ref{Wick-thm-general-case}, we obtain the following generalized 
Wick's theorem needed for the proof of Theorem \ref{it-is-a-module}:

\begin{thm}\label{Mod-wick-thm-general-case} 
    Let $r, s\in \Z_+$, $a_1, \dots, a_r, b_1, \dots, b_s\in \h^\B \cup \h^\F$, $m_1, \dots, m_r, n_1, \dots, n_s\in \Z_+$. Then, 
    \begin{align}
     & \nord a_1^{(m_1-1)}(x_1) \cdots a_r^{(m_r-1)}(x_r)\nord  \cdot \nord b_1^{(n_1-1)}(y_1) \cdots b_s^{(n_s-1)}(y_s)\nord   \nonumber\\
    &\quad = \sum_{\rho, \eta=0}^{\min(r,s)} \sum_{I^b\in A^\B_\eta, J^b\in B^\B_\eta}  \perm \mat \left(\ell_\B\left(a_{i_\alpha^b},\, b_{j_\beta^b}\right) F^\B_{m_{i_\alpha^b}n_{i_\beta^b}}\left(x_{i_{\alpha}^b},\, y_{j_\beta^b}\right)\right)_{\alpha, \beta=1}^\eta \nonumber\\
     & \hspace{3em} \cdot\sum_{I^f\in A^\F_\rho, J^f\in B^\F_\rho} (-1)^{r^\F(i_1^f)+\cdots +r^\F(i_\rho^f)+s^\F(j_1^f)+\cdots + s^\F(j_\rho^f)}(-1)^{r^\F\rho + \frac{\rho(\rho+1)}{2}} \nonumber\\
     &\hspace{8.5em} \cdot \det \mat \left(\ell_\F\left(a_{i_\gamma^f},\, b_{j_\delta^f}\right) G_{m_{i_\delta^f}n_{j_\delta^f}}\left(x_{i_\gamma^f},\, y_{j_\delta^f}\right)\right)_{\gamma, \delta=1}^\rho \nonumber\\
     &\hspace{8.8em}\cdot \nord a_1^{(m_1-1)}(x_1) \cdots a_r^{(m_r-1)}(x_r) b_1^{(n_1-1)}(y_1) \cdots b_s^{(n_s-1)}(y_s)\nord _{O(I^b, I^f, J^b, J^f)}. \label{Eq-mod-Wick-thm-general-case}
    \end{align}
\end{thm}

\subsection{Products and iterates of the naive vertex operator \(\bar{Y}_W\)}

From Theorem \ref{Mod-wick-thm-general-case}, by a direct substitution, we obtain immediately the formula of the product of two naive vertex operators. 
\begin{cor}\label{Mod-prod-general} 
    For \(r, s \in \N\), $a_1, \dots, a_r, b_1, \dots, b_s\in \h^\B \cup \h^\F, m_1, \ldots, m_r, n_1, \ldots, n_s \in \Z_+$, 
    \begin{align}
     & \bar{Y}_W(a_1(-m_1+|a_1|/2)\cdots a_r(-m_r+|a_r|/2)\one, x)\bar{Y}_W(b_1(-n_1+|b_1|/2)\cdots b_s(-n_s+|b_s|/2)\one, y) 
     \nonumber\\
     &\quad=\sum_{\rho, \eta=0}^{\min(r,s)} \sum_{I^b\in A^\B_\eta, J^b\in B^\B_\eta}  \perm \mat \left(\ell_\B\left(a_{i_\alpha^b},\, b_{j_\beta^b}\right) F^\B_{m_{i_\alpha^b}n_{i_\beta^b}}\left(x, y\right)\right)_{\alpha, \beta=1}^\eta \nonumber\\
     & \hspace{3em} \sum_{I^f\in A^\F_\rho, J^f\in B^\F_\rho} (-1)^{r^\F(i_1^f)+\cdots +r^\F(i_\rho^f)+s^\F(j_1^f)+\cdots + s^\F(j_\rho^f)}(-1)^{r^\F\rho + \frac{\rho(\rho+1)}{2}} \nonumber\\
     &\hspace{8.5em} \cdot \det \mat \left(\ell_\F\left(a_{i_\gamma^f},\, b_{j_\delta^f}\right) G_{m_{i_\delta^f}n_{j_\delta^f}}\left(x, y\right)\right)_{\gamma, \delta=1}^\rho \nonumber\\
     &\hspace{8.8em}\cdot \nord a_1^{(m_1-1)}(x) \cdots a_r^{(m_r-1)}(x) b_1^{(n_1-1)}(y) \cdots b_s^{(n_s-1)}(y)\nord_{O(I^b, I^f, J^b, J^f)}. \label{Product-naive-vo}
    \end{align}
\end{cor}

Using Corollary \ref{alg-iterate-first-step} and a similar argument as in the proof of Corollary \ref{alg-iterate-second-step}, we also obtain a formula for the iterate with the naive vertex operator. 
\begin{cor}\label{Mod-iterate}
For \(r, s \in \N\), $a_1, \dots, a_r, b_1, \dots, b_s\in \h^\B \cup \h^\F, m_1, \ldots, m_r, n_1, \ldots, n_s \in \Z_+$, \begin{align}
    & \bar{Y}_W(Y(a_1(-m_1+|a_1|/2)\cdots a_r(-m_r+|a_r|/2)\one, x)b_1(-n_1+|b_1|/2)\cdots b_s(-n_s+|b_s|/2)\one, y)
     \nonumber\\
     &\quad = \sum_{\rho, \eta=0}^{\min(r,s)} \sum_{I^b\in A^\B_\eta, J^b\in B^\B_\eta}  \perm \mat \left(\ell_\B\left(a_{i_\alpha^b},\, b_{j_\beta^b}\right) F^\B_{m_{i_\alpha^b}n_{i_\beta^b}}\left(x, 0\right)\right)_{\alpha, \beta=1}^\eta \nonumber\\
     & \hspace{3em}\cdot \sum_{I^f\in A^\F_\rho, J^f\in B^\F_\rho} (-1)^{r^\F(i_1^f)+\cdots +r^\F(i_\rho^f)+s^\F(j_1^f)+\cdots + s^\F(j_\rho^f)}(-1)^{r^\F\rho + \frac{\rho(\rho+1)}{2}} \nonumber\\
     &\hspace{8em} \cdot \det \mat \left(\ell_\F\left(a_{i_\gamma^f},\, b_{j_\delta^f}\right) F^\F_{m_{i_\delta^f}n_{j_\delta^f}}\left(x, 0\right)\right)_{\gamma, \delta=1}^\rho \nonumber\\
     &\hspace{8.1em}\cdot \nord a_1^{(m_1-1)}(y+x) \cdots a_r^{(m_r-1)}(y+x) b_1^{(n_1-1)}(x) \cdots b_s^{(n_s-1)}(x) \nord_{O(I^b, I^f, J^b, J^f)}.\label{eq-Mod-iterate}
\end{align}
\end{cor}
\begin{proof}
    The conclusion follows from an argument similar to that for Corollary 4.7 in \cite{FQ-Fermion-1}. 
\end{proof}
Comparing Corollaries \ref{Mod-prod-general} and \ref{Mod-iterate}, we can see that the product formula of the naive vertex operator \(\bar{Y}_W\) involves determinants of matrices whose entries involve the algebraic function \(G_{m_i n_j}(x,y)\), while the iterate formula of \(\bar{Y}_W\) involves determinants of matrices whose entries only involve the monomial \(F^\B_{m_i n_j}(x)=\left(x^{-n_j-1}\right)^{(m_i)}\). Therefore the weak associativity does not hold for \(\bar{Y}_W\), and we will introduce a correction in the next subsection.

\subsection{Properties and actions of \(\exp(\Delta(x))\)}
\label{subsec-properties-of-exp}

Recall the definition of the operator $\Delta(x)$. We choose an orthonormal basis $e_1, \dots, e_d$ of $\h^\F$. Then 
$$\Delta(x) = \sum_{i=1}^d \sum_{m,n\geq 1} C_{mn} e_i(m-1/2) e_i(n-1/2) x^{-m-n+1}. $$
Since $e_i(m-1/2)e_i(n-1/2)+e_i(n-1/2)e_i(m-1/2) = 0$, $\Delta(x)$ is independent of the choice of $e_i$ if and only if $C_{mn} = -C_{nm}$. Clearly, 
$$C_{mn} = \frac{1}{2} \frac{m-n}{m+n-1} \binom{-1/2}{m-1}\binom{-1/2}{n-1}$$
satisfies the condition. Note that the indices $m, n$ are shifted by 1 compared to the corresponding definition in \cite{Q-Fermion-2}. Correspondingly, Lemma 3.9 in \cite{FFR} should be modified as follows: for each $k\in \N$, 
\begin{align}
    \sum_{m=0}^k \binom{m+r-1} {r-1} \binom{k-m+t-1} {t-1} C_{(m+r)(k-m+t)} = \binom{-r-t+1}k C_{rt} .\label{C_rt-formula}
\end{align}
Lemma 3.31 in \cite{FFR} should be modified as
\begin{align}
    (x^{-n})^{(m-1)} + y^{-n-m+1}\sum_{p\geq 0} C_{n,p+m}\binom{p+m-1}{m-1} \left(\frac{x}{y}\right)^p= \iota_{yx}G_{mn}(y+x,y)  \label{C_mn-g-formula}
\end{align}
(recall the notation \eqref{x-p-q-notation}). 
Note also that we do not require any polarizations as in \cite{Q-Fermion-2}. The current formulation works regardless of the parity of $\dim \h$.

\begin{defn}\label{Total-contraction-number}
    Let $a_1, \dots, a_r\in \h^\B\cup \h^\F, m_1, \dots, m_r\in \Z_+$. Then for each $t\in \Z_+$ and $I=(i_1, \dots, i_{2t}) \in A^\F_{2t}$, we define the total contraction number 
    $$T_{a_1,\dots, a_r}^{m_1,\dots, m_r}(i_1, \dots, i_{2t}) = \operatorname{Pfaff} \mat \left( (a_{i_\alpha}, a_{i_\beta})C_{m_{i_\alpha}m_{i_\beta}} \right)_{\alpha, \beta=1}^{2t} .$$
    Here \(\operatorname{Pfaff}\) means the Pfaffian of a matrix. 
\end{defn}

\begin{rmk}
    Definition 5.2 of \cite{Q-Fermion-2} defined the total contraction number by a recursion in case $A^\F = \{1, \dots, r\}$. It is straightforward to check that the recursion defines the Pfaffian.
\end{rmk}

\begin{lemma}\label{exp-delta-on-basis}
    For $a_1, \ldots, a_r \in \mathfrak{h}^\B\cup \h^\F,\ m_1, \ldots, m_r \in \Z_+$,
    \begin{align}
        & \exp (\Delta(x)) a_1\left(-m_1+|a_1|/ 2\right) \cdots a_r\left(-m_r + |a_r|/ 2\right) \mathbf{1} \nonumber\\
      &\quad   = \sum_{t=0}^{\infty} \sum_{I\in A^\F_{2t}}(-1)^{r^\F(i_1)+\cdots+r^\F(i_{2 t})} T_{a_1, . . . a_r}^{m_1, \ldots, m_r}\left(i_1, \ldots, i_{2 t}\right)  x^{-m_1-\cdots - m_r + r}\nonumber\\
        & \hspace{4em} \cdot \bigg(a_1\left(-m_1 + |a_1| / 2\right) \cdots a_r\left(-m_r + |a_r| / 2\right) \mathbf{1}\bigg)_{O(I)}. \label{exp-basis-contraction}
    \end{align}
\end{lemma}

\begin{proof}
The proof is a straightforward modification of that of Corollary 5.7 in  \cite{FQ-Fermion-1}. 
\end{proof}

\begin{lemma}\label{exp-basic-commutator}\,
    For \(a\in\h^\B\), \(m\in \Z_+\), we have
    \begin{equation} \label{exp-boson-commutator}
        \left[\exp (\Delta(y)), a^{(m-1)}(x)_V\right]=0.
    \end{equation}
    For \(a\in\h^\F\), \(m\in \Z_+\), we have
    \begin{equation}\label{exp-fermion-commutator}
            \begin{aligned}
    {\left[\exp (\Delta(y)), a^{(m-1)}(x)_V\right] } & =\sum_{\alpha, \beta \geq 1} C_{\alpha \beta} a(\alpha-1 / 2) y^{-\alpha-\beta+1}\left(x^{\beta-1}\right)^{(m-1)} \exp (\Delta(y)) \\
    & =\sum_{\alpha \geq 1} a(\alpha-1 / 2)\left(\iota_{y x} G_{m \alpha}(y+x, y)-\left(x^{-\alpha}\right)^{(m-1)}\right) \exp (\Delta(y)).
    \end{aligned}
    \end{equation}
\end{lemma}
\begin{proof}
    (\ref{exp-boson-commutator}) follows from the commutativity between fermionic positive modes and all the bosonic modes. (\ref{exp-fermion-commutator}) can be proved by a straightforward modification of the proof of Proposition 5.8 in \cite{FQ-Fermion-1}. 
\end{proof}

\begin{thm}\label{exp-delta-A-B-lemma}
For $r, s \in \N, a_1, \dots, a_r, b_1, \dots, b_s\in \h^\B\cup\h^\F, m_1, \dots, m_r, n_1, \dots, n_s\in \N$, 
    \begin{align}
        & \exp(\Delta(y)):a_1^{(m_1-1)}(x)_V\cdots a_r^{(m_r-1)}(x)_V: b_1(-n_1+|b_1|/2)\cdots b_s(-n_s+|b_s|/2)\one \nonumber \\
       &\quad  =  \sum_{\eta = 0}^{\min(r, s)} \sum_{I^b \in A_\eta^\B, J^b\in B_\eta^\B} \perm \mat \left(\ell_\B \left(a_{i^b_\alpha}, b_{j^b_\beta}\right) F^\B_{m_{i^b_\alpha} n_{j^b_\beta}}(x,0)\right)_{\alpha, \beta=1}^\rho\nonumber\\
        & \quad \quad\cdot \sum_{k, l \geq 0} \sum_{\rho=0}^{\min(r-2k, s-2l)} (-1)^{r^\F\rho + \rho(\rho+1)/2} \nonumber\\
        & \qquad \cdot \sum_{I^f \in A_{2k,\rho}^\F} (-1)^{\sigma(I^f)} T_{a_1, \dots, a_r}^{m_1,\dots,m_r}\left(i^f_1, \dots, i^f_{2k}\right)\iota_{yx}(y+x)^{-m_{i^f_1}-\cdots - m_{i^f_{2k}} + k} \nonumber \\
        & \qquad \cdot \sum_{J^f \in B_{2l,\rho}^\F} (-1)^{\sigma(J^f)} T_{b_1, \dots, b_s}^{n_1, \dots, n_s}\left(j^f_1, \dots, j^f_{2l}\right) y^{-n_{j^f_1}-\cdots - n_{j^f_{2l}} + l}\nonumber \\
        & \hspace{2em} \cdot \det \mat \left( \ell_\F \left( a_{i^f_{2k+\gamma}}, b_{j^f_{2l+\delta}}\right) \iota_{yx} G_{m_{i^f_{2k+\gamma}},n_{j^f_{2l+\delta}}}(x+y, y)\right)_{\gamma, \delta = 1}^\rho\nonumber \\
        & \hspace{2em}\cdot \bigg(a_1^{(m_1-1)}(x)_V^-\cdots a_r^{(m_r-1)}(x)_V^-\bigg)_{O(I^b, I^f)} \bigg(b_1(-n_1+|b_1|/2)\cdots b_s(-n_s+|b_s|/2)\one\bigg)_{O(J^b, J^f)}  \label{exp-delta-A-B}
    \end{align}
    as a series in $V[[x, x^{-1}, y, y^{-1}]]$. 
\end{thm}

\begin{proof}
    The proof is a technical induction. We provide a sketch here. The conclusion clearly holds when $r=0$ and $s$ is arbitrary. Assume the conclusion for all smaller $r$ and arbitrary $s$. 
    
    \noindent\textbf{Case 1}. If $a_1 \in \h^\B$, then by Proposition \ref{Pos-modes-commute} and the definition of normal ordering, we have 
    \begin{align}
        & \exp(\Delta(y)):a_1^{(m_1-1)}(x)_V \cdot a_2^{(m_2-1)}(x)_V \cdots a_r^{(m_r-1)}(x)_V: \nonumber \\
        &\quad= \exp(\Delta(y))a_1^{(m_1-1)}(x)_V^- :a_2^{(m_2-1)}(x)_V \cdots a_r^{(m_r-1)}(x)_V: \label{exp-delta-A-B-boson-1}\\
        &\qquad + \exp(\Delta(y)) a_1^{(m_1-1)}(x)_V^0:a_2^{(m_2-1)}(x)_V \cdots a_r^{(m_r-1)}(x)_V: \label{exp-delta-A-B-boson-2}\\
        &\qquad + \exp(\Delta(y)) :a_2^{(m_2-1)}(x)_V \cdots a_r^{(m_r-1)}(x)_V:a_1^{(m_1-1)}(x)_V^+ \label{exp-delta-A-B-boson-3}.
    \end{align}
    With the induction hypothesis, we may check the following.
    \begin{enumerate}
        \item The action of (\ref{exp-delta-A-B-boson-1}) on $b_1(-n_1+|b_1|/2)\cdots b_s(-n_s+|b_s|/2)\one $ contributes to the summand of (\ref{exp-delta-A-B}) with $i^b_1 \neq 1$. 
        \item The action of (\ref{exp-delta-A-B-boson-2}) on $b_1(-n_1+|b_1|/2)\cdots b_s(-n_s+|b_s|/2)\one $ results in zero. 
        \item The action of (\ref{exp-delta-A-B-boson-3}) on $b_1(-n_1+|b_1|/2)\cdots b_s(-n_s+|b_s|/2)\one $ contributes to the summand of (\ref{exp-delta-A-B}) with $i^b_1 = 1$, where the permanent appearing in the summand is expanded by its first row. 
    \end{enumerate}
        
    \noindent\textbf{Case 2}. If $a_1 \in \h^\F$, then by Proposition \ref{Pos-modes-commute} and the definition of normal ordering, we have 
        \begin{align}
        & \exp(\Delta(y)):a_1^{(m_1-1)}(x)_V a_2^{(m_2-1)}(x)_V \cdots a_r^{(m_r-1)}(x)_V: \nonumber \\
        &\quad= \exp(\Delta(y))a_1^{(m_1-1)}(x)_V^- :a_2^{(m_2-1)}(x)_V \cdots a_r^{(m_r-1)}(x)_V: \nonumber\\
        & \qquad+ (-1)^{r^\F - 1} \exp(\Delta(y)) :a_2^{(m_2-1)}(x)_V \cdots a_r^{(m_r-1)}(x)_V:a_1^{(m_1-1)}(x)_V^+ \nonumber\\
        &\quad=  a_1^{(m_1-1)}(x)_V^-\exp(\Delta(y)):a_2^{(m_2-1)}(x)_V \cdots a_r^{(m_r-1)}(x)_V: \label{exp-delta-A-B-fermion-1}\\
        & \qquad+ \sum_{\alpha, \beta \geq 1} C_{\alpha\beta} a_1(\alpha-1/2) y^{-\alpha-\beta+1} (x^{\beta-1})^{(m_1-1)} \exp(\Delta(y)):a_2^{(m_2-1)}(x)_V \cdots a_r^{(m_r-1)}(x)_V: \label{exp-delta-A-B-fermion-2}\\
        &\qquad + (-1)^{r^\F - 1} \exp(\Delta(y)) :a_2^{(m_2-1)}(x)_V \cdots a_r^{(m_r-1)}(x)_V:a_1^{(m_1-1)}(x)_V^+ \label{exp-delta-A-B-fermion-3}. 
    \end{align}
    Here, the sign \((-1)^{r^\F - 1}\) appears because for \(a^{(m-1)}(x)^+\) to pass to the rightmost, it has to anti-commute with \(r^\F-1\) positive fermionic fields. 
    
    With the induction hypothesis, we may check the following.
    \begin{enumerate}
        \item Denote by \(\left(I^f\right)^c=\left(\left(i^f\right)^c_1, \dots, \left(i^f\right)^c_{r^\F-2k-\rho}\right)\) the element in \(A^\F_{r^\F-2k-\rho}\) determined uniquely by \(I^f\in A^\F_{2k, \rho}\). Then the action of (\ref{exp-delta-A-B-fermion-1}) on $b_1(-n_1+|b_1|/2)\cdots b_s(-n_s+|b_s|/2)\one $ contributes to the summand of (\ref{exp-delta-A-B}) with \(\left(i^f\right)^c_1=1\).
        \item The action of (\ref{exp-delta-A-B-fermion-2}) on $b_1(-n_1+|b_1|/2)\cdots b_s(-n_s+|b_s|/2)\one $ contributes to the summand of (\ref{exp-delta-A-B}) with $i^f_1 = 1$, together with the following part
        \begin{align}
            & \sum_{\eta = 0}^{\min(r, s)} \sum_{J^b \in A_\eta^\B, J^b\in B_\eta^\B} \perm \mat \left(\ell_\B \left(a_{i^b_\alpha}, b_{j^b_\beta}\right) F^\B_{m_{i^b_\alpha} n_{j^b_\beta}}(x,0)\right)_{\alpha, \beta=1}^\rho\nonumber\\
            & \quad \cdot \sum_{k, l \geq 0} \sum_{\rho=0}^{\min(r-2k, s-2l)} (-1)^{r^\F\rho + \rho(\rho+1)/2} \nonumber\\
            & \quad \cdot \sum_{I^f \in A_{2k,\rho}^\F} (-1)^{\sigma(I^f)} T_{a_1, \dots, a_r}^{m_1,\dots,m_r}\left(i^f_1, \dots, i^f_{2k}\right)\iota_{yx}(y+x)^{-m_{i^f_1}-\cdots - m_{i^f_{2k}} + k} \nonumber \\
            & \quad \cdot \sum_{J^f \in B_{2k,\rho}^\F} (-1)^{\sigma(J^f)} T_{b_1, \dots, b_s}^{n_1, \dots, n_s}\left(j^f_1, \dots, j^f_{2l}\right) y^{-n_{j^f_1}-\cdots - n_{j^f_{2l}} + l}\nonumber \\
            & \hspace{2em} \cdot  \det \begin{pmatrix}
                \langle 1, j^f_{2l+1}\rangle_G - \langle 1, j^f_{2l+1}\rangle_x  & \cdots & \langle 1, j^f_{2l+\rho}\rangle_G - \langle 1, j^f_{2l+\rho}\rangle_x\\
                \langle i^f_{2k+2}, j^f_{2l+1}\rangle_G & \cdots & \langle i^f_{2k+2}, j^f_{2l+\rho}\rangle_G \\
                \vdots & & \vdots \\
                \langle i^f_{2k+\rho}, j^f_{2l+1}\rangle_G & \cdots & \langle i^f_{2k+\rho}, j^f_{2l+\rho}\rangle_G 
            \end{pmatrix} \nonumber\\
            & \hspace{2em}\cdot \bigg(a_1^{(m_1-1)}(x)_V^-\cdots a_r^{(m_r-1)}(x)_V^-\bigg)_{O(1, I^b, I^f)} \bigg(b_1(-n_1+|b_1|/2)\cdots b_s(-n_s+|b_s|/2)\one\bigg)_{O(J^b, J^f)} ,  \label{exp-delta-A-B-fermion-4}
        \end{align}
        where for $i\in A^\F, j\in B^\F$, 
        $$\langle i, j\rangle_x = (a_i, b_j) (x^{-n_i})^{(m_i-1)}, \langle i, j\rangle_G = (a_i, b_j)\iota_{yx}G_{m_i n_j}(x+y, y).  $$
        \item The action of (\ref{exp-delta-A-B-fermion-3}) on $b_1(-n_1+|b_1|/2)\cdots b_s(-n_s+|b_s|/2)\one $, together with (\ref{exp-delta-A-B-fermion-4}), contributes to the summand of (\ref{exp-delta-A-B}) with $i^f(2k+1) = 1$. 
    \end{enumerate}
     A detailed argument for the above statements in the case $A^\F = \{1, \dots, r\}, B^\F=\{1, \dots, s\}$ can be found in the proof of Theorem 5.5 in \cite{Q-Fermion-2}, where the signs of the relevant 3-shuffles are expanded by (\ref{sign-3-shuffle}). These cover all the possiblities of where \(1\) could be in the summand (\ref{exp-delta-A-B}), thus we are done.
\end{proof}

\subsection{Products and iterates of the twisted vertex operator, and weak associativity with pole-order condition}

\begin{cor}\label{iterate-actual-vo}
    For $r, s \in \N, a_1, \dots, a_r, b_1, \dots, b_s\in \h^\B\cup\h^\F, m_1, \dots, m_r, n_1, \dots, n_s\in \N$, 
    \begin{align*}
        & Y_W(Y(a_1(-m_1+|a_1|/2)\cdots a_r(-m_r+|a_r|/2)\one, x)b_1(-n_1+|b_1|/2)\cdots b_s(-n_s+|b_s|/2)\one, y) \nonumber\\
        &\quad =  \sum_{\eta = 0}^{\min(r, s)} \sum_{I^b \in A_\eta^\B, J^b\in B_\eta^\B} \perm \mat \left(\ell_\B \left(a_{i^b_\alpha}, b_{j^b_\beta}\right) F^\B_{m_{i^b_\alpha} n_{j^b_\beta}}(x,0)\right)_{\alpha, \beta=1}^\rho\nonumber\\
        & \qquad \cdot \sum_{k, l \geq 0} \sum_{\rho=0}^{\min(r-2k, s-2l)} (-1)^{r^\F\rho + \rho(\rho+1)/2} \nonumber\\
        & \qquad \cdot \sum_{I^f \in A_{2k,\rho}^\F} (-1)^{\sigma(I^f)} T_{a_1, \dots, a_r}^{m_1,\dots,m_r}\left(i^f_1, \dots, i^f_{2k}\right)\iota_{yx}(y+x)^{-m_{i^f_1}-\cdots - m_{i^f_{2k}} + k} \nonumber \\
        & \qquad \cdot \sum_{J^f \in B_{2l,\rho}^\F} (-1)^{\sigma(J^f)} T_{b_1, \dots, b_s}^{n_1, \dots, n_s}\left(j^f_1, \dots, j^f_{2l}\right) y^{-n_{j^f_1}-\cdots - n_{j^f_{2l}} + l}\nonumber \\
        & \hspace{5em} \cdot \det \mat \left( \ell_\F \left( a_{i^f_{2k+\gamma}}, b_{j^f_{2l+\delta}}\right) \iota_{yx} G_{m_{i^f_{2k+\gamma}},n_{j^f_{2l+\delta}}}(x+y, y)\right)_{\gamma, \delta = 1}^\rho\nonumber \\
        & \hspace{5em}\cdot \nord a_1^{(m_1-1)}(y+x)\cdots a_r^{(m_r-1)}(y+x)\cdot b_1^{(n_1-1)}(y)\cdots b_s^{(n_s-1)}(y)\nord_{O(I^b, I^f, J^b, J^f)} 
    \end{align*}
    as a series in $V[[x, x^{-1}, y, y^{-1}]]$. 
\end{cor}

\begin{proof}
    Note that the left-hand side is precisely 
    $$\bar{Y}_W(\exp(\Delta(y)):a_1^{(m_1-1)}(x)_V\cdots a_r^{(m_r-1)}(x)_V: b_1(-n_1+|b_1|/2)\cdots b_s(-n_s+|b_s|/2)\one, y).$$
    The conclusion follows from Theorem \ref{exp-delta-A-B-lemma} and an argument similar to that of the proof of Corollary 4.7 in \cite{FQ-Fermion-1}. 
\end{proof}

\begin{cor}\label{product-actual-vo}
    For $r, s \in \N, a_1, \dots, a_r, b_1, \dots, b_s\in \h^\B\cup\h^\F, m_1, \dots, m_r, n_1, \dots, n_s\in \N$, 
    \begin{align*}
        & Y_W(a_1(-m_1+|a_1|/2)\cdots a_r(-m_r+|a_r|/2)\one, x)Y_W(b_1(-n_1+|b_1|/2)\cdots b_s(-n_s+|b_s|/2)\one, y) \nonumber\\
       &\quad =  \sum_{\eta = 0}^{\min(r, s)} \sum_{I^b \in A_\eta^\B, J^b\in B_\eta^\B} \perm \mat \left(\ell_\B \left(a_{i^b_\alpha}, b_{j^b_\beta}\right) F^\B_{m_{i^b_\alpha} n_{j^b_\beta}}(x,0)\right)_{\alpha, \beta=1}^\rho\nonumber\\
        & \qquad \cdot \sum_{k, l \geq 0} \sum_{\rho=0}^{\min(r-2k, s-2l)} (-1)^{r^\F\rho + \rho(\rho+1)/2} \nonumber\\
        & \quad \cdot \sum_{I^f \in A_{2k,\rho}^\F} (-1)^{\sigma(I^f)} T_{a_1, \dots, a_r}^{m_1,\dots,m_r}\left(i^f_1, \dots, i^f_{2k}\right)\iota_{xy}(x+y)^{-m_{i^f_1}-\cdots - m_{i^f_{2k}} + k} \nonumber \\
        & \qquad \cdot \sum_{J^f \in B_{2l,\rho}^\F} (-1)^{\sigma(J^f)} T_{b_1, \dots, b_s}^{n_1, \dots, n_s}\left(j^f_1, \dots, j^f_{2l}\right) y^{-n_{j^f_1}-\cdots - n_{j^f_{2l}} + l}\nonumber \\
        & \hspace{5em} \cdot \det \mat \left( \ell_\F \left( a_{i^f_{2k+\gamma}}, b_{j^f_{2l+\delta}}\right) \iota_{xy} G_{m_{i^f_{2k+\gamma}},n_{j^f_{2l+\delta}}}(x+y, y)\right)_{\gamma, \delta = 1}^\rho\nonumber \\
        & \hspace{5em}\cdot \nord a_1^{(m_1-1)}(x+y)\cdots a_r^{(m_r-1)}(x+y) \cdot b_1^{(n_1-1)}(y)\cdots b_s^{(n_s-1)}(y)\nord_{O(I^b, I^f, J^b, J^f)} .
    \end{align*}
\end{cor}

\begin{proof}
The conclusion follows from the definition of $Y_W$, Lemma \ref{exp-delta-on-basis}, and Corollary \ref{Mod-iterate}. 
\end{proof}

\begin{prop}
      For any $a_1, \dots, a_r, b_1, \dots, b_s\in \h^\B\cup \h^\F, m_1, \dots, m_r, n_1, \dots, n_s\in \Z_+$, let 
      $$u = a_1(-m_1+|a_1|/2)\cdots a_r(-m_r+|a_r|/2)\one,\ v = b_1(-n_1+|b_1|/2) \cdots b_s(-n_s+|b_s|/2)\one$$
      then for every homogeneous $w\in W$, the identity
    $$(x+y)^{P+|u|/2} y^{|v|/2}Y_W(u, x)Y_W(v, y)w = (x+y)^{P+|u|/2} y^{|v|/2} Y_W(Y(u, x)v, y)w$$
    holds, where
    $$P = \wt(w) + m_1 + \cdots + m_r$$
    that is independent of $v$. 
\end{prop}

\begin{proof}
    Clearly, the iterate formula in Corollary \ref{iterate-actual-vo} and the product formula in Corollary \ref{product-actual-vo} differ by an expansion of zero. Moreover, if we act the product and the iterate on a homogeneous element $w \in W$, then the expansion of zero can be removed by multiplying $(x+y)^{P+|u|/2}y^{|v|/2}$, where $P$ may be chosen as any integer that is greater than or equal to
    $$m_1 + \cdots + m_r + \wt w. $$
    This lower bound is independent of $v$. Thus the weak associativity with pole-order condition holds. 
\end{proof}
This finishes our proof of Theorem \ref{it-is-a-module}.

\noindent {\small \sc Department of Mathematics and Computer Science,
Rutgers University,
101 Warren St., Newark, NJ 07102, USA}

\noindent {\em E-mail address}: \texttt{qixuan.fang@rutgers.edu | qixuan.fang.math@gmail.com}

\noindent {\small \sc Department of Mathematics, Rutgers University,
110 Frelinghuysen Rd., Piscataway, NJ 08854, USA}

\noindent {\em E-mail address}: \texttt{yzhuang@math.rutgers.edu}

\noindent {\small \sc School of Mathematics (Zhuhai), Sun Yat-Sen University, Zhuhai, 519082, Guangdong, China}

\noindent {\em E-mail address}: \texttt{qifei@mail.sysu.edu.cn | fei.qi.math.phys@gmail.com}

\begin{thebibliography}{KWak2}

\bibitem[B]{Barron-Review}  Katrina Barron, On twisted modules for N=2 supersymmetric vertex operator superalgebras, \textit{Proceedings of the IXth International Workshop on Lie Theory and Its Applications in Physics}, June 2011, Varna, Bulgaria; ed. V. Dobrev, Springer 2013, 411--420. Longer preprint: Twisted modules for N=2 supersymmetric vertex operator superalgebras arising from finite automorphisms of the N=2 Neveu-Schwarz algebra, arxiv:1110.0229. 



\bibitem[FFR]{FFR}
Alex J. Feingold, Igor B. Frenkel and John F. X. Ries, {\it Spinor construction of 
vertex operator algebras, triality, and $ E_8^{(1)}$}, Contemporary Math., Vol. 121,
Amer. Math. Soc., Providence, 1991.
  

\bibitem[FQ]{FQ-Fermion-1}
Francesco Fiordalisi and Fei Qi,  Fermionic construction of the $\frac{\Z}{2}$-graded meromorphic 
open-string vertex algebra and its $\Z_{2}$-twisted module, I, 
{\it Lett. Math.Phys.} {\bf 114} (2024), article no. 59. 

\bibitem[FLM]{FLM}
Igor B. Frenkel, James Lepowsky and Arne Meurman,
{\em Vertex Operator Algebras and the Monster},
Pure and Appl. Math., Vol. 134,  Academic Press,  Boston, 1988.


\bibitem[H1]{H-mosva}
Yi-Zhi Huang, Meromorphic open-string vertex algebras, {\it J. Math. Phys.}
{\bf  54} (2013), 051702. 

\bibitem[H2]{H-mosva-rm}
Yi-Zhi Huang, Meromorphic open-string vertex algebras and Riemannian manifolds,
{\it Chin. Ann. Math.}, to appear. 


\bibitem[Q1]{Q-Thesis}
Fei Qi, Representation theory and cohomology theory of meromorphic open-string vertex algebras, PhD
Thesis, Rutgers University, 2018. 

\bibitem[Q2]{Q-Mod}
Fei Qi,  On modules for meromorphic open-string vertex algebras, 
{\it J. Math. Phys.} {\bf 60} (2019), paper no. 031701.

\bibitem[Q3]{Q-Fermion-2}
Fei Qi,  Fermionic construction of the $\frac{\Z}{2}$-graded meromorphic 
open-string vertex algebra and its $\Z_{2}$-twisted module, II, 
{\it Lett. Math.Phys.} {\bf 114} (2024), article no. 67. 

\bibitem[W]{W}
G. C. Wick, The evaluation of the collision matrix, {\it Phys. Rev.} {\bf 80} (1950)., 268--272.


\end{thebibliography}
\end{document}